\documentclass[11pt]{article}

\usepackage{amsmath,amsfonts,amsthm,amscd,amssymb,graphicx}
\numberwithin{equation}{section}

\usepackage{color}

\newtheorem{Lemma}{Lemma}[section]
\newtheorem{Corollary}[Lemma]{Corollary}

\newtheorem{Hypothesis}[Lemma]{Hypothesis}
\newtheorem{Proposition}[Lemma]{Proposition}

\newtheorem{Theorem}[Lemma]{Theorem}
\newtheorem{lemma}{Lemma}[section]

\newtheorem{proposition}[Lemma]{Proposition}
\newtheorem{remark}[Lemma]{Remark}

\def\Re{\mathop\mathrm{Re}\nolimits}    
\def\Im{\mathop\mathrm{Im}\nolimits}    

\newcommand{\R}{\mathbb{R}}             

\newcommand{\rmd}{\mathrm{d}}           
\newcommand{\rme}{\mathrm{e}}           
\newcommand{\rmi}{\mathrm{i}}           
 
\newcommand{\rmg}{\mathrm{g}}

\def\cO{\mathcal{O}}

\def\e{\epsilon}
\def\errfn{\text{\rm errfn\ }}

\begin{document}

\title{\vspace{-1in} Nonlinear stability of source defects in the complex Ginzburg-Landau equation}

\author{Margaret Beck\footnotemark[1] \and Toan T. Nguyen\footnotemark[2] \and  Bj\"orn Sandstede\footnotemark[3] \and Kevin Zumbrun\footnotemark[4]}

\date{\today}

\maketitle

\begin{abstract} 

In an appropriate moving coordinate frame, source defects are time-periodic solutions to reaction-diffusion equations that are spatially asymptotic to spatially periodic wave trains whose group velocities point away from the core of the defect. In this paper, we rigorously establish nonlinear stability of spectrally stable source defects in the complex Ginzburg-Landau equation. Due to the outward transport at the far field, localized perturbations may lead to a highly non-localized response even on the linear level. To overcome this, we first investigate in detail the dynamics of the solution to the linearized equation. This allows us to determine an approximate solution that satisfies the full equation up to and including quadratic terms in the nonlinearity. This approximation utilizes the fact that the non-localized phase response, resulting from the embedded zero eigenvalues, can be captured, to leading order, by the nonlinear Burgers equation. The analysis is completed by obtaining detailed estimates for the resolvent kernel and pointwise estimates for the Green's function, which allow one to close a nonlinear iteration scheme. 

\end{abstract}

\renewcommand{\thefootnote}{\fnsymbol{footnote}}

\footnotetext[1]{Department of Mathematics, Boston University,
Boston, MA~02215, USA, and Heriot-Watt University, Edinburgh, EH14 4AS, UK. Email: mabeck@math.bu.edu. Research supported in part by NSF grant DMS-1007450 and a Sloan Fellowship.
} 

\footnotetext[2]{Department of Mathematics, Pennsylvania State University, State College, PA~16803, USA. Email: nguyen@math.psu.edu. Research supported in part by NSF grant DMS-1338643.}

\footnotetext[3]{Division of Applied Mathematics, Brown University, Providence, RI 02912, USA. 
Email: Bjorn\underline{~}Sandstede@Brown.edu. Research supported in part by NSF grant DMS-0907904.
}

\footnotetext[4]{Department of Mathematics,
Indiana University,
Bloomington, IN~47405, USA. Email: kzumbrun@indiana.edu. Research supported in part by NSF grant DMS-0300487.
}

\tableofcontents

\section{Introduction}\label{sec-Intro}

In this paper we study stability of source defect solutions of the complex cubic-quintic Ginzburg-Landau (qCGL) equation
\begin{equation}\label{eqs-qCGLintro}
A_t = (1+i\alpha)A_{xx} + \mu A - (1+i\beta)A|A|^2 +  (\gamma_1 + i \gamma_2 )A |A|^4.  
\end{equation}
Here $A = A(x,t)$ is a complex-valued function, $x\in \mathbb{R}$, $t\ge0$, and $\alpha, \beta, \mu, \gamma_1,$ and $\gamma_2$ are all real constants with $\gamma = \gamma_1 + i \gamma_2$ being small but nonzero. Without loss of generality we assume that $\mu = 1$, which can be achieved by rescaling the above equation. It is shown, for instance in \cite{BekkiNozaki85,PoppStillerAranson95, Doelman96, KapitulaRubin00, Lega01,SandstedeScheel04}, that the qCGL equation exhibits a family of defect solutions known as sources (see equation (\ref{E:source})). We are interested here in establishing nonlinear stability of these solutions, under suitable spectral stability assumptions.

In general, a defect is a solution $u_\mathrm{d}(x,t)$ of a reaction-diffusion equation
\[
u_t = Du_{xx} + f(u), \qquad u: \mathbb{R} \times \mathbb{R}^+ \to \mathbb{R}^n
\]
that is time-periodic in an appropriate moving frame $\xi=x-c_\mathrm{d}t$, where $c_\mathrm{d}$ is the speed of the defect, and spatially asymptotic to wave trains, which have the form $u_\mathrm{wt}(kx-\omega t;k)$ for some profile $u_\mathrm{wt}(\theta;k)$ that is $2\pi$-periodic in $\theta$. Thus, $k$ and $\omega$ represent the spatial wave number and the temporal frequency, respectively, of the wave train. Wave trains typically exist as one-parameter families, where the frequency $\omega=\omega_\mathrm{nl}(k)$ is a function of the wave number $k$. The function $\omega_\mathrm{nl}(k)$ is referred to as the nonlinear dispersion relation, and its domain is typically an open interval. The group velocity $c_\mathrm{g}(k_0)$ of the wave train with wave number $k_0$ is defined as
\[
c_\mathrm{g}(k_0) := \frac{\rmd \omega_\mathrm{nl}}{\rmd k}(k_0).
\]
The group velocity is important as it is the speed with which small localized perturbations of the wave train propagate as functions of time, and we refer to \cite{DoelmanSandstedeScheel09} for a rigorous justification of this. 

Defects have been observed in a wide variety of experiments and reaction-diffusion models and can be classified into several distinct types that have different existence and stability properties \cite{SaarloosHohenberg,Hecke,SandstedeScheel04}. This classification involves the group velocities $c_g^\pm := c_g(k_\pm)$ of the asymptotic wave trains, whose wavenumbers are denoted by $k_\pm$. Sources are defects for which $c_\mathrm{g}^-<c_\mathrm{d}<c_\mathrm{g}^+$, so that perturbations are transported away from the defect core towards infinity. Generically, sources exist for discrete values of the asymptotic wave numbers $k_\pm$, and in this sense they actively select the wave numbers of their asymptotic wave trains. Thus, sources can be thought of as organizing the dynamics in the entire spatial domain; their dynamics are inherently not localized. 

For equation (\ref{eqs-qCGLintro}), the properties of the sources can be determined in some detail. We will focus on standing sources, for which $c_\mathrm{d} = 0$. They have the form
\begin{equation}\label{E:source}
A_\mathrm{source}(x, t) = r(x) \rme^{\rmi \varphi(x) } \rme^{-\rmi \omega_0 t},
\end{equation}
where 
\[
\lim_{x\to\pm\infty} \varphi_x(x)= \pm k_0, \qquad \lim_{x\to\pm\infty} r(x) = \pm r_0(k_0), \qquad \omega_0 = \omega_0(k_0),
\]
where the details of the functions $r, \varphi, r_0$ and $\omega_0$ are described in Lemma~\ref{lem-NBexistence}, below. In order for such solutions to be nonlinearly stable, they must first be spectrally stable, meaning roughly that the linearization about the source must not contain any spectrum in the positive right half plane -- see Hypothesis~\ref{H:spec-stab}, below. Our goal is to prove that, under this hypothesis, the sources are nonlinearly stable.

To determine spectral stability one must locate both the point and the essential spectrum. The essential spectrum is determined by the asymptotic wave trains. As we will see below in \S~\ref{sec-linearization}, there are two parabolic curves of essential spectrum. One is strictly in the left half plane and the other is given by the linear dispersion relation
\[
\lambda_\mathrm{lin}(\kappa) = - i c_\mathrm{g} \kappa  -\mathrm{d} \kappa^2+ \mathcal{O}(\kappa^3)
\]
for small $\kappa\in \mathbb{R}$, where $c_\rmg = 2k_0 (\alpha - \beta_*)$ denotes the group velocity and
\begin{equation}\label{def-betad}
\mathrm{d} := (1+\alpha \beta_*) - \frac{2k_0^2 (1+\beta^2_*)}{r_0^2(1-2\gamma_1 r_0^2)} ,\qquad \beta_* := \frac{\beta - 2 \gamma_2 r_0^2}{1-2\gamma_1 r^2_0}.
\end{equation}
Thus, this second curve touches the imaginary axis at the origin and, if $\mathrm{d}> 0$, then it otherwise lies in the left half plane. In this case, the asymptotic plane waves, and therefore also the essential spectrum, are stable, at least with respect to small wave numbers $k_0$. Otherwise, they are unstable. Throughout the paper, we assume that $\mathrm{d}>0$. 

Determining the location of the point spectrum is more difficult. For all parameter values there are two zero eigenvalues, associated with the eigenfunctions $\partial_x A_\mathrm{source}$ and $\partial_t A_\mathrm{source}$, which correspond to space and time translations, respectively.
When $\gamma_1 = \gamma_2 = 0$, one obtains the cubic Ginzburg-Landau equation (cCGL). In this case, the sources are referred to as Nozaki-Bekki holes, and they are a degenerate family, meaning that they exist for values of the asymptotic wave number in an open interval (if one chooses the wavespeed appropriately), rather than for discrete values of $k_0$. Therefore, in this case there is a third zero eigenvalue associated with this degeneracy. Moreover, in the limit where $\alpha = \beta = \gamma_1 = \gamma_2 = 0$, which is the real Ginzburg Landau (rGL) equation, the sources are unstable. This can be shown roughly using a Sturm-Liouville type argument: in this case, the amplitude is $r(x) = \mathrm{tanh}(x)$ and so $r'(x)$, which corresponds to a zero eigenvalue, has a single zero, which implies the existence of a positive eigenvalue. 

The addition of the quintic term breaks the underlying symmetry to remove the degeneracy \cite{Doelman96} and therefore also one of the zero eigenvalues. To find a spectrally stable source, one needs to find parameter values for which both the unstable eigenvalue (from the rGL limit) and the perturbed zero eigenvalue (from the cCGL limit) become stable. This has been investigated in a variety of previous studies, including \cite{Lega01, PoppStillerAranson95, ChateManneville92, KapitulaRubin00, SandstedeScheel05_absolute, LegaFauve97}. Partial analytical results can be found in \cite{KapitulaRubin00, SandstedeScheel05_absolute}. Numerical and asymptotic evidence in \cite{ChateManneville92, PoppStillerAranson95} suggests that the sources are stable in an open region of parameter space near the NLS limit of (\ref{eqs-qCGLintro}), which corresponds to the limit $|\alpha|, |\beta| \to \infty$ and $\gamma_1, \gamma_2 \to 0$. In the present work, we will assume the parameter values have been chosen so that the sources are spectrally stable. 

The main issue regarding nonlinear stability will be to deal with the effects of the embedded zero eigenvalues. This has been successfully analyzed in a variety of other contexts, most notably viscous conservation laws \cite{ZH, HowardZumbrun06, BeckSandstedeZumbrun10}. Typically, the effect of these neutral modes is studied using an appropriate Ansatz for the form of the solution that involves an initially arbitrary function. That function can subsequently be chosen to cancel any non-decaying components of the resulting perturbation, allowing one to close a nonlinear stability argument. The key difference here is that the effect of these eigenvalues is to cause a nonlocalized response, even if the initial perturbation is exponentially localized. This makes determining the appropriate Ansatz considerably more difficult, as it effectively needs to be based not just on the linearized operator but also on the leading order nonlinear terms. 

The remaining generic defect types are sinks (both group velocities point towards the core), transmission defects (one group velocity points towards the core, the other one away from the core), and contact defects (both group velocities coincide with the defect speed). Spectral stability implies nonlinear stability of sinks \cite[Theorem 6.1]{SandstedeScheel04} and transmission defects \cite{GallaySchneiderUecker2004} in appropriately weighted spaces; the proofs rely heavily on the direction of transport and do not generalize to the case of sources. We are not aware of nonlinear stability results for contact defects, though their spectral stability was investigated in \cite{SandstedeScheel04a}.

We will now state our main result in more detail, in \S~\ref{S:main_results}. Subsequently, 
we will explain in \S~\ref{S:diff_frame} the importance of the result and its relationship to the existing literature. The proof will be contained in sections \S\ref{sec-Prelim}-\S\ref{sec-stability}.

\subsection{Main result: nonlinear stability}\label{S:main_results}
Let $A_\mathrm{source}(x, t)$ be a source solution of the form \eqref{E:source} and let $A(x,t)$ be the solution of \eqref{eqs-qCGLintro} with smooth initial data $A_\mathrm{in}(x)$. In accordance with \eqref{E:source}, we assume that the initial data $A_\mathrm{in}(x)$ is of the form $R_\mathrm{in}(x) e^{i\phi_\mathrm{in}(x)}$ and close to the source solution in the sense that the norm
\begin{equation}\label{in-norm}
\| A_\mathrm{in}(\cdot) - A_\mathrm{source}(\cdot,0)\|_{\mathrm{in}} : = \| e^{x^2/M_0} (R_\mathrm{in} - r) (\cdot)\|_{C^3(\mathbb{R})} +  \| e^{x^2/M_0} (\phi_\mathrm{in} - \varphi)(\cdot)\|_{C^3(\mathbb{R})},
\end{equation} 
where $M_0$ is a fixed positive constant and $\|\cdot \|_{C^3}$ is the usual $C^3$-sup norm, is sufficiently small. The solution $A(x,t)$ will be constructed in the form
\[
A(x + p(x,t),t) = (r(x) + R(x,t))e^{\rmi(\varphi(x) + \phi(x,t))}e^{-\rmi \omega_0 t},
\]
where the function $p(x,t)$ will be chosen so as to remove the non-decaying terms from the perturbation. The initial values of $p(x,0), R(x,0), \phi(x,0)$ can be calculated in terms of the initial data $A_\mathrm{in}(x)$.  

Below we will compute the linearization of \eqref{eqs-qCGLintro} about the source \eqref{E:source} and use this information to choose $p(x,t)$ in a useful way. Furthermore, the linearization and the leading order nonlinear terms will imply that $\phi(x,t) \to \phi^a(x,t)$ as $t \to \infty$, where $\phi^a$ represents the phase modulation caused by the zero eigenvalues. The notation is intended to indicate that $\phi^a$ is an approximate solution to the equation that governs the dynamics of the perturbation $\phi$. In particular, $\phi^a$ is a solution to an appropriate Burgers-type equation that captures the leading order dynamics of $\phi$. (See equation \eqref{intro-Burgers}.) The below analysis will imply that the leading order dynamics of the perturbed source are given by the modulated source
\[
A_\mathrm{mod}(x + p(x,t),t) :=  A_\mathrm{source}(x,t) \rme^{\rmi \phi^a(x,t)}  = r(x) \rme^{\rmi (\varphi(x) + \phi^a(x,t))} \rme^{-\rmi \omega_0 t}. 
\]

The functions $p(x,t)$ and $\phi^a(x,t)$ together will remove from the dynamics any non-decaying or slowly-decaying terms, resulting from the zero eigenvalues and the quadratic terms in the nonlinearity,  thus allowing a nonlinear iteration scheme to be closed. To describe these functions in more detail, we define 
\begin{equation}\label{def-epm}
e(x,t) := \errfn\left(\frac{x+c_\mathrm{g} t}{\sqrt{4\mathrm{d}t}}\right) -  \errfn\left(\frac{x-c_\mathrm{g} t}{\sqrt{4\mathrm{d}t}}\right) ,\qquad \errfn(z):=
{1\over 2\pi } \int_{-\infty}^{z} \rme^{-x^2}dx
\end{equation}
and the Gaussian-like term
\begin{equation}\label{Gaussian-like}
\theta(x,t) := \frac{1}{(1+t)^{1/2}} \left( \rme^{-\frac{(x-c_\rmg t)^2}{M_0(t+1)}} + \rme^{-\frac{(x+c_\rmg t)^2}{M_0(t+1)}}\right),
\end{equation}
where $M_0$ is a fixed positive constant. Now define
\begin{equation}\label{def-trans-phase}
\begin{aligned} 
\phi^a(x,t)  &:= -\frac{\mathrm{d}}{2q} \Big[ \log \Big(1 + \delta^+ (t)e(x,t+1)\Big) + \log \Big(1 + \delta^- (t)e(x,t+1)\Big)\Big],\\
 p(x,t)  &:= \frac {\mathrm{d}}{2qk_0} \Big[ \log \Big(1 + \delta^+ (t)e(x,t+1)\Big) - \log \Big(1 + \delta^- (t)e(x,t+1)\Big)\Big],
 \end{aligned}\end{equation}
where the constant $q$ is defined in \eqref{E:defq} and $\delta^\pm = \delta^\pm (t)$ are smooth functions that will be specified later. Our main result asserts that the shifted solution $A(x+p(x,t),t)$ converges to the modulated source 
with the decay rate of a Gaussian. 

\begin{Theorem}\label{theo-main}
Assume that the initial data is of the form $A_\mathrm{in}(x) =  R_\mathrm{in}(x)e^{i\phi_\mathrm{in}(x)}$ with $R_\mathrm{in},\phi_\mathrm{in}\in C^3(\mathbb{R})$. There exists a positive constant $\epsilon_0$ such that, if 
\begin{equation} \label{E:ass-ic}
\epsilon := \| A_\mathrm{in}(\cdot) - A_\mathrm{source}(\cdot,0)\|_{\mathrm{in}}  \le \epsilon_0,
\end{equation} then the solution $A(x,t)$ to the qCGL equation \eqref{eqs-qCGLintro} exists globally in time. In addition, there are constants $\eta_0,C_0,M_0>0$, $\delta_\infty^\pm \in \mathbb{R}$ with  $|\delta^\pm_\infty|\le \epsilon C_0$, and smooth functions $\delta^\pm(t)$ so that 
\[
|\delta^\pm(t)-\delta^\pm_\infty| \leq \epsilon C_0 \rme^{-\eta_0 t}, \qquad \forall t\ge 0
\] 
and
\begin{equation}\label{conv-Anb}
\Big| \frac{\partial^\ell}{\partial x^\ell}  \Big[ A(x + p(x,t),t) - A_\mathrm{mod}(x,t)\Big] \Big | \leq \epsilon C_0 (1+t)^{\kappa} [(1+t)^{-\ell/2} + e^{-\eta_0|x|}]\theta(x,t) , \qquad \forall x\in \mathbb{R}, \quad \forall t\ge 0, 
\end{equation}
for $\ell = 0,1,2$ and for each fixed $\kappa \in(0,\frac12)$. In particular, $\| A(\cdot + p(\cdot,t),t) - A_\mathrm{mod}(\cdot,t)\|_{W^{2,r}}\to0$ as $t\to\infty$ for each fixed $r>\frac{1}{1-2\kappa}$.
\end{Theorem}

Not only does Theorem~\ref{theo-main} rigorously establish the nonlinear stability of the source solutions of \eqref{eqs-qCGLintro}, but it also provides a rather detailed description of the dynamics of small perturbations. The amplitude of the shifted solution $A(x+p(x,t),t)$ converges to the amplitude of the source $A_\mathrm{souce}(x,t)$ with the decay rate of a Gaussian: $R(x,t) \sim \theta(x,t)$. In addition, the phase dynamics can be understood as follows. If we define
\begin{equation}\label{def-deltap}
\delta_\phi(t)  := -\frac{\mathrm{d} }{2q} \log \Big[ (1 + \delta^+ (t))(1+\delta^-(t)) \Big ] , \qquad \delta_p(t) := \frac{\mathrm{d}}{2qk_0} \log \Big[{ 1 + \delta^+ (t) \over 1 + \delta^- (t)}\Big],
\end{equation}
it then follows from \eqref{def-trans-phase} that 
\begin{equation}\label{shift-est}
\Big |\phi^a(x,t) - \delta_\phi(t) e(x,t+1)\Big | + \Big |p(x,t) - \delta_p(t) e(x,t+1)\Big | \le  \epsilon C_0(1+t)^{1/2}  \theta(x,t).
\end{equation}
The function $e(x,t)$ resembles an expanding plateau of height approximately equal to one that spreads outwards with speed $\pm c_\rmg$, while the associated interfaces widen like $\sqrt{t}$; see Figure~\ref{fig-plateau}. Hence, the phase $\varphi(x) + \phi(x,t)$ tends to $\varphi(x) + \phi^a(x,t)$, where $\phi^a(x,t)$ looks like an expanding plateau as time increases. 

\begin{figure}[h]
\centering\includegraphics[scale=.5]{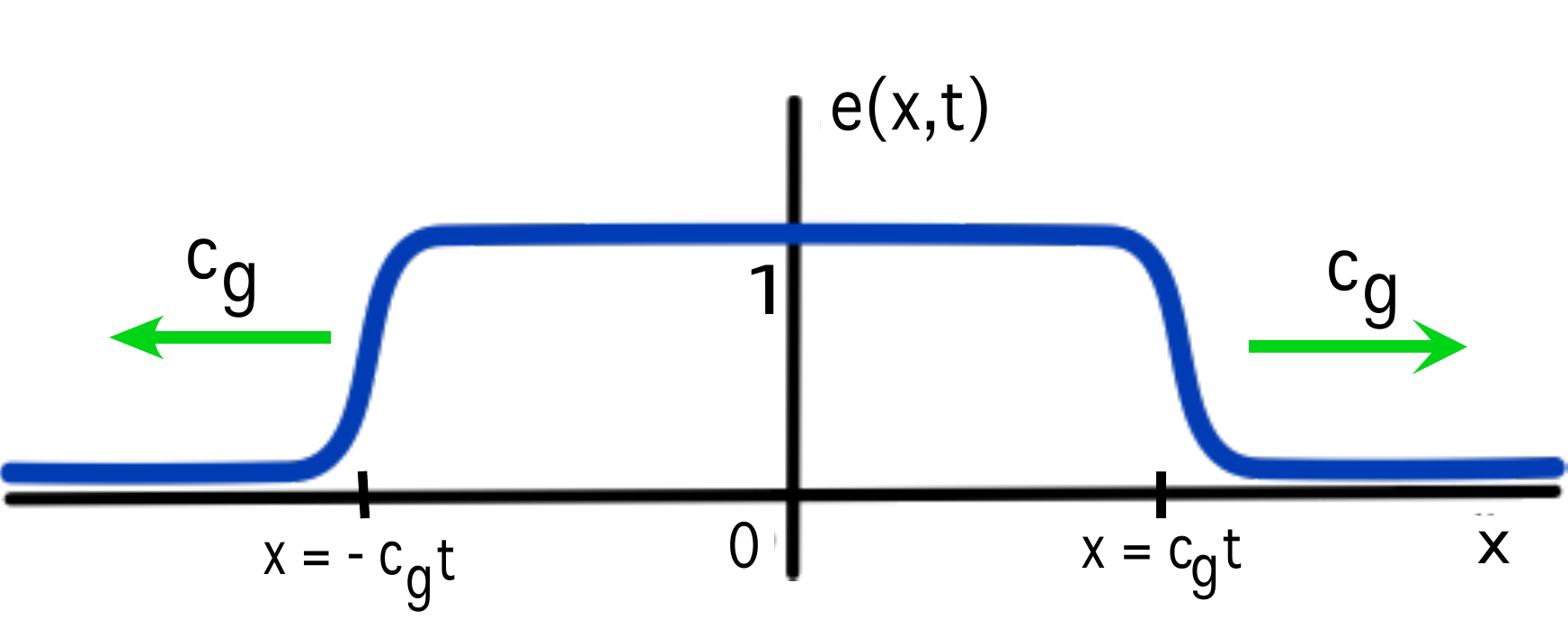} 
\caption{\em Illustration of the graph of $e(x,t)$, the difference of two error functions, for a fixed value of $t$.}
\label{fig-plateau}
\end{figure}

As a direct consequence of Theorem \ref{theo-main}, we obtain the following corollary. 

\begin{Corollary} Let $\eta$ be an arbitrary positive constant and let $V$ be the space-time cone defined by the constraint: $ - (c_\mathrm{g} - \eta )t  \le x \le (c_\mathrm{g}  - \eta) t$. Under the same assumptions as in Theorem \ref{theo-main}, there are positive constants $\eta_1, C_1$ so that the solution $A(x,t)$ to the qCGL equation \eqref{eqs-qCGLintro} satisfies  
$$ |A(x,t) - A_{\mathrm{source}} (x - \delta_p(\infty), t -\delta_\phi(\infty) / \omega_0) | \le \epsilon C_1 e^{-\eta_1 t}$$
for all $(x,t) \in V$, in which $\delta_p$ and $\delta_\phi$ are defined in \eqref{def-deltap}. 
\end{Corollary}
\begin{proof} Indeed, within the cone $V$, we have 
$$|e(x,t+1) - 1| + \theta(x,t) \le C_1 e^{-\eta_1 t}$$
for some constants $\eta_1, C_1>0$. The estimate \eqref{shift-est} shows that $p(x,t)$ and $\phi^a(x,t)$ are constants up to an error of order $e^{-\eta_1 t}$. 
The main theorem thus yields the corollary at once.
\end{proof}

As will be seen in the proof of Theorem \ref{theo-main}, the functions $\delta^\pm(t)$ will be constructed via integral formulas that are introduced to precisely capture the non-decaying part of the Green's function of the linearized operator.  The choices of $p(x,t)$ and $\phi^a(x,t)$ are made based on the fact that the asymptotic dynamics of the translation and phase variables is governed (to leading order) by a nonlinear Burgers-type equation:
\begin{equation}\label{intro-Burgers}\Big( \partial_t +\frac{c_\mathrm{g}}{k_0} \varphi_x \partial_x- \mathrm{d} \partial_x^2 \Big) (\phi^a \pm k_0 p) = q (\partial_x \phi^a \pm k_0 \partial_x p)^2,\end{equation}
where $q$ is defined in \eqref{E:defq}. See Section~\ref{sec-Asymptotic}. The formulas \eqref{def-trans-phase} are related to an application of the Cole-Hopf transformation to the above equation.


\subsection{Difficulties and a framework} \label{S:diff_frame}

In the proof, we will have to overcome two difficulties. The first, the presence of the embedded zero eigenvalues, can be dealt with using the now standard, but nontrivial, technique first developed in \cite{ZH}. Roughly speaking, this technique involves the introduction of an initially arbitrary function into the perturbation Ansatz, which is later chosen to cancel with the nondecaying parts of the Green's function that result from the zero eigenvalues. The second difficulty is dealing with the quadratic order nonlinearity. 

To illustrate this second difficulty, for the moment ignore the issue of the zero eigenvalues.
Suppose we were to linearize equation \eqref{eqs-qCGLintro} in the standard way and set
\[
A(x,t) = A_\mathrm{source}(x,t) + \tilde A(x,t),
\]
with the hope of proving that the perturbation, $\tilde A(x,t)$, decays. The function $\tilde A(x,t)$ would then satisfy an equation of the form 
\[
(\partial_t  - \mathcal{L})\tilde A = Q(\tilde A),
\]
where $\mathcal{L}$ denotes the linearized operator, with the highest order derivatives being given by $(1+i\alpha)\partial_x^2$, and $Q(\tilde A) = \cO(|\tilde A|^2)$ denotes the nonlinearity, which contains quadratic terms. Since the temporal Green's function (also known as the fundamental solution) for the heat operator is the Gaussian $t^{-1/2} e^{-|x-y|^2/4t}$ centered at $x=y$, the Green's function of $\partial_t - \mathcal{L}$ at best behaves like a Gaussian centered at $x = y \pm c_\rmg t$. (In fact it is much worse, once we take into account the effects of the embedded zero eigenvalues.) Quadratic terms can have a nontrivial and subtle effect on the dynamics of such an equation: consider, for example, $u_t = u_{xx} - u^2$. The zero solution is stable with respect to positive initial data, but is in general unstable. For such situations, standard techniques for studying stability are often not effective. In particular, the nonlinear iteration procedure that is typically used in conjunction with pointwise Green's function estimates does not work when quadratic terms are present (unless they have a special conservative structure). This is because the convolution of a Gaussian (the Green's function) against a quadratic function of another Gaussian, $Q(\tilde A)$, would not necessarily yield Gaussian behavior. Therefore, if we were to use this standard Ansatz, it would not be possible to perform the standard nonlinear iteration scheme and show that $\tilde A$ also decays like a Gaussian. To overcome this, we must use an Ansatz that removes the quadratic terms from the equation. 

Returning to the first difficulty, as mentioned above (see also Lemma~\ref{lem-Lspectrum}), the essential spectrum of $\mathcal{L}$ touches the imaginary axis at the origin and $\mathcal{L}$ has a zero eigenvalue of multiplicity two. The associated eigenfunctions are $\partial_t A_\mathrm{source}$ and $\partial_x A_\mathrm{source}$, which correspond to time and space translations, respectively. Neither of these eigenfunctions are localized in space (nor are they localized with respect to the $(R, \phi)$ coordinates - see \eqref{def-V12}). This is due to the fact that the group velocities are pointing outward, away from the core of the defect, and so (localized) perturbations will create a non-local response of the phase. More precisely, the perturbed phase $\phi(x,t)$ will resemble an outwardly expanding plateau. This behavior will need to be incorporated in the analysis if we are to close a nonlinear iteration scheme.

In the proof, we write the solution $A(x,t)$ in the form 
\[
A(x+ p(x,t),t) =  [r(x) + R(x,t)] \rme^{\rmi (\varphi(x) + \phi(x,t))} \rme^{-\rmi \omega_0 t} ,
\]
and work with perturbation variables $(R(x,t),\phi(x,t))$. The advantages when working with these polar coordinates are i) they are consistent with the phase invariance (or gauge invariance) associated with \eqref{eqs-qCGLintro}; ii) the quadratic nonlinearity is a function of $R, \phi_x,$ and their higher derivatives, without any zero order term involving $\phi$; iii) based upon the leading order terms in the equation (see \S \ref{sec-asymptotic}), we expect that the time-decay in the amplitude $R$ is faster than that of the phase $\phi$. Roughly speaking, these coordinates effectively replace the equation $u_t = u_{xx} - u^2$, which is essentially what we would have for $\tilde A$, with  an equation like $u_t = u_{xx} - uu_x$, which is essentially what we obtain in the $(R, \phi)$ variables (but without the conservation law structure). In other words, with respect to $\tilde A$, the nonlinearity is relevant, but with respect to $(R, \phi)$, it is marginal \cite{BricmontKupiainen94}. 

In the case of a marginal nonlinearity, if there is an additional conservation law structure, as in for example Burgers equation ($uu_x = (u^2)_x/2$), then one can often exploit this structure to close the nonlinear stability argument. Here, however, that structure is absent, and so we must find another way to deal with the marginal terms. The calculations of \S\ref{sec-Asymptotic} show that, to leading order, the dynamics of $(R, \phi)$ are essentially governed by 
\begin{eqnarray*}
\partial_t \begin{pmatrix} R \\ r\phi \end{pmatrix} &\approx& \begin{pmatrix} 1-\alpha \beta^* & \alpha r \\ -\alpha(1+\beta^*) & r(1+\alpha\beta^*) \end{pmatrix} \partial_x^2 \begin{pmatrix} R \\ \phi \end{pmatrix} + \begin{pmatrix} -2(\alpha + \beta^*) \varphi_x & 2r\varphi_x \\ -2\varphi_x(1+(\beta^*)^2) & 2r\varphi_x(\beta^*-\alpha) \end{pmatrix} \partial_x \begin{pmatrix} R \\ \phi \end{pmatrix} \\
&& \qquad +\begin{pmatrix} -2r^2(1-2\gamma_1r^2) & 0 \\ 0 & 0 \end{pmatrix} \begin{pmatrix} R \\ \phi \end{pmatrix} + \begin{pmatrix}\mathcal{O}(R^2, \phi_x^2, R\phi_x) \\ q \phi_x^2 \end{pmatrix},
\end{eqnarray*}
where $q$ is defined in \eqref{E:defq}. The presence of the zero-order term $-2r^2(1-2\gamma_1r^2)R$ in the $R$ equation implies that it will decay faster than $\phi$. In fact, the above equation implies that to leading order $R \sim \phi_x$. Moreover, if we chose an approximate solution so that $R \sim k_0 \phi_x/(r_0(1-2\gamma_1r_0^2))$, then we see that $\phi$ satisfies exactly the Burgers equation given in \eqref{intro-Burgers} (up to terms that are exponentially localized). In order to close the nonlinear iteration, we will then need to incorporate these Burgers-type dynamics for $\phi$ into the Ansatz, which is done exactly through the approximate solution $\phi^a$. This is similar to the analysis of the toy model in \cite{BNSZ}.

When working with the polar coordinates, however, there is an apparent singularity when $r(x)$ vanishes. Such a point is inevitable since $r(x)\to \pm r_0$ with $r_0 \not = 0$ as $x\to \pm \infty$. We overcome this issue by writing the perturbation system as 
$$(\partial_t - \mathcal{L}) U = \mathcal{N}(R,\phi,p),$$ 
for $U = (R, r\phi)$, instead of $(R,\phi)$. Here $\mathcal{L}$ again denotes the linearized operator and $ \mathcal{N}(R,\phi,p)$ collects the remainder; see Lemmas~\ref{lem-linearization} and~\ref{lem-nonlinearUp} for details. Note that we do not write the remainder in terms of $U$, but leave it in terms of $R$ and $\phi$. Later on, once all necessary estimates for $U(x,t)$ and its derivatives are obtained, we recover the estimates for $(R(x,t),\phi(x,t))$ from those of $U(x,t)$, together with the observation that $\phi(x,t)$ should contain no singularity near the origin if $r(x)\phi(x,t)$ and its derivatives are regular; see Section~\ref{sec-esth2}. 

To make the above discussion rigorous, there will be four main steps. After stating some preliminary facts about sources and their linearized stability in \S\ref{sec-Prelim}, the first step in \S\ref{sec-resolvent} will be to construct the resolvent kernel by studying a system of ODEs that corresponds to the eigenvalue problem. In the second step, in \S\ref{sec-Green}, we derive pointwise estimates for the temporal Green's function associated with the linearized operator. These first two steps, although nontrivial, are by now routine following the seminal approach introduced by Zumbrun and  Howard \cite{ZH}. The third step, in \S\ref{sec-asymptotic},  is to construct the approximate Ansatz for the solution of qCGL, and the final step, in \S\ref{sec-stability}, is to introduce a nonlinear iteration scheme to prove stability. These last two steps are the novel and most technical ones in our analysis.  

To our knowledge, this is the first nonlinear stability result for a defect of source type, extending the theoretical framework to include this case. An interesting open problem at a practical level is to verify the spectral stability assumptions made here in some asymptotic regime; this is under current investigation. An important extension in the theoretical direction would be to treat the case of source defects of general reaction-diffusion equations not possessing a gauge invariance naturally identifying the phase. This would involve constructing a suitable approximate phase, sufficiently accurate to carry out a similar nonlinear analysis, a step that appears to involve substantial additional technical difficulty. We hope address this in future work.

\bigskip

{\bf Universal notation.} Throughout the paper, we write $g = \cO(f)$ to mean that there exists a universal constant $C$ so that $|g|\le C |f|$. 


\section{Preliminaries}\label{sec-Prelim}

\subsection{Existence of a family of sources for qCGL}\label{sec-NB}
In this subsection, we prove the following lemma concerning the existence and some qualitative properties of the source solutions defined in \eqref{E:source}. 

\begin{lemma}\label{lem-NBexistence} There exists a $k_0 \in \mathbb{R}$ with $|k_0| < 1$ 
such that a source solution $A_\mathrm{source}(x, t)$ of \eqref{eqs-qCGLintro} of the form \eqref{E:source} exists and satisfies the following properties.

1.  The functions $r(x)$ and $\varphi(x)$ are $C^\infty$. Let $x_0$ be a point at which $r(x_0)=0$. Necessarily, $r'(x_0)\not =0$ and $r_{xx}(x_0) = \varphi_x(x_0) = 0$. 

2. The functions $r$ and $\varphi$ satisfy $ r(x) \to  \pm r_0(k_0)$ and $\varphi_x(x) \to  \pm k_0$ as $x\to \pm \infty$, respectively, where $r_0$ is defined in (\ref{asy-wnl}), below. Furthermore, 
\[
\Big |\frac{d^\ell}{dx^\ell } \Big (r(x) \mp r_0 (k_0)\Big )\Big | + \Big |\frac{d^{\ell+1} }{dx^{\ell+1} } \Big ( \varphi(x) \mp k_0 x\Big )\Big | \le C_0 \rme^{-\eta_0 |x|},
\]
for integers $\ell  \ge 0$ and for some positive constants $C_0$ and $\eta_0$.

3. As $x\to \pm \infty$, $A_{\mathrm{source}}(x,t)$ converges to the wave trains $A_\mathrm{wt}(x,t;\pm k_0) = \pm r_0(k_0) e^{i(\pm k_0x - \omega_\mathrm{nl} (k_0) t)}$, respectively, with 
\begin{equation}\label{asy-wnl} r_0^2  =  1 - k_0^2  + \gamma_1 r_0 ^4  , \qquad  \omega_\mathrm{nl}(k_0)= \beta + (\alpha - \beta) k_0^2 + (\beta \gamma_1- \gamma_2) r_0^4. \end{equation}
Necessarily, $\omega_0 = \omega_\mathrm{nl}(k_0)$. 

4. If $k_0 \not =0$, the asymptotic group velocities $c_\rmg^\pm : = \frac{d  \omega_\mathrm{nl}(k)}{d k}\vert_{k = \pm k_0} $ have opposite sign at $\pm \infty$ and satisfy
\begin{equation}\label{def-cg} c_\rmg^\pm  =  \pm c_\rmg,\qquad \quad c_\rmg: = 2k_0 (\alpha - \beta_*), \end{equation}
where $\beta_*$ is defined \eqref{def-betad}. 
Without loss of generality, we assume that $c_\rmg >0$. 

\end{lemma}

Before proving this, let us recall that, for the cubic CGL equation ($\gamma_1 = \gamma_2 = 0$), an explicit formula for the (traveling) source, which is known in this case as a Nozaki-Bekki hole, is given in \cite{BekkiNozaki85, Lega01, PoppStillerAranson95}.  
These Nozaki-Bekki holes are degenerate solutions of CGL in the sense that they exist in a non-transverse intersection of stable and unstable manifolds. More precisely,  when $\gamma_1 = \gamma_2 =0$, the formula for the standing Nozaki-Bekki holes is given by
\[
A_\mathrm{source}(x, t) = r_0 \mbox{tanh}(\kappa x) e^{-\rmi \omega_0 t} e^{-\rmi \delta\log(2\cosh \kappa x)},
\]
where $r_0 = \sqrt{1-k_0^2}$, $k_0 = -\delta \kappa$, and  $\delta$ and $\kappa$ are defined by
\[
\kappa^2 = \frac{(\alpha-\beta)}{(\alpha-\beta)\delta^2 - 3\delta(1+\alpha^2)}, \qquad \delta^2 + \frac{3(1+\alpha\beta)}{(\beta-\alpha)} \delta -2 = 0,
\]
with $\delta$ chosen to be the root of the above equation such that $\delta(\alpha-\beta) < 0$. See \cite{Lega01} for details. The asymptotic phases and group velocities are
\[
\begin{aligned}
\omega_\pm (k_0) &= \beta   + (\alpha  - \beta)k_0^2, \qquad c_{\rmg}^\pm &=\frac{d  \omega_\pm(k)}{d k}\vert_{k = \pm k_0} = \pm 2k_0 (\alpha  - \beta) ,
\end{aligned}
\]
which are the identities \eqref{asy-wnl} and \eqref{def-cg} with $\gamma_1=\gamma_2=0$. We note that since $c_{\rmg}^+ + c_{\rmg}^- = 0$, the asymptotic group velocities must have opposite signs. To see that the solution is really a source, one can check that
\[
c_{\rmg}^+-c_{\rmg}^- = -4\delta\kappa(\alpha-\beta) =  4\mbox{sgn}(\kappa) |\kappa \delta(\alpha-\beta)|. 
\]
If $\kappa > 0$, this difference is positive, and $c_{\rmg}^+$, which is then the group velocity at $+\infty$ is positive, and $c_{\rmg}^-$, which is then the group velocity at $-\infty$, is negative. Thus, the solution is indeed a source. If $\kappa < 0$, the signs of $c_{\rmg}^\pm$ are reversed, but so are the ends to which they correspond. Thus, the solution is indeed a source in all cases.


The (qCGL) equation \eqref{eqs-qCGLintro} is a small perturbation of the cubic CGL, and it has been shown that the above solutions persist as standing sources for an open set of parameter values \cite{Doelman96, SandstedeScheel04}. Furthermore, they are constructed via a transverse intersection of the two-dimensional center-stable manifold of the asymptotic wave train at infinity and the two-dimensional center-unstable manifold of the wave train at minus infinity that is unfolded with respect to the wavenumbers of these wave trains: in particular, the standing sources connect wave trains with a selected wavenumber. These facts are essential for the proof of Lemma 2.1.

\begin{proof}[Proof of Lemma~\ref{lem-NBexistence}] As mentioned above, standing sources have been proven to exist in  \cite{Doelman96} and \cite{SandstedeScheel04}. Let $A_{\mathrm{source}}(x,t)$ be that source, which we can write in the form 
$$A_\mathrm{source}(x,t) = r(x)e^{i(\varphi(x) - \omega_0t)}.$$
Plugging this into \eqref{eqs-qCGLintro}, we find that $(r,\varphi)$ solves 
 \begin{equation}\label{eqs-NB}\begin{aligned}
0 &= r_{xx} + r - r \varphi_x^2- 2\alpha  r_x \varphi_x  - \alpha r \varphi_{xx} - r^3  + \gamma_1 r^5 \\
0 &= r \varphi_{xx} + 2 r_x \varphi_x + \alpha r_{xx}  + \omega_0 r - \alpha r \varphi_x^2 - \beta r^3 + \gamma_2 r^5 .
\end{aligned}\end{equation}
It is shown in \cite{Doelman96} that there exists a locally unique wavenumber $k_0$ and a smooth solution $(r,\varphi)$ of \eqref{eqs-NB} so that $A_\mathrm{source}(x,t)$ converges to the wave trains of the form 
\[
A_\mathrm{wt}(x,t;\pm k_0 ) = \pm r_0(k_0)e^{i(\pm k_0 x - \omega_\mathrm{nl}(k_0)t)},
\]
respectively as $x \to \pm \infty$. Putting this asymptotic Ansatz into \eqref{eqs-NB} then yields
$$\begin{aligned}
\gamma_1 r_0 ^4  - r_0^2 + 1 - k_0^2   = 0, \qquad  \omega_\mathrm{nl} (k_0) = \alpha k_0^2 + \beta r_0^2 - \gamma_2 r_0^4.
\end{aligned}$$
Rearranging terms then gives \eqref{asy-wnl}, and hence item 3 in the lemma. Differentiating the above identities with respect to $k$ and solving for $\frac{d  \omega_\mathrm{nl}(k)}{d k}$, we obtain item 4 as claimed. 

For the first item, since $r(x) \to \pm r_0$ with $r_0 \approx \sqrt {1-k_0^2} \not =0$ (because $\gamma_1\approx 0$), there must be a point at which $r(x)$ vanishes. Without loss of generality, we assume that $r(0)=0$. Since $r(x)$ and $\varphi(x)$ are smooth, evaluating the system \eqref{eqs-NB} at $x=0$ gives
$$ r_{xx} (0) - 2\alpha  r_x(0) \varphi_x(0) = 0 , \qquad 2 r_x(0) \varphi_x(0) + \alpha r_{xx}(0) =0.$$
These equations imply that $r_{xx}(0) = r_x(0)\varphi_x(0) =0$. We now argue that $r_x(0)$ must be nonzero. Otherwise, since $A_\mathrm{source}(0) = r(0) e^{\rmi(\varphi(0) - \omega_0t)}$,
$\partial_{x}A_\mathrm{source}(0) = [r_x(0) + \rmi \varphi_x(0) r(0)] e^{\rmi(\varphi(0) - \omega_0t)}$, and $\partial_{xx}A_\mathrm{source}(0) = [r_{xx}(0) + 2 \rmi r_x(0)\varphi_x(0) + \rmi r(0)\varphi_{xx}(0) - r(0) \varphi_x^2(0)] e^{\rmi(\varphi(0) -\omega_0t)}$, 
we would have $A_\mathrm{source}(0) = \partial_{x}A_\mathrm{source}(0) = \partial_{xx}A_\mathrm{source}(0) = 0$ and so $\partial_x^k A_\mathrm{source}(0) = 0$ for all $k \ge 0$ by repeatedly using the equation \eqref{eqs-qCGLintro}. This would imply at once that $A_\mathrm{source}(x) \equiv 0$. Item 1 is thus proved.  The exponential decay stated in item 2 is a direct consequence of fact that the solutions in \cite{Doelman96, SandstedeScheel04} are shown to lie in the intersection of stable and unstable manifolds of saddle equilibria.
\end{proof}


\subsection{Linearization} \label{sec-linearization}

In order to linearize \eqref{eqs-qCGLintro} around $A_\mathrm{source}(x,t)$, we introduce the perturbation variables $(R,\phi)$ via
\begin{equation}\label{A-intro}
A(x,t) =  [r(x) + R(x,t)] \rme^{\rmi (\varphi(x) + \phi(x,t))}\rme^{-\rmi \omega_0 t},
\end{equation}
where $r$ and $\varphi$ are the amplitude and phase of $A_\mathrm{source}(x,t)$. Throughout the paper, we shall work with the vector variable 
$$ U: = \begin{pmatrix} R \\ r \phi \end{pmatrix} $$ 
for the perturbation $(R,\phi)$. We have the following lemma. 
\begin{lemma}\label{lem-linearization} If the $A(x,t)$, defined in \eqref{A-intro}, solves \eqref{eqs-qCGLintro}, then the linearized dynamics of $U$ are 
\[
(\partial_t - \mathcal{L}) U = 0,
\]
where 
 \begin{equation}\label{def-L}
\mathcal{L}: = D_2 \partial_x^2 - 2 \varphi_x D_1 \partial_x + D_0^\varphi(x) + D_0(x)
\end{equation}
with  
\begin{equation}\label{E:defD}
\begin{aligned}
D_1: = \begin{pmatrix} \alpha & 1 \\ -1 & \alpha\end{pmatrix}, & \qquad  D_2 :=  \begin{pmatrix} 1 & -\alpha \\ \alpha & 1\end{pmatrix}  
\\ D_0^\varphi(x):= \begin{pmatrix} 0&  \frac 1r (2\varphi_x r_x + \alpha r_{xx})\\
0& \frac 1r ( 2\alpha \varphi_x r_x - r_{xx})
\end{pmatrix} ,&\qquad D_0(x) : =  \begin{pmatrix} 1 - 3r^2 - \alpha\varphi_{xx} - \varphi_x^2 + 5 \gamma_1 r^4 & 0\\
\omega_0 -3\beta r^2 - \alpha \varphi_x^2 + \varphi_{xx} + 5\gamma_2 r^4 & 0\end{pmatrix} .
\end{aligned}
\end{equation}
\end{lemma}
\begin{proof} Plugging \eqref{A-intro} into \eqref{eqs-qCGLintro} and using equation \eqref{eqs-NB}, we obtain the linearized system
$$\begin{aligned}
R_t &= R_{xx} -2\alpha\varphi_x R_x + (1 - 3r^2 - \alpha\varphi_{xx} - \varphi_x^2 + 5 \gamma_1 r^4)R
-\alpha r \phi_{xx} - 2(\alpha r_x + r\varphi_x)\phi_x
 \\ r\phi_t &= r\phi_{xx}  - (2\alpha r \varphi_x - 2r_x )\phi_x + \alpha R_{xx} + 2\varphi_x R_x  + (\omega_0 + \varphi_{xx} - \alpha\varphi_x^2-3\beta r^2 + 5 \gamma_2 r^4) R .
\end{aligned}$$ 
A rearrangement of terms yields the lemma. \end{proof}

Next, note that $\mathcal{L}$ is a bounded operator from $\dot H^2(\mathbb{R};\mathbb{C}^2) \bigcap L^\infty(\mathbb{R}; \mathbb{C}^2)$ to $L^2(\mathbb{R};\mathbb{C}^2)$, where the function space $\dot H^2(\mathbb{R};\mathbb{C}^2)$ consists of functions $U = (u_1,u_2)$ so that $u_1 \in H^2(\mathbb{R};\mathbb{C})$ and  $\partial_x u_2, \partial_x^2u_2 \in L^2(\mathbb{R};\mathbb{C}).$ 
(Note that we do not require that $u_2 \in L^2$, since $r\phi$ need not be localized.) We make the following assumption about the spectral stability of $\mathcal{L}$.

\begin{Hypothesis} \label{H:spec-stab}
The spectrum of the operator $\mathcal{L}$ in $L^2(\mathbb{R}, \mathbb{C}^2)$ satisfies the following two conditions:
\begin{itemize}
\item The spectrum does not intersect the closed right half plane, except at the origin.
\item In the weighted space $L^2_\eta(\mathbb{R}, \mathbb{C}^2)$, defined by $\|u\|_\eta^2 = \int e^{-\eta|x|}|u(x)|^2 \rmd x$ with $\eta > 0$ sufficiently small, there are exactly two eigenvalues in the closed right half plane, and they are both at the origin.
\end{itemize}
\end{Hypothesis}

\begin{lemma}\label{lem-Lspectrum} The essential spectrum of the linearized operator $\mathcal{L}$ in $L^2(\mathbb{R},\mathbb{C}^2)$ lies entirely in the left half-plane $\{ \Re \lambda \le  0\}$ and touches the imaginary axis only at the origin as a parabolic curve; see Figure~\ref{fig-Omega}. In addition, the nullspace of $\mathcal{L}$ is spanned by functions $V_1(x)$ and $V_2(x)$, defined by 
\begin{equation}\label{def-V12}V_1(x) := \begin{pmatrix} r_x \\ r\varphi_x \end{pmatrix} \qquad  \mbox{and}\qquad  V_2(x) := \begin{pmatrix} 0 \\ r\end{pmatrix}, \end{equation} 
and the corresponding adjoint eigenfunctions, $\psi_{1,2}(x)$, are both exponentially localized: $|\psi_{1,2}(x)| \leq C e^{-\eta_0|x|}$ for some $C, \eta_0 > 0$. \end{lemma}

\begin{figure}[h]
\centering\includegraphics[scale=.3]{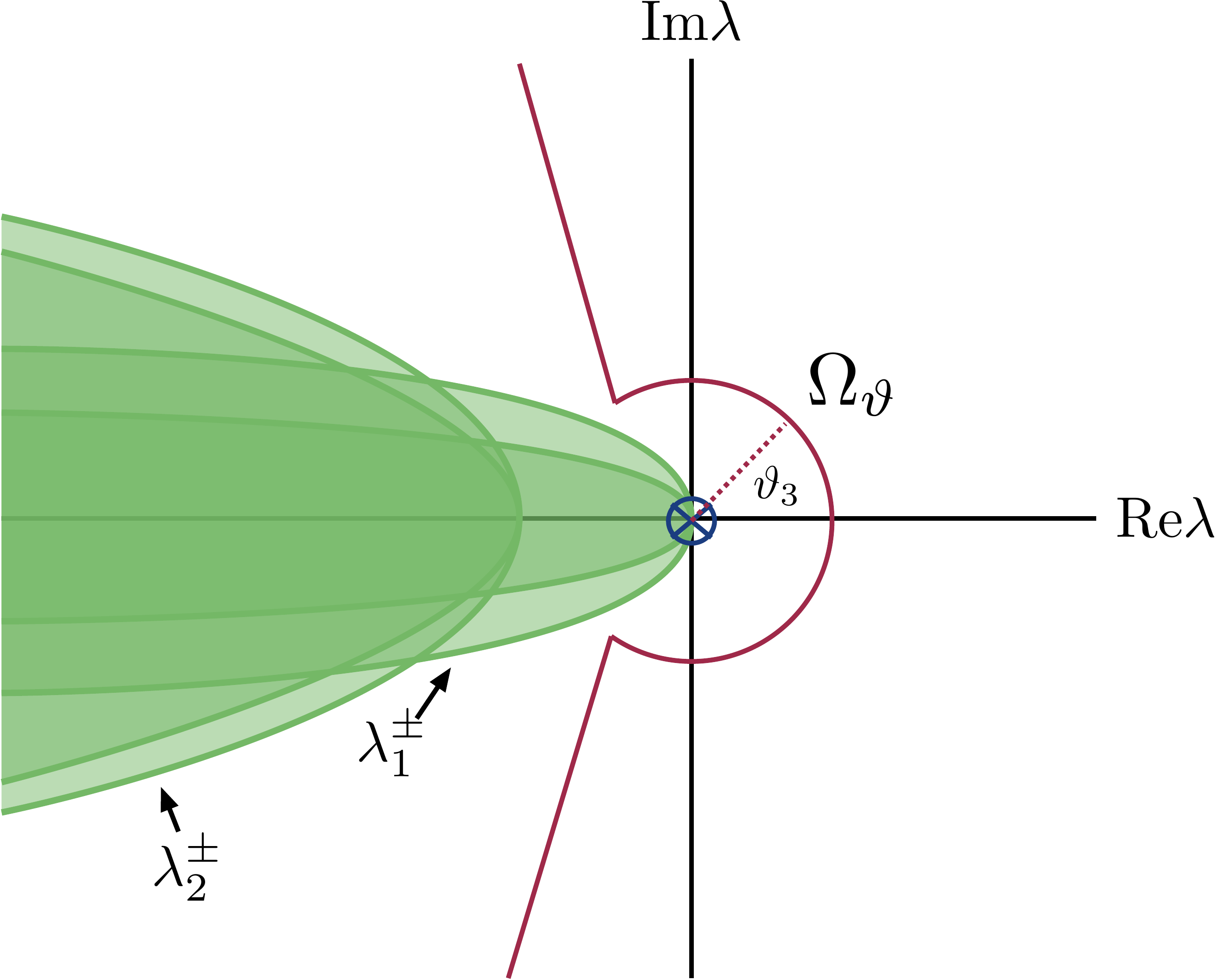} 
\caption{\em The essential spectrum of $\mathcal{L}$ is contained in the shaded region determined by the algebraic curves $\lambda_1^\pm(\kappa), \lambda_2^\pm(\kappa)$. The set $\Omega_\vartheta$ and its boundary $\Gamma = \partial \Omega_\vartheta$, defined in \eqref{def-Omega}, are also illustrated. Note that they both lie outside of the essential spectrum.}
\label{fig-Omega}
\end{figure}

\begin{proof} It follows directly from the translation and gauge invariance of qCGL that $\partial_x A_\mathrm{source}(x,t)$ and $\partial_t A_\mathrm{source}(x,t)$ are exact solutions of the linearized equation about $A_\mathrm{source}(x,t)$. Consequently, $V_1(x),V_2(x)$ belong to the kernel of $\mathcal{L}$. By hypothesis~\ref{H:spec-stab} these are the only elements of the kernel.

Next, by standard spectral theory (see, for instance, Henry \cite{Henry81}), the essential spectrum of $\mathcal{L}$ in $L^2(\mathbb{R},\mathbb{C}^2)$ is the same as that of the limiting, constant-coefficient operator $\mathcal{L}_\pm$ defined by 
\[
\begin{aligned} 
\mathcal{L}_\pm &: = D_2 \partial_x^2 \mp 2 k_0 D_1 \partial_x +  D_0^\infty,\qquad \text{with}\quad D_0^\infty: = -2 r_0^2 \begin{pmatrix} 1 - 2 \gamma_1 r_0^2 &  0\\
\beta  -2\gamma_2 r_0^2 & 0\end{pmatrix} .
\end{aligned}
\]
To find the spectrum of $\mathcal{L}_\pm$, let us denote by $\Lambda_\kappa^\pm$ for each fixed $\kappa \in \mathbb{R}$ the constant matrix
\[
\Lambda_\kappa^\pm : = - \kappa^2 D_2  \mp  2i\kappa k_0  D_1 -   D_0^\infty,
\]
and let $\lambda_1^\pm(\kappa),\lambda_2^\pm(\kappa)$ be the two eigenvalues of $\Lambda_\kappa^\pm$ with $\Re \lambda_1^\pm(0) = 0$. Clearly, $\Re \lambda_j^\pm(\kappa) <0$ for all $\kappa \not = 0$, and $\lambda_1^\pm(\kappa)$ touches the origin as a parabolic curve. Precisely, we have 
\begin{equation}\label{gv-evalue}
\lambda_1^\pm(\kappa) = - i c_\mathrm{g}^\pm  \kappa  -\mathrm{d} \kappa^2+ \mathcal{O}(\kappa^3) 
\end{equation}
for sufficiently small $\kappa$, whereas $\lambda_2^\pm(\kappa)$ is bounded away from the imaginary axis.  It then follows that the spectrum of $\mathcal{L}_\pm$ is confined to the shaded region to the left of the curves $\lambda_j^\pm(\kappa)$ for $\kappa \in \mathbb{R}$, as depicted in Figure~\ref{fig-Omega}. 

Finally, one can see from studying the asymptotic limits of \eqref{E:adj} that the adjoint eigenfunctions are exponentially localized; alternatively, this property was shown more generally in \cite[Corollary~4.6]{SandstedeScheel04}. \end{proof}


\section{Construction of the resolvent kernel}\label{sec-resolvent}
We now construct the resolvent kernel and derive resolvent estimates for the linearized operator $\mathcal{L}$, defined in \eqref{def-L}. These resolvent estimates will be used below, in \S\ref{sec-Green}, to obtain pointwise estimates for the Green's function.

Let $G(x,y, \lambda)$ denote the resolvent kernel associated with the operator $\mathcal{L}$, which is defined to be the distributional solution of the system 
\begin{equation}\label{eqs-resG} (\lambda - \mathcal{L}) G(\cdot,y,\lambda) = \delta_y(\cdot)\end{equation}
for each $y \in \mathbb{R}$, where $\delta_y(x)$ denotes the Dirac delta function centered at $x = y$. The main result in this section provides pointwise bounds on $G(x,y,\lambda)$. We divide it into three different regions: low-frequency ($\lambda \to 0$), mid-frequency ($\theta \le |\lambda|\le M$), and high-frequency ($\lambda \to \infty$). The reason for this separation has to do with the behavior of the spatial eigenvalues of the four-dimensional, first-order ODE associated with \eqref{eqs-resG}, which is given below in \eqref{eqW}. These spatial eigenvalues determine the key features of the resolvent kernel $G$ and depend on the spectral parameter $\lambda$ in such a way so that the asymptotics of $G$ can be characterized differently in these three regions.

Define
\begin{equation}\label{def-Omega} 
\Omega_\vartheta : = \Big \{\lambda\in \mathbb{C}:\quad \Re \lambda \ge -\vartheta_1 - \vartheta_2 |\Im \lambda| , \quad |\lambda| \ge \vartheta_3 \Big\},
\end{equation} 
where the real constants $\vartheta_1,\vartheta_2,\vartheta_3$ are chosen so that $\Omega_\vartheta$ does not intersect the spectrum of $\mathcal{L}$; see Figure~\ref{fig-Omega} and Lemma~\ref{lem-Lspectrum}.  
The following propositions are the main results of this section, and their proofs will be given below.

\begin{proposition}[High-frequency bound]\label{prop-resHF} There exist positive constants $\vartheta_{1,2}$, $M$, $C$, and $\eta$ so that 
\begin{equation*}
\begin{aligned}
|\partial_x^k G(x,y,\lambda)| \le C |\lambda|^{\frac{k-1}{2}} \rme^{-\eta |\lambda|^{1/2} |x-y|} 
\end{aligned}
\end{equation*}
for all $\lambda \in \Omega_\vartheta \bigcap \{|\lambda|\ge M\}$ and $k = 0,1,2$. 
\end{proposition}

\begin{proposition}[Mid-frequency bound]\label{prop-resMF} For any positive constants $\vartheta_3$ and $M$, there exists a $C=C(M, \vartheta_3)$ sufficiently large so that 
\begin{equation*}
|\partial_x^kG(x,y,\lambda) | \le C(M, \vartheta_3),
\end{equation*} 
for all $\lambda \in \Omega_\theta \bigcap \{|\lambda |\le M\}$ and $k = 0,1,2$. 
\end{proposition}

\begin{Proposition}[Low-frequency bound]\label{prop-resLF} There exists an $\eta_3 > 0$ sufficiently small such that, for all $\lambda$ with $|\lambda| < \eta_3$, we have the expansion
\[
\begin{aligned}
G(x,y,\lambda) &= \frac{1}{\lambda} \sum_{j=1}^2  \rme^{\nu^c(\lambda)|x|} V_j(x) \Big\langle \psi_j(y), \cdot \Big\rangle_{\mathbb{C}^2}   + \cO(\rme^{\nu^c(\lambda)|x-y|} ) + \cO(\rme^{-\eta|x-y|}),
\end{aligned}
\]
where
$$
\nu^c(\lambda) = -\frac{\lambda}{c_{\rmg}}+ \frac{\rmd\lambda^2}{c_\mathrm{g}^3} + \cO(\lambda^3).
$$ 
Here $c_\rmg>0$ is the group velocity defined in \eqref{def-cg}, $V_1,V_2$ are the eigenfunctions defined in \eqref{def-V12}, and the adjoint eigenfunctions satisfy $\psi_j(y) = \cO(\rme^{-\eta |y|}),$ for some fixed $\eta>0$.
\end{Proposition}


\subsection{Spatial eigenvalues}
Let us first consider the linearized eigenvalue problem
\begin{equation}\label{eqs-evalues} 
(\lambda - \mathcal{L}) U= 0
\end{equation}
and derive necessary estimates on behavior of solutions as $x \to \pm \infty$. We write the eigenvalue problem \eqref{eqs-evalues} as a four-dimensional first order ODE system. For simplicity, let us denote 
$$B_0 (x,\lambda): = D_2^{-1}\Big[  \lambda \mathbb{I} -  D_0(x) - D_0^\varphi(x) \Big], \qquad C_0: = D_2^{-1} D_1 = \begin{pmatrix} 0 & 1 \\ -1 & 0\end{pmatrix}
$$
where $\mathbb{I}$ denotes the $2\times 2$ identity matrix and $D_{0,1,2}$ are defined in \eqref{E:defD}. Let $W = (U,U_x)$ be the new variable. By \eqref{def-L}, the eigenvalue problem \eqref{eqs-evalues} then becomes 
\begin{equation}\label{eqW}
W_x = \mathcal{A}(x,\lambda)W, \qquad   \mathcal{A}(x,\lambda): = \begin{pmatrix} 0 & \mathbb{I} \\  B_0(x,\lambda) & 2\varphi_x C_0 \end{pmatrix}.
\end{equation}	
Let $\mathcal{A}_\pm(\lambda)$ be the asymptotic limits of $\mathcal{A}(x,\lambda)$ at $x = \pm \infty$, and let 
$$B_0 (\lambda): = D_2^{-1}\Big[  \lambda \mathbb{I} - D_0^\infty \Big].
$$
so that, by Lemma~\ref{lem-NBexistence}, $B_0(x,\lambda) \to B_0(\lambda)$ as $x\to \pm\infty$. Thus, we have
$$ \mathcal{A}_\pm (\lambda) = \begin{pmatrix} 0 & \mathbb{I} \\ B_0(\lambda) & \pm 2 k_0 C_0 \end{pmatrix}.$$
The solutions of the limiting ODE system $W_x = \mathcal{A}_\pm(\lambda)W$ are of the form $W_\infty(\lambda)\rme^{\nu_\pm(\lambda) x}$, where $W_\infty= (w, \nu_\pm(\lambda) w)$, $w\in \mathbb{C}^2$, and $\nu_\pm(\lambda)$ are the eigenvalues of $ \mathcal{A}_\pm (\lambda)$. These eigenvalues are often referred to as spatial eigenvalues, to distinguish them from the temporal eigenvalue parameter $\lambda$, and they satisfy
\begin{equation}
\label{eqs-nu} \det (B_0(\lambda) \pm 2k_0 \nu_\pm(\lambda) C_0 - \nu_\pm^2(\lambda) \mathbb{I} )  =0. 
\end{equation}
The behavior of these spatial eigenvalues as functions of $\lambda$, which determines the key properties of the resolvent kernel, can be understood by considering the limiting cases $\lambda \to 0$ and $|\lambda| \to \infty$, along with the intermediate regime between these limits. The easiest cases are the mid- and high-frequency regimes.


\subsection{Mid- and high-frequency resolvent bounds} \label{sec-HF}

\begin{proof}[Proof of Proposition~\ref{prop-resHF}] The spatial eigenvalues and resolvent kernel can be analyzed in this regime using the following scaling argument. Define
$$\tilde x =|\lambda|^{1/2}x, \quad \tilde \lambda =|\lambda|^{-1}\lambda,\qquad \widetilde W(\tilde x) = W (|\lambda|^{-1/2} x) .$$
In these scaled variables, \eqref{eqW} becomes
\begin{equation}\label{eqtW} 
\widetilde W_{\tilde x} = \tilde{\mathcal{A}}(\tilde \lambda) \widetilde W + \cO(|\lambda|^{-1/2} \widetilde W)  \quad \text{with}\quad \tilde{\mathcal{A}}(\tilde \lambda) = \begin{pmatrix} 
0&0&1&0\\
0&0&0&1\\
\frac{\tilde \lambda}{1+\alpha^2} & \frac{\alpha \tilde \lambda}{1+\alpha^2} &0& 0\\
-\frac{\alpha \tilde \lambda }{1+\alpha^2} & \frac{\tilde \lambda}{1+\alpha^2} & 0 & 0
 \end{pmatrix}.
 \end{equation}
The eigenvalues of $\tilde{\mathcal{A}}(\tilde \lambda)$ are 
\[
\pm \sqrt{\frac{\tilde \lambda (1\pm i \alpha)}{(1+\alpha^2)}}. 
\]
For $\lambda \in \Omega_\vartheta$, $\tilde \lambda$ is on the unit circle and bounded away from the negative real axis. Therefore, there exists some fixed $\eta>0$ so that the real part of $\sqrt{\tilde \lambda (1\pm i \alpha)}$ is greater than $\eta$ for all $\lambda \in \Omega_\vartheta$ as $|\lambda| \to \infty$. See Figure~\ref{fig-spatial-evalues}a.  

\begin{figure}[h]
\centering\includegraphics[scale=.35]{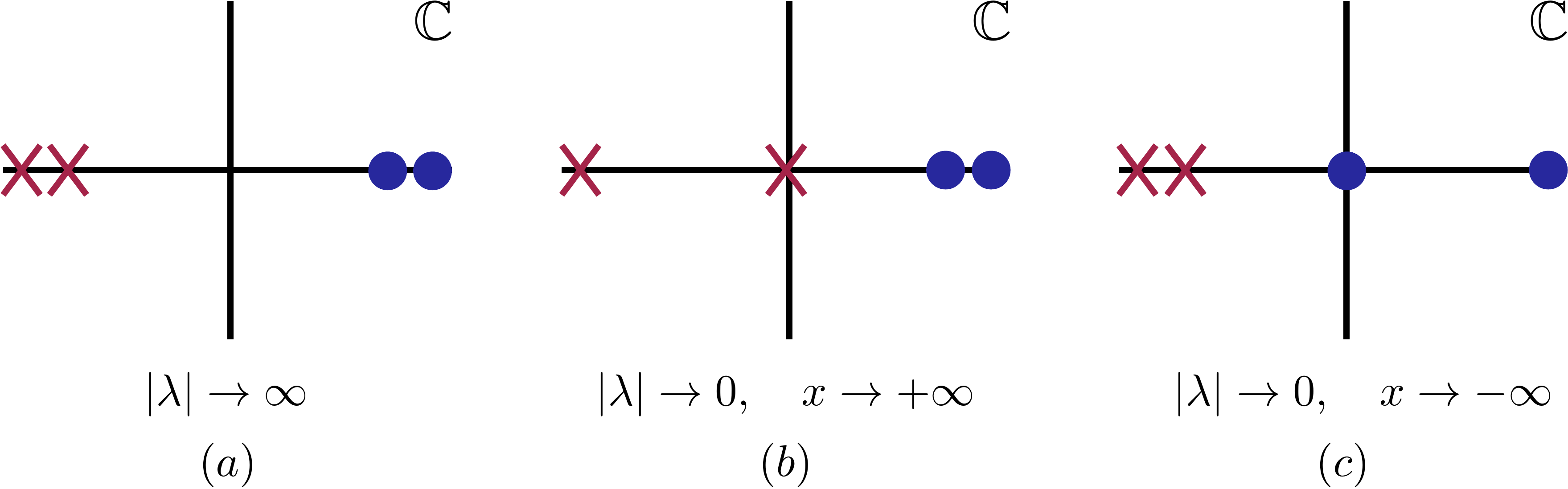} 
\caption{\em Illustration of the spatial eigenvalues (depicted on the real axis for convenience). For $|\lambda| \to \infty$ they are well-separated from the imaginary axis. When $|\lambda| \to 0$ (see \S~\ref{sec-ResLFbound}), one spatial eigenvalue, determined by the spatial limit $x \to \pm\infty$, approaches the origin.}
\label{fig-spatial-evalues}
\end{figure}

Furthermore, we can use this asymptotic behavior to determine information about the resolvent kernel. Let $\tilde P^{s,u}(\tilde \lambda)$ be the eigenprojections associated with $\tilde{\mathcal{A}}(\tilde \lambda)$ onto the stable and unstable subspaces, which are smooth with respect to $\lambda$. It we define $\widetilde W^{s,u} (\tilde x)= \tilde P^{s,u}(\tilde \lambda) \widetilde W(\tilde x)$, then
\begin{equation}
\label{Wpm-flow} \begin{pmatrix} \widetilde W^s\\\widetilde W^u \end{pmatrix}_{\tilde x} = \begin{pmatrix} \tilde{\mathcal{A}}^s(\tilde \lambda) & 0 \\ 0& \tilde{\mathcal{A}}^u(\tilde \lambda) \end{pmatrix}   \begin{pmatrix} \widetilde W^s\\\widetilde W^u \end{pmatrix} +   \cO(|\lambda|^{-1/2} \widetilde W)
\end{equation}
where the matrices $\tilde{\mathcal{A}}^{s,u}(\tilde \lambda)$ satisfy
$$ \Re \langle \tilde{\mathcal{A}}^s(\tilde \lambda)  W,W \rangle_{\mathbb{C}^4} \le  - \eta |W| ^2,\qquad  \Re \langle \tilde{\mathcal{A}}^u(\tilde \lambda)  W,W \rangle_{\mathbb{C}^4} \ge   \eta |W| ^2 ,$$
for all $W\in \mathbb{C}^4$. Here, $\langle \cdot,\cdot\rangle_{\mathbb{C}^4}$ denotes the usual inner product in $\mathbb{C}^4$. Taking the inner product of equation \eqref{Wpm-flow} with $(\widetilde W^s,\widetilde W^u)$, we get 
$$\begin{aligned}
 \frac 12 \frac{d}{d\tilde x} |\widetilde W^s |^2 &\quad \le\quad  -\eta |\widetilde W^s|^2  + \cO(|\lambda|^{-1/2} |\widetilde W|^2) 
 \\
  \frac 12 \frac{d}{d\tilde x} |\widetilde W^u |^2 &\quad \ge~\quad   \eta |\widetilde W^u|^2  + \cO(|\lambda|^{-1/2} |\widetilde W|^2) .
  \end{aligned}$$
This shows that for $|\lambda|$ is sufficiently large, the stable solutions decay faster than $\rme^{-\eta |\tilde x|}$ and the unstable solutions grow at least as fast as $\rme^{\eta|\tilde x|}$. Thus, the resolvent kernel $\widetilde G_W(\tilde x,\tilde y,\tilde \lambda)$ associated with the scaled equation \eqref{eqtW} satisfies the uniform bound 
$$ |\widetilde G_W(\tilde x,\tilde y,\tilde \lambda)| \le C \rme^{-\eta |\tilde x-\tilde y|}.$$
Going back to the original variables, the resolvent kernel associated with \eqref{eqW} satisfies 
$$ | G_W(x,y, \lambda)| \le C \rme^{-\eta |\lambda|^{1/2} |x-y|}.$$
The resolvent kernel $G(x,y,\lambda)$ is by definition is the two-by-two upper-left block of the matrix $G_W(x,y,\lambda)$. This proves the proposition.    \end{proof}

\begin{proof}[Proof of Proposition~\ref{prop-resMF}] The proof is immediate by the analyticity of $G(x,y,\lambda)$ in $\lambda$. \end{proof}


\subsection{Low frequency resolvent bounds via exponential dichotomies}\label{sec-ResLFbound}

As $\lambda$ approaches the origin, it approaches the boundary of the essential spectrum, due to Lemma~\ref{lem-Lspectrum}. Here the asymptotic matrices $\mathcal{A}_\pm(\lambda)$ lose hyperbolicity, and there is one spatial eigenvalue from each $+\infty$ and $-\infty$ that approaches zero. Equation \eqref{eqs-nu} implies that this eigenvalue satisfies $\nu_\pm(\lambda) = \mathcal{O}(\lambda)$.  More precisely, we have
\begin{equation}\label{E:eval_exp}
\nu_\pm(\lambda) = -\lambda/c_{\rmg}^\pm + \frac{\rmd \lambda^2}{c_\mathrm{g}^3} + \cO(\lambda^23,
\end{equation}
where $c_\rmg^\pm$ is defined as in \eqref{def-cg}. Since $c_\rmg^+ >0$ and $c_\rmg^- < 0$, we see  that, at $+\infty$, there are three spatial eigenvalues that are bounded away from the imaginary axis as $\lambda \to 0$, two of which have positive real part and one of which has negative real part, whereas at $-\infty$ there are three spatial eigenvalues that are bounded away from the imaginary axis as $\lambda \to 0$, two of which have negative real part and one of which has positive real part. See Figures~\ref{fig-spatial-evalues}b and~\ref{fig-spatial-evalues}c.

Consider the positive half line $x \geq 0$. Let $\nu_+^{s}(\lambda)$ denote the spatial eigenvalue, coming from $+\infty$, that has negative real part bounded away from zero for all $\lambda > 0$ and hence corresponds to a strong-stable direction. Similarly, let $\nu_+^{c}(\lambda)$ denote the spatial eigenvalue, coming from $+\infty$, that has negative real part that is $\mathcal{O}(\lambda)$ as $\lambda \to 0$, and hence corresponds to a center-stable direction. We then have $\nu_+^{ s}(\lambda)  =  -\eta + \cO(\lambda) $ for some $\eta >0$ and $\nu_+^c(\lambda) = -\lambda/ c_{\rmg}+\mathcal{O}(\lambda^2)$ as $\lambda \to 0$. A direct application of the conjugation lemma (see, for example, \cite{ZH}) shows that a basis of bounded solutions to equation \eqref{eqW} on the positive half line $\mathbb{R}^+$ consists of 
\begin{equation}\label{E:poshalfline}
\rme^{\nu_+^c(\lambda) x } W_+^c (x,\lambda) , \qquad \rme^{\nu_+^s(\lambda) x} W_+^s(x,\lambda) 
\end{equation}
where $W_+^j(x,\lambda)$  approaches $W_\infty^j(\lambda)$ exponentially fast as $ x \to +\infty$ for each $j = c,  s$. The eigenfunctions $V_{1,2}$, defined in \eqref{def-V12}, correspond to solutions $W_{1,2}$ of $\eqref{eqW}$ for $\lambda = 0$. Since,
\[
\lim_{x\to\pm\infty}W_1(x) = \begin{pmatrix} 0 \\ r_0k_0 \\ 0 \\ 0 \end{pmatrix}, \qquad \lim_{x\to\pm\infty}W_1(x) = \begin{pmatrix} 0 \\ \pm r_0 \\ 0 \\ 0 \end{pmatrix}
\]
we see that
\[
W_{1,2}(x) = a^c_{1,2} \rme^{\nu_+^c(0) x } W_+^c (x,0)  + a^s_{1,2} \rme^{\nu_+^s(0) x } W_+^s (x,0), \qquad x \geq 0, 
\]
where $a^c_{1,2} \neq 0$. There are also corresponding solutions $\rme^{ - \nu_+^j(\lambda) y } \Psi_+^j(y,\lambda)$ of the adjoint equation associated with \eqref{eqW},
\begin{equation}\label{E:adj}
W^*_x = - \mathcal{A}(x, \bar \lambda)^T W^*.
\end{equation} 
For $\lambda = 0$, the adjoint eigenfunctions are denoted by $\Psi_{1,2}(x)$, and they are related to the adjoint eigenfunctions $\psi_{1,2}$ via $(\Psi_{1,2})_2 = D_2^T \psi_{1,2}$, where $(\Psi_{1,2})_2$ denotes the second component of $\Psi_{1,2}$.

The four-dimensional ODE associated with equation \eqref{eqs-resG} is just the system \eqref{eqW} with an additional nonautonomous term corresponding to the Dirac delta function. Let $\Phi_+^{s,u}$ denote the associated exponential dichotomy. In other words, $\Phi_+^s(x,y,\lambda)$ decays to zero exponentially fast as $x, y \to \infty$, for $x \geq y \geq 0$, and $\Phi_+^u(x,y,\lambda)$ decays to zero exponentially fast as $x, y \to \infty$, for $y \geq x \geq 0$.
The resolvent kernel, which solves \eqref{eqs-resG}, on the positive half line corresponds to the upper-left two-by-two block of $\Phi_+^{s,u}$. Equation \eqref{E:poshalfline} implies that $\Phi_+^\mathrm{s}(x,y,\lambda)$ can be written
\begin{eqnarray}\label{def-Phis+}
\Phi_+^\mathrm{s}(x,y,\lambda)  &= & \rme^{\nu_+^c(\lambda)(x-y)} W_+^c(x,\lambda)\langle \Psi_+^c(y,\lambda), \cdot \rangle + \rme^{\nu_+^{ s}(\lambda)(x-y)} W_+^{ s}(x,\lambda)\langle \Psi_+^{ s}(y,\lambda), \cdot \rangle 
\end{eqnarray}
for $0 \leq y \leq x $, where all terms are analytic in $\lambda$. Here $\langle \cdot,\cdot\rangle$ denotes the usual inner product in $\mathbb{C}^4$. Similarly, we have $\Phi_+^\mathrm{u}(x,y,\lambda)  =  \mathcal{O}(\rme^{-\eta|x-y|}) $ for $0 \leq x \leq y$. A similar construction can be obtained for the negative half line using the unstable spatial eigenvalues $\nu_-^{u,c}(\lambda)$.

We need to extend the exponential dichotomy so that it is valid not just on the half line, but on the entire real line. This extension will not be analytic, as there is an eigenvalue at the origin, which will correspond to a pole in the resolvent kernel. However, we can construct the extension so that it is meromorphic. The strategy will be similar to that of \cite[\S4.4]{BeckHupkesSandstede10} and \cite[\S4.2]{BeckSandstedeZumbrun10}. 

Write
\[
E_+^\mathrm{j}(\lambda) := \mbox{Rg}\Phi_+^j(0,0,\lambda), \quad E_-^\mathrm{j}(\lambda) := \mbox{Rg}\Phi_-^j(0,0,\lambda), \quad j = s,u.
\]
And note that $\mbox{span}\{ W_1, W_2\} = E_+^\mathrm{s}(0)$. Similarly, since $\mbox{span}\{ W_1, W_2\} = E_-^\mathrm{u}(0)$. Next, set
\[
E_0^\mathrm{pt} := \mbox{span}\{ W_1(0), W_2(0)\}, \quad E_0^\psi := \mbox{span}\{ \Psi_1(0), \Psi_2(0)\},
\]
so that $E_0^\mathrm{pt} \oplus E_0^\psi = \mathbb{C}^4$. The following lemma is analogous to \cite[Lemma 4.10]{BeckHupkesSandstede10} and \cite[Lemma 6]{BeckSandstedeZumbrun10}, and more details of the proof can be found in those papers.

\begin{Lemma}\label{lem:mer-ext}
There exists an $\epsilon > 0$ sufficiently small such that, for each $\lambda \in B_\epsilon(0)\setminus \{0\}$, there exists a unique map $h^+(\lambda): E_+^u(\lambda) \to E_+^s(\lambda)$ such that $E_-^u(\lambda) = \mathrm{graph} h^+(\lambda)$. Furthermore, this mapping is of the form
\begin{equation}\label{def-hp}
h^+(\lambda): E_0^\psi \to E_0^\mathrm{pt}, \qquad h^+(\lambda) = h_p^+(\lambda) + h_a^+(\lambda),
\end{equation}
where $h^+_a$ is analytic for $\lambda \in B_\epsilon(0)\setminus \{0\}$ and 
\begin{equation}\label{def-hp2}
h_p^+ (\lambda) W^\psi = \frac{1}{\lambda} M W^\psi
\end{equation}
with respect to the bases $\{W_1(0), W_2(0)\}$ and $\{\Psi_1(0), \Psi_2(0)\}$, where $M = (M^\psi(0))^{-1}$ and $M^\psi(0)$ is defined in \eqref{E:def_mij}. Similarly, for each $\lambda \in B_\epsilon(0)\setminus \{0\}$, there is a unique map $h^-(\lambda) : E_-^s(\lambda) \to E_-^u(\lambda)$ so that $E_+^s(\lambda) = \mathrm{graph}h^-(\lambda)$. This map has a meromorphic representation analogous to the one given above for $h_+(\lambda)$, with $h^-_p (\lambda) = -h^+_p (\lambda)$.
\end{Lemma}

\begin{proof}
In this proof we will use the coordinates $(W^\mathrm{pt}, W^\psi) \in E_0^\mathrm{pt} \oplus E_0^\psi$. Because the exponential dichotomies $\Phi_\pm^\mathrm{s}(x,y,\lambda)$ and $\Phi_\pm^\mathrm{u}(x,y,\lambda)$ are analytic for all $x \geq y \geq 0$ and $y \geq x \geq 0$, respectively, there exist functions $h^\psi(\lambda)$, $g^{\psi, \mathrm{pt}}(\lambda)$ that are analytic for all $\lambda$ near zero and such that
\begin{eqnarray*}
E_-^\mathrm{u}(\lambda): & \quad &
\tilde{W} =  \tilde{W}^\mathrm{pt} + \lambda h^{\psi}(\lambda) \tilde{W}^\mathrm{pt} \nonumber \\ 
E_+^\mathrm{s}(\lambda): & \quad &
W = W^\mathrm{pt} + \lambda g^{\psi}(\lambda)W^\mathrm{pt} \\ \nonumber
E_+^\mathrm{u}(\lambda): & \quad &
W  =  W^\psi + g^{\mathrm{pt}}(\lambda) W^\psi. 
\end{eqnarray*} 
The superscripts indicate the range of the associated function. For example, $g^\psi(\lambda)W^\mathrm{pt} \in E_0^\psi$ for all $W^\mathrm{pt} \in E_0^\mathrm{pt}$. We want to write $E_-^\mathrm{u}(\lambda)$ as the graph of a function $h_+(\lambda): E_+^\mathrm{u}(\lambda) \to E_+^\mathrm{s}(\lambda)$. This requires 
\[
\underbrace{ \tilde{W}^\mathrm{pt} + \lambda h^{\psi}(\lambda)\tilde{W}^\mathrm{pt}}_{\in E_-^\mathrm{u}}  = 
\underbrace{\left[W^\psi + g^{\mathrm{pt}}(\lambda) W^\psi\right]}_{\in E_+^\mathrm{u}}
+ \underbrace{\left[W^\mathrm{pt} + \lambda g^{\psi}(\lambda) W^\mathrm{pt}\right]}_{\in E_+^\mathrm{s}},
\]
where for each $\tilde W^\mathrm{pt}$ we need to write $W^\mathrm{pt}$ in terms of $W^\psi$ 
so that the above equation holds. In components, 
\[
\tilde{W}^\mathrm{pt} = W^\mathrm{pt} + g^{\mathrm{pt}}(\lambda) W^\psi, \quad \lambda h^{\psi}(\lambda) \tilde{W}^\mathrm{pt}= W^\psi + \lambda g^{\psi}(\lambda)W^\mathrm{pt}, 
\]
which implies
\[
\lambda h^{\psi}(\lambda) \left(W^\mathrm{pt} + g^{\mathrm{pt}}(\lambda) W^\psi\right) = W^\psi + \lambda g^{\psi}(\lambda)W^\mathrm{pt}. 
\]
Rearranging the terms in this equation, we find
\begin{equation} \label{E:one}
\lambda\left(h^{\psi}(\lambda)-  g^{\psi}(\lambda)\right) W^\mathrm{pt} = \left( 1 - \lambda h^{\psi}(\lambda)g^{\mathrm{pt}}(\lambda)\right) W^\psi.
\end{equation}
We need to consider 
\[
M^\psi(\lambda) := h^{\psi}(\lambda)-  g^{\psi}(\lambda): E^\mathrm{pt}_0 \to E^\psi_0.
\]
For the moment, assume that $M^\psi(0)$ is invertible, with 
\begin{equation}\label{E:mpsi-inv}
M^\mathrm{pt}(0) := M^\psi(0)^{-1}: E_0^\psi \to E_0^\mathrm{pt}, 
\end{equation}
where $(M^\psi(0))_{ij}$ is given below in \eqref{E:def_mij}.
We then know that $M^\psi(\lambda)$ is invertible for $\lambda$ sufficiently near zero, and we can write
$M^\mathrm{pt}(\lambda) = M^\mathrm{pt}(0) + \lambda \tilde{M}^\mathrm{pt}(\lambda)$, where $\tilde{M}^\mathrm{pt}(\lambda)$ is analytic in $\lambda$. Hence, (\ref{E:one}) can be solved via
\begin{eqnarray*}
W^\mathrm{pt} &=& \frac{1}{\lambda} \left(M^\mathrm{pt}(0) + \lambda \tilde{M}^\mathrm{pt}(\lambda)\right) \left( 1 - \lambda h^{\psi}(\lambda)g^{\mathrm{pt}}(\lambda)\right) W^\psi \\
&=& \frac{1}{\lambda} M^\mathrm{pt}(0)V^\psi + [\tilde{M}^\mathrm{pt}(\lambda)(1 - \lambda h^{\psi}(\lambda)g^{\mathrm{pt}}(\lambda)) -M^\mathrm{pt}(0)h^{\psi}(\lambda)g^{\mathrm{pt}}(\lambda)]W^\psi + \mathcal{O}(\lambda)W^\psi \\
&=& \Big(\frac{1}{\lambda} M^\mathrm{pt}(0) + \tilde{h}_+(\lambda)\Big)W^\psi, 
\end{eqnarray*}
where $\tilde{h}_+$ is analytic in $\lambda$. 

It only remains to justify (\ref{E:mpsi-inv}). We have that $M^\psi(0) = h^\psi(0) - g^\psi(0)$, where $h^\psi$ and $g^\psi$ represent the graphs of $E_-^u$ and $E_+^s$, respectively. Following the same argument as in the proof of \cite[Lemma 6]{BeckSandstedeZumbrun10}, we find that
\[
\langle \Psi_i(0), (h^\psi(0) - g^\psi(0))W_j(0)\rangle = \int_\mathbb{R} \langle \Psi_i(y), \begin{pmatrix} 0 & 0 \\ D_2^{-1} & 0 \end{pmatrix} W_j(y) \rangle \rmd y.
\]
Recall that the first component of $W_j(y)$ corresponds to the eigenfunctions $V_{j}$ defined in (\ref{def-V12}), and $(D_2^{-1})^T (\Psi_{j})_2$ corresponds to $\psi_j$. Thus, we find that
\begin{equation} \label{E:def_mij}
(M^\psi(0))_{i,j} = \int_\mathbb{R} \langle \psi_i(y),V_j(y)\rangle \rmd y .
\end{equation}
It now follows from \cite[Corollary 4.6]{SandstedeScheel04} in conjunction with Hypothesis \ref{H:spec-stab} that $M^\psi(0)$ is invertible: as shown there, invertibility of $M^\psi(0)$ encodes precisely the property that the linearization about the source in an appropriately weighted space has zero as an eigenvalue with multiplicity two.
\end{proof}

Following \cite{BeckHupkesSandstede10, BeckSandstedeZumbrun10}, the meromorphic extension of the exponential dichotomy for $x > y$ is then given by 
\begin{equation}\label{E:res_ker}
\Phi(x,y, \lambda) := \left\{\begin{array}{lcl}
\tilde{\Phi}^\mathrm{s}_+(x,y,\lambda) & & 0\leq y< x \\[1ex]
\tilde{\Phi}_+^\mathrm{s}(x,0,\lambda)\tilde{\Phi}_-^\mathrm{s}(0,y, \lambda) & & y<0\leq x \\[1ex]
\tilde{\Phi}^\mathrm{s}_-(x,y,\lambda) & & y< x <0
\end{array}\right.
\end{equation}
with 
\begin{equation*}
\begin{array}{rclcl}
\tilde{\Phi}_+^\mathrm{s}(x,y,\lambda) & := & \Phi_+^\mathrm{s}(x,y,\lambda) - \Phi_+^\mathrm{s}(x,0,\lambda) h^+(\lambda) \Phi_+^\mathrm{u}(0,y,\lambda)  & & 0\leq y\leq x \\[1ex]
\tilde{\Phi}_-^\mathrm{s}(x,y,\lambda) & := &  \Phi_-^\mathrm{s}(x,y,\lambda) - \Phi_-^\mathrm{u}(x,0,\lambda) h^-(\lambda) \Phi_-^\mathrm{s}(0,y,\lambda) & & y\leq x\leq0
\end{array}
\end{equation*}
and \begin{equation}\label{bound-Phisu}
\begin{array}{rclcl}
\Phi_+^\mathrm{s}(x,y,\lambda)  &= & \rme^{\nu_+^c(\lambda)(x-y)} W_+^c(x,\lambda)\langle \Psi_+^c(y,\lambda), \cdot \rangle + \mathcal{O}(\rme^{-\eta|x-y|}) & & 0\leq y\leq x \\
\Phi_-^\mathrm{s}(x,y,\lambda) & = & \mathcal{O}(\rme^{-\eta|x-y|}) & &  y\leq x \leq 0 \\
\Phi_+^\mathrm{u}(0,y,\lambda) & = & \mathcal{O}(\rme^{-\eta|y|}) & & 0\leq y.
\end{array}
\end{equation}
We are now ready to derive pointwise bounds on the resolvent kernel for $\lambda \in B_\epsilon(0)$.

\begin{proof}[Proof of Proposition~\ref{prop-resLF}] We give the details only for $y \leq x$; the analysis for $x \leq y$ is similar. Recall that the resolvent kernel $G$ is just the upper-left two-by-two block of $\Phi$.

\paragraph{Case I: $0\leq y\leq x$.} Due to \eqref{def-hp} and \eqref{E:res_ker}, we have
\[
\Phi(x,y,\lambda) =- \Phi_+^\mathrm{s}(x,0,\lambda)h_p^+(\lambda)\Phi_+^\mathrm{u}(0,y,\lambda) +  \rme^{\nu_+^c(\lambda)(x-y)} W_+^c(x,\lambda)\langle \Psi_+^c(y,\lambda), \cdot \rangle +  \mathcal{O}( \rme^{-\eta|x-y|}) .
\]
By definition \eqref{def-hp2}, we can expand the first term as
\begin{eqnarray*}
\Phi_+^\mathrm{s}(x,0,\lambda)h_p^+(\lambda)\Phi_+^\mathrm{u}(0,y,\lambda) &=& \Phi_+^\mathrm{s}(x,0,\lambda) \sum_{j=1}^2 \frac{1}{\lambda} W_j(0) \langle \Psi_j(y), \Phi_+^\mathrm{u}(0,y,\lambda) \rangle \\
&=& \frac{1}{\lambda} \sum_{j=1}^2 \Phi_+^\mathrm{s}(x,0,\lambda) W_j(0)  \langle\Phi_+^\mathrm{u}(0,y,\lambda)^* \Psi_j(y), \cdot \rangle   
\end{eqnarray*}
By definition \eqref{def-Phis+}, we have 
\[
\Phi_+^\mathrm{s}(x,0,\lambda) W_+(0)  =  \rme^{\nu_+^c(\lambda)x} W_+(x)(1+ \mathcal{O}(\lambda)  )+ \cO(\lambda \rme^{\nu_+^s(\lambda)x} ) = \rme^{\nu_+^c(\lambda)x} W_+(x) +  \cO(\lambda \rme^{\nu_+^c(\lambda)x} )
\]
for any $W_+$ that is a linear combination of $W_1$ and $W_2$, where we have used the fact that $\langle \Psi_+^s(0,0), W_j(0)\rangle = 0$. Similarly, \eqref{bound-Phisu} implies that we also have $\Phi_+^\mathrm{u}(0,y,\lambda)^* \Psi_j(y) = \Psi_j(y) + \mathcal{O}(\lambda \rme^{-\eta|y|})$. Thus, we have 
\begin{eqnarray*}
\Phi_+^\mathrm{s}(x,0,\lambda)h_p^+(\lambda)\Phi_+^\mathrm{u}(0,y,\lambda) &=&  \frac{1}{\lambda} \sum_{j=1}^2  \rme^{\nu_+^c(\lambda)x}W_j(x) \langle \Psi_j(y), \cdot \rangle    + \cO(\rme^{\nu_+^c(\lambda)x} \rme^{-\eta |y|}).
\end{eqnarray*}
We can also write 
$$\rme^{\nu_+^c(\lambda)(x-y)} W^c_+(x,\lambda)\langle \Psi^c_+(y,\lambda), \cdot \rangle  = \rme^{\nu_+^c(\lambda)(x-y)} W_+(x)\langle \Psi_+(y), \cdot \rangle + \cO(\lambda \rme^{\nu_+^c(\lambda)(x-y)}),$$
where $ W_+(x) =  W_+^c(x,0)$ and $ \Psi_+(y)= \Psi_+^c(y,0)$. 

Thus, we have obtained
\begin{equation}\label{eqs-resolventI}\begin{aligned}
G(x,y,\lambda) &= \frac{1}{\lambda} \sum_{j=1}^2  \rme^{\nu_+^c(\lambda)x} V_j(x) \langle \psi_j(y), \cdot \rangle  +\rme^{\nu_+^c(\lambda)(x-y)} V_+(x)\langle \psi_+(y), \cdot \rangle
\\&\quad + \cO((\lambda+\rme^{-\eta |y|})  \rme^{\nu_+^c(\lambda)(x-y)} ) + \cO(\rme^{-\eta|x-y|}),
\end{aligned}\end{equation}
where $V_+(x)$ is a linear combination of $V_1(x)$ and $V_2(x)$.  This proves Proposition~\ref{prop-resLF} in this case. 

\paragraph{Case II: $y\leq x\leq0$.} A a result of \eqref{E:res_ker},  we obtain
\[
\Phi(x,y,\lambda) = \Phi_-^\mathrm{u}(x,0,\lambda)h_p^-(\lambda)\Phi_-^\mathrm{s}(0,y,\lambda) + \mathcal{O}(\rme^{\nu_-^c(\lambda)x}\rme^{-\eta|y|} ) + \mathcal{O}(\rme^{-\eta|x-y|}) 
\]
with \begin{eqnarray*}
\Phi_-^\mathrm{u}(x,0,\lambda)h_p^-(\lambda)\Phi_-^\mathrm{s}(0,y,\lambda) &=& \frac{1}{\lambda} \sum_{j=1}^2   \Phi_-^\mathrm{u}(x,0,\lambda) W_j(x)  \langle\Phi_-^\mathrm{s}(0,y,\lambda)^* \Psi_j(y), \cdot \rangle  .
\end{eqnarray*}
Using an argument similar to the previous case, by \eqref{bound-Phisu} we can write $ \Phi_-^\mathrm{u}(x,0,\lambda) W_j(0)  =  \rme^{\nu_-^c(\lambda)x} W_j(x) + \cO(\lambda \rme^{\nu_-^c(\lambda)x} )$, and also $\Phi_-^\mathrm{s}(0,y,\lambda)^* \Psi_j(0) = \Psi_j(y) + \mathcal{O}(\lambda \rme^{-\eta|y|})$. Thus, we get
\[
G(x,y,\lambda) = \frac{1}{\lambda} \sum_{j=1}^2  \rme^{\nu_-^c(\lambda)x} V_j(x) \langle \psi_j(y), \cdot \rangle  + \cO(\rme^{\nu_-^c(\lambda)x }\rme^{-\eta |y|}+ \rme^{-\eta|x-y|}),
\]
which yields the proposition for $y\le x\le 0$. 

\paragraph{Case III: $y\leq0\leq x$.} Again by \eqref{E:res_ker}, we write
\begin{eqnarray*}
\Phi(x,y,\lambda) &=&\Phi_+^\mathrm{s}(x,0,\lambda)\Phi_-^\mathrm{u}(0,0,\lambda)h_p^-(\lambda)\Phi_-^\mathrm{s}(0,y,\lambda)  + \mathcal{O}(\rme^{\nu_+^c(\lambda)x}\rme^{-\eta|y|} ) \\
&=& \frac{1}{\lambda} \sum_{j=1}^2\Phi_+^\mathrm{s}(x,0,\lambda)\Phi_-^\mathrm{u}(0,0,\lambda) W_j(x)    \langle \Psi_j(y) + \mathcal{O}(\lambda \rme^{-\eta|y|}), \cdot \rangle + \mathcal{O}(\rme^{\nu_+^c(\lambda)x}\rme^{-\eta|y|} ) 
\\ 
&=&\frac{1}{\lambda} \sum_{j=1}^2   \rme^{\nu_+^c(\lambda)x}W_j(x) \langle \Psi_j(y), \cdot \rangle  + \cO(\rme^{\nu_+^c(\lambda)x }\rme^{-\eta |y|}).\end{eqnarray*}
Combining all the three cases yields Proposition~\ref{prop-resLF}. 
\end{proof} 

\begin{remark}
{\em If we denote by $G_R(x,y,\lambda)$ the first row in the matrix $G(x,y,\lambda)$, then $G_R(x,y,\lambda)$ has a better bound than that of $G(x,y,\lambda)$ by $\cO(e^{-\eta |x|})$. This is due to the structure of $V_1(x) = ( r_x , r\varphi_x)$ and $V_2(x) = (0,r),$ since $r_x = \cO(\rme^{-\eta |x|})$. More precisely, we have 
\[\begin{aligned}
G_R(x,y,\lambda) &= -\frac{r_x}{\lambda}\rme^{\nu_+^c(\lambda)x} \langle \psi_1(y), \cdot \rangle +  \cO((\lambda+\rme^{-\eta |y|})  \rme^{\nu_+^c(\lambda)(x-y)} ) + \cO(\rme^{-\eta|x-y|}).\end{aligned}
\]
This last estimate means that the $R$-component of the Green's function will decay faster than the $\phi$-component by a factor of $t^{-1/2}$, after subtracting the terms resulting from the eigenfunctions.}
\end{remark}


\section{Temporal Green's function}\label{sec-Green}

In this section, we derive pointwise bounds on the temporal Green's
function defined by 
\begin{equation}\label{def-Green} 
\mathcal{G}(x,y,t) := \frac{1}{2\pi i} \int_{\Gamma }e^{\lambda t} G(x,y,\lambda) \rmd\lambda,
\end{equation} 
where $\Gamma$ is a contour outside of the essential spectrum and $G(x,y,\lambda)$ is the resolvent kernel constructed in Section~\ref{sec-resolvent}. Since $G(x,y,\lambda)$ is analytic in $\lambda$ outside of the essential spectrum, the contour $\Gamma$ can be taken to be the boundary of the set $\Omega_\vartheta$ which is defined as in \eqref{def-Omega}; see Figure~\ref{fig-Omega}.  

\begin{lemma}\label{lem-Greenfn} If $\mathcal{G}(x,y,t)$ is defined in \eqref{def-Green}, then $\mathcal{G}(x,y,t)$ is the temporal Green's function of $\partial_t - \mathcal{L}$, where $\mathcal{L}$ is the linearized operator defined as in \eqref{def-L}. In particular, the solution to the inhomogeneous  system 
$$ (\partial_t - \mathcal{L}) U (x,t)= f(x,t) $$
is given by the standard Duhamel formula
$$ U(x,t) = \int_{\mathbb{R}} \mathcal{G}(x,y,t) U(y,0)\; \rmd y + \int_0^t \int_\mathbb{R} \mathcal{G}(x,y,t-s) f(y,s)\; \rmd s \rmd s ,$$
as long as the integrals on the right hand side are well-defined. 
\end{lemma}
\begin{proof} By Propositions~\ref{prop-resHF} and~\ref{prop-resMF}, $G(x,y,\lambda)$ is uniformly bounded outside of the essential spectrum, and therefore $e^{\lambda t} G(x,y,\lambda)$ is integrable by moving the contour $\Gamma$ so that $\Gamma = \partial \Omega_\vartheta$ for $\lambda$ large.  Hence, $\mathcal{G}(x,y,t)$ is well-defined for $x\not = y$ and for $t>0$. A direct calculation then yields
$$ (\partial_t - \mathcal{L}) \mathcal{G}(x,y,t) = \delta(x-y)\delta(t).$$
This verifies that $\mathcal{G}(x,y,t)$ is indeed the Green's function. 
\end{proof}

The main result of this section is the following proposition, which contains pointwise bounds on the Green's function. This proposition will be proven in \S\ref{S:large} and \S\ref{S:bounded}, below. 
\begin{Proposition}\label{prop-Green} The Green's function $\mathcal{G}(x,y,t)$ defined as in \eqref{def-Green} may be decomposed as
$$\begin{aligned}
 \mathcal{G}(x,y,t) = e(x,t) V_1(x) \langle \psi_1(y),\cdot \rangle + e(x,t) V_2(x) \langle \psi_2(y),\cdot \rangle + \widehat{ \mathcal{G} } (x,y,t), 
\end{aligned}
$$
where $V_{1,2}(x)$ are defined in \eqref{def-V12}, the adjoint functions satisfy $\psi_{1,2}(y)= \cO(e^{-\eta |y|})$, $e(x,t)$ is the sum of error functions defined in \eqref{def-epm}, and $\widehat{\mathcal{G}}(x,y,t)$ satisfies 
$$\widehat{\mathcal{G}}(x,y,t) = \sum_\pm\cO(t^{-\frac{1}{2}}\rme^{-\frac{(x-y\pm c_\mathrm{g}t)^2}{Mt}} ) +  \cO(\rme^{-\eta(|x-y|+t)}).$$
In addition, the first row of $\widehat{\mathcal{G}}(x,y,t)$, denoted by $\widehat{\mathcal{G}}_{R}(x,y,t)$, satisfies the better bound
$$\widehat{\mathcal{G}}_{R}(x,y,t) =  \sum_\pm\cO(t^{-\frac{1}{2}} (t^{-\frac 12} + \rme^{-\eta |y|})\rme^{-\frac{(x-y\pm c_\mathrm{g}t)^2}{Mt}} ) + \cO(\rme^{-\eta(|x-y|+t)}).$$
\end{Proposition}

This proof is essentially contained in \cite{ZH}. However, for completeness, we outline the key steps. 
The main idea is to deform the contour $\Gamma$ in such a way as to minimize the integral in the definition of $\mathcal{G}$. The choice of the minimizing contour is based on the method known as the
saddle point method, the method of stationary phase, or the method of steepest descents. As we are only integrating with respect to the spectral parameter $\lambda$, we are free to chose the contour to depend on $(x,y,t)$. 

We first split the contour $\Gamma$ into $\Gamma_1 \bigcup \Gamma_2$, with 
\begin{equation}\label{def-Gamma12}\Gamma_1:= \partial B(0,M)\cap \overline \Omega_\vartheta, \qquad \Gamma_2:=
\partial \Omega_\vartheta \setminus B(0,M), 
\end{equation} 
for some appropriate $M$ to be determined below. We consider two cases: when $|x-y|/t$ is sufficiently large and when it is bounded. 

\begin{figure}[h]
\centering\includegraphics[scale=.35]{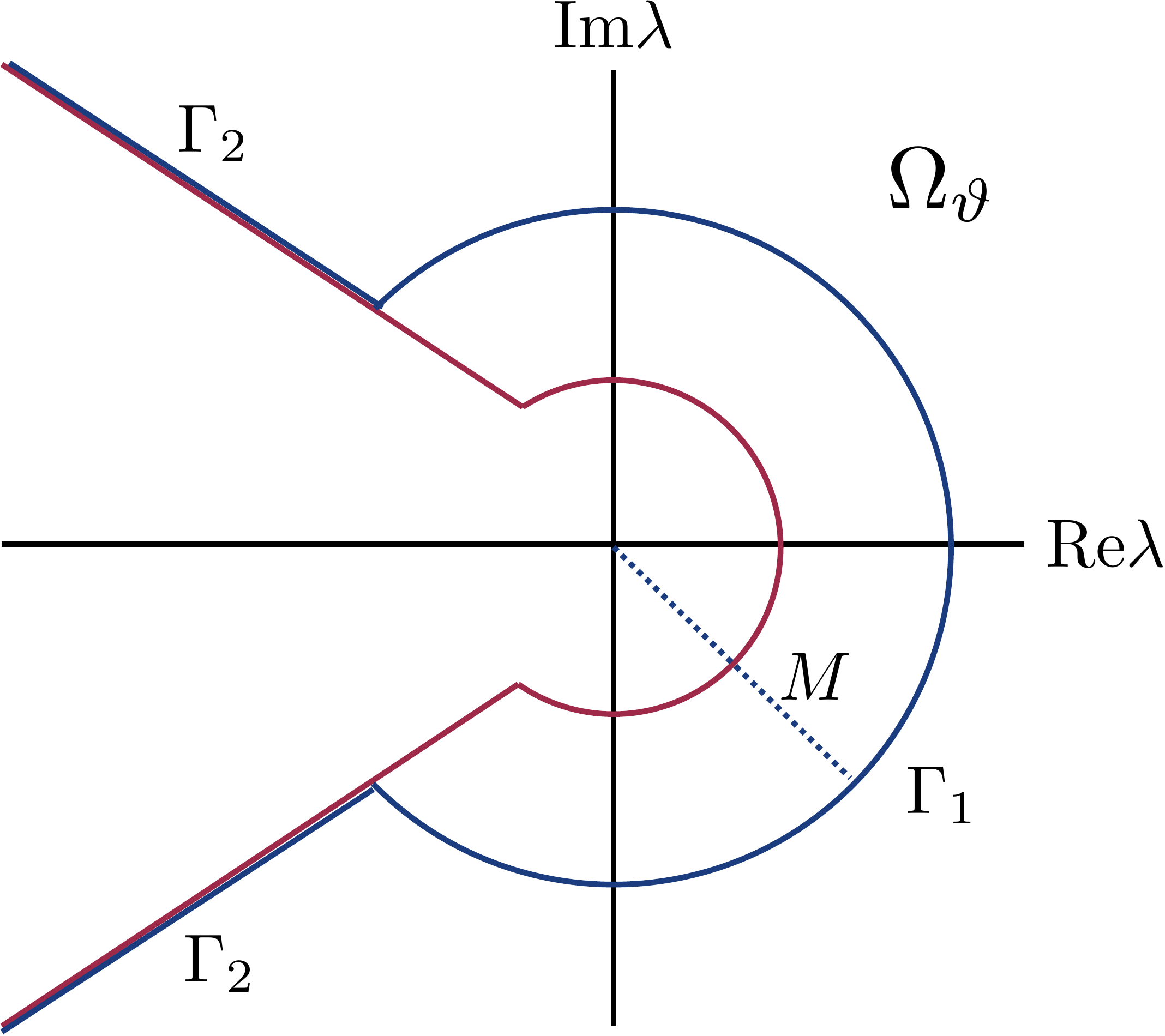} 
\caption{\em Illustration of the contours $\Gamma_1$ and $\Gamma_2$.}
\label{fig-Gamma}
\end{figure}


\subsection{Large $|x-y|/t$} \label{S:large} Set $z: = \eta^2|x-y|/(2t)$ and $M: = (z/\eta)^2$. Since $M$ is sufficiently large, we can apply Proposition~\ref{prop-resHF} so that 
$$
|G(x,y,\lambda) | \le C |\lambda|^{-1/2}
\rme^{-\eta|\lambda|^{1/2}|x-y|}.
$$ Let  $\lambda_0$ and $\lambda_0^*$ be the two points of intersection of $\Gamma_1$ and $\Gamma_2$. We observe that, after possibly making $\eta$ smaller, for all $\lambda \in \Gamma_1$ we have 
$$
\Re \lambda  = M \; \cos(\text{arg}(\lambda)) \le  M(1- \eta \; \text{arg} (\lambda)^2).$$
Also, on $\Gamma_2$ we have $\Re \lambda = - \vartheta (1+|\Im \lambda|)$ (choosing $\vartheta = \min\{\vartheta_1, \vartheta_2\}$ and deforming the contour slightly) and so 
$$\Re \lambda  =  \Re \lambda_0 - \vartheta (|\Im \lambda| - |\Im
\lambda_0|), \quad \lambda \in \Gamma_2.
$$
Thus, we can compute 
$$\begin{aligned}
\Big|\int_{\Gamma_{1}} \rme^{\lambda t} G(x,y,\lambda) \rmd \lambda\Big| &\le C
\int_{\Gamma_{1}} |\lambda|^{-1/2}\, \rme^{\Re \lambda t 
-\eta |\lambda|^{1/2}|x-y| } \, |d\lambda| \\ &\le C \rme^{-\frac{ 
z^{2}t}{\eta^2}} \int_{-\text{arg}(\lambda_0)}^{\text{arg}(\lambda_0)}
M^{-1/2}\rme^{-\eta  M t \; \text{arg}(\lambda)^2}  \;\rmd(\text{arg}(\lambda)) \\&\le Ct^{-1/2}
\rme^{-\frac{z^2 t}{\eta^2}},
\end{aligned} $$
and
$$
\begin{aligned}
\Big |\int_{\Gamma_{2}} \rme^{\lambda t} G(x,y,\lambda)  \rmd\lambda\Big | &\le
C \int_{\Gamma_{2}} |\lambda|^{-1/2}\, \rme^{\Re \lambda t -\eta 
|\lambda|^{1/2}|x-y|} |\rmd\lambda| \\ &\le C
\rme^{\Re(\lambda_{0})t-\eta |\lambda_0|^{1/2} |x-y|} \int_{\Gamma_{2}}
|\lambda|^{-1/2}\rme^{(\Re \lambda-\Re\lambda_{0})t}\
|\rmd\lambda|\\
&\le C \rme^{-\frac{z ^{2}t}{\eta^2}} \int_{\Gamma_2}  |\Im \,
\lambda|^{-1/2} \rme^{-\vartheta[|\Im \, \lambda|-|\Im \, \lambda_{0}|]t}\ | \rmd\, \Im
\lambda|\\ & \le  Ct^{-1/2}
\rme^{-\frac{z ^{2}t}{\eta^2}}.
\end{aligned}
$$
Combining these last two estimates and noting that $z = \frac{\eta^2|x-y|}{2t}$ is large and $\eta$ is small, we have
$$
|\mathcal{G}(x,y,t)| \le C t^{-1/2}\rme^{-\frac{z ^{2}t}{\eta^2}} = C t^{-1/2}\rme^{-\frac{z ^{2}t}{2\eta^2}}\rme^{-\frac{z ^{2}t}{2\eta^2}} \le Ct^{-1/2} \rme^{-\frac \eta 2 t}
\rme^{-\frac{z}{2\eta^2}\frac{\eta^2 |x-y|^2}{2}}
\le C t^{-1/2} \rme^{-\frac \eta 2 (|x-y| + t)} $$
 for $\eta>0$ independent of the amplitude of $\frac{|x-y|}{t}$, as long as it is sufficiently large.


\subsection{Bounded $|x-y|/t$} \label{S:bounded} We now turn to the critical case
where $|x-y|/t \le S$ for some fixed $S$. We first deform $\Gamma$ into $\Gamma_1 \bigcup \Gamma_2$ as in \eqref{def-Gamma12} with $M$ now being sufficiently small. 

The integral over $\Gamma_2$ is relatively straightforward. Again let $\lambda_0$, $\lambda^*_0$ be the points where
$\Gamma_1$ meets $\Gamma_2$. We have $\Re\lambda_0 = -\vartheta(1+\tilde\vartheta)$ and $\Im \lambda_0 = \tilde\vartheta$ for some $\tilde\vartheta>0$. Moreover, on $\Gamma_2$ we have $|G(x,y,\lambda)| \le C|\lambda|^{-\frac{1}{2}} $ by the mid- and high-frequency resolvent bounds. Thus, we can estimate 
$$\begin{aligned}
\Big|\int_{\Gamma_{2}} \rme^{\lambda t} G(x,y,\lambda) \rmd\lambda| &\le C \rme^{\Re
\, \lambda_0 t} \int_{\theta}^\infty  |\Im \, \lambda|^{-\frac{1}{2}}
\rme^{-\vartheta[|\Im \, \lambda|-|\Im \, \lambda_{0}|]t}\ | \rmd\, \Im \lambda|&
\le Ct^{-\frac{1}{2}} \rme^{-\vartheta t}.
\end{aligned}
$$
Noting that $|x-y|\le S t$, we have
\[
\Big|\int_{\Gamma_{2}} \rme^{\lambda t} G(x,y,\lambda) \rmd\lambda|  \leq Ct^{-\frac{1}{2}} \rme^{-\eta (|x-y|+t)}
\]
for some small $\eta>0$.  

In order to estimate the integral over $\Gamma_1$, we take $M$ and $\vartheta$ small enough so that $\Gamma_2$ remains outside of the essential spectrum and the $\lambda$-expansion obtained in \S~\ref{sec-ResLFbound} for the low-frequency resolvent kernel $G(x,y,\lambda)$ holds. We consider several cases, depending on the position of $x$ and $y$. 

\paragraph{Case I: $0\le y\le x$.} In this case, we have the expansion \eqref{eqs-resolventI}, and by equation \eqref{E:eval_exp} we also have
\[
\nu_+^c(\lambda) = -\frac{\lambda}{c_\mathrm{g}} +\frac{\mathrm{d}\lambda^2}{c_\mathrm{g}^3} + \cO(\lambda^3),
\]
where $\mathrm{d}, c_\mathrm{g}>0$ are defined as in \eqref{def-betad} and \eqref{def-cg}, respectively. Let us first estimate the contribution of the term $\mathcal{O}(\lambda^q\rme^{\nu_+^c(\lambda)(x-y)})$. Set \[
z_1 := \frac {x-y-c_\mathrm{g} t}{2t}, \qquad
z_2 := \frac{\mathrm{d}(x-y)}{c_\mathrm{g}^2 t}.
\]
Then $\lambda=z_1/z_2$ minimizes $\nu_+^c(\lambda)(x-y)$ when $\lambda$ is real. Define $\Gamma_{1}$ to  be the portion contained in $\Omega_r:=\{|\lambda|\le r\}$ of the hyperbola
\[
-\frac{1}{c_\mathrm{g}} \Re (\lambda - \lambda^2 \mathrm{d}/c_\mathrm{g}^2) \equiv -\frac{1}{c_\mathrm{g}} (\lambda_\mathrm{min} - \lambda^2_\mathrm{min} \mathrm{d}/c_\mathrm{g}^2)
\]
where $\lambda_\mathrm{min}$ is defined by $z_1/z_2$ if $|z_1/z_2|\le \epsilon$ and  by $\pm \epsilon$ if $z_1/z_2\gtrless \epsilon$, for $\epsilon $ small. 

With these definitions, we readily obtain that
\[
\Re(\lambda t + \nu_+^c(\lambda) (x -y))
\le -\frac{1}{c_\mathrm{g}} (z_1^2t/2z_2) - \eta \Im (\lambda)^2 t 
\le - z_1^2 t/C_0 - \eta \Im (\lambda)^2 t,
\]
for $\lambda \in \Gamma_{1}$; here, we note that $z_2$ is bounded above, and we have used the crucial fact that $z_1$ controls $(|x|+|y|)/t$, in bounding the error term $\cO(\lambda^3)(|x|+|y|)/t$ arising from expansion. Thus, we obtain for any $q$ that
\begin{eqnarray*} 
\int_{\Gamma_{1}} 
|\lambda|^q \rme^{\Re(\lambda t + \nu_+^c(\lambda) (x -y)) } d\lambda
& \le & C \rme^{-z_1^2 t / C_0} \int_{\Gamma_1} (|\lambda_\mathrm{min}|^q + |\Im (\lambda)|^q)\rme^{-\eta \Im(\lambda)^2t} d\lambda \\ & \le &
C t^{-\frac{1}{2} -\frac{q}{2}} \rme^{-z_1^2 t / C_0} ,
\end{eqnarray*}
for suitably large $C,\, M_0>0$. 

Thus, the contribution from $ \mathcal{O}(\lambda^q\rme^{\nu_+^c(\lambda)(x-y)})$ to the Green function bounds
is 
$$\cO(t^{-\frac{1}{2}-\frac q2}\rme^{-\frac{(x-y-c_\mathrm{g}t)^2}{Mt}} ) + \cO(\rme^{-\eta(|x-y|+t)}).
$$
%
%
Clearly, the term  $ \mathcal{O}(\rme^{-\eta(x-y)})$ in $\Phi^s(x,y,\lambda)$ contributes a time- and space-exponential decay: $\cO(\rme^{-\eta(|x-y|+t)})$.  Thus, we are left with the term involving $\lambda^{-1}$. Precisely, consider the term $$- \frac{1}{\lambda} \sum_{j=1}^2   \rme^{(-\lambda/c_\mathrm{g} +\mathrm{d}\lambda^2/c_\mathrm{g}^3 )x} V_j(x) .$$

\begin{figure}[h]
\centering\includegraphics[scale=.35]{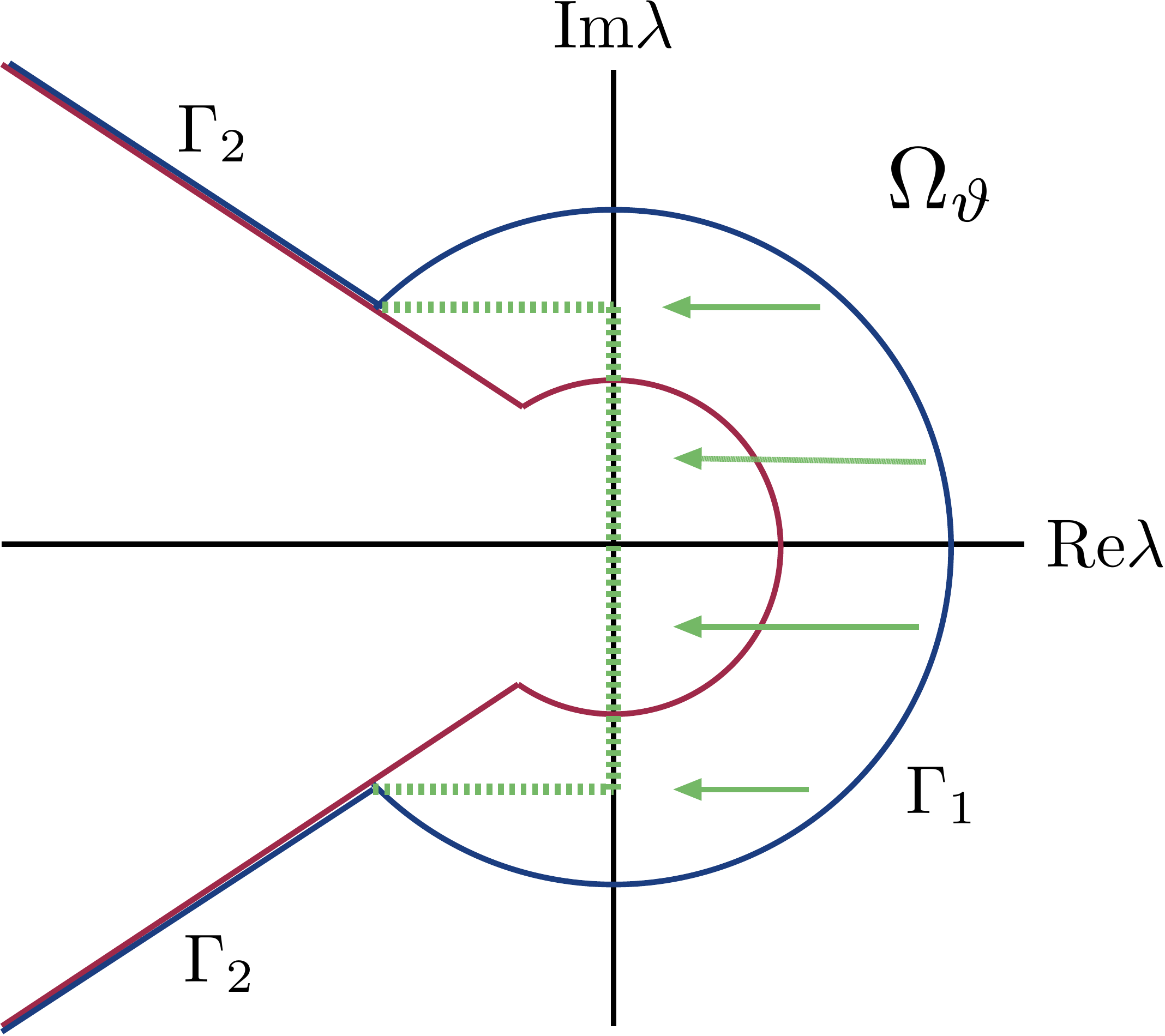} 
\caption{\em The contour $\Gamma_1$ is deformed into three straight lines.}
\label{fig-Gamma1}
\end{figure}
We denote $$\alpha(x,t):= \frac{1}{2\pi i} \int_{\Gamma_1 \bigcap \{|\lambda|\le r\}} \lambda^{-1}
\rme^{\lambda t}\rme^{(-\lambda/c_\mathrm{g} +\mathrm{d}\lambda^2/c_\mathrm{g}^3 )x}d\lambda.$$
Thus, by the Cauchy's theorem, we can move the contour $\Gamma_1$ (as shown in Figure~\ref{fig-Gamma1}) to obtain 
$$
\begin{aligned}
\alpha(x,t)&=
{1\over 2\pi }
\text{\rm P.V. }\int_{-r}^{+r}
(i\xi)^{-1} 
\rme^{i\xi t}
\rme^{(-i\xi/c_\mathrm{g} - \xi^2 \mathrm{d}/c_\mathrm{g}^3))x}
\, d\xi\\
&\quad +{1\over 2\pi i}
\left( \int_{-\eta-ir}^{-ir}
+ \int_{ir}^{-\eta+ir} \right)
\lambda^{-1} 
\rme^{\lambda t}
\rme^{(-\lambda/c_\mathrm{g} +\mathrm{d}\lambda^2/c_\mathrm{g}^3 )x}
\, d\lambda \\
&\quad +
\frac{1}{2}\text{\rm Residue }_{\lambda=0}\, \rme^{\lambda t}
\rme^{(-\lambda/c_\mathrm{g} +\mathrm{d}\lambda^2/c_\mathrm{g}^3 )x},
\end{aligned}
$$ for some $\eta>0$. Rearranging and evaluating the residue term, one then has
$$
\begin{aligned}
\alpha(x,t)&=
\left({1\over 2\pi }
\text{\rm P.V. }\int_{-\infty}^{+\infty}
(i\xi)^{-1} 
\rme^{i\xi(t- x/c_\mathrm{g})}
\rme^{-\xi^2 (\mathrm{d}/c_\mathrm{g}^3 )x}
\, d\xi
+\frac{1}{2} \right)\\
&\quad -{1\over 2\pi }
\left(\int_{-\infty}^{-r} +
\int_{r}^{+\infty} \right)
(i\xi)^{-1} 
\rme^{i\xi(t- x/c_\mathrm{g})}
\rme^{-\xi^2 (\mathrm{d}/c_\mathrm{g}^3 )x}
\, d\xi\\
&\quad +{1\over 2\pi i}
\left( \int_{-\eta-ir}^{-ir}
+ \int_{ir}^{-\eta+ir} \right)
\lambda^{-1} 
\rme^{\lambda t}
\rme^{(-\lambda/c_\mathrm{g} +\mathrm{d}\lambda^2/c_\mathrm{g}^3 )x}
\, d\lambda.
\end{aligned}
$$
Note that the first term in $\alpha(x,t)$ can be explicitly evaluated, again by the Cauchy's theorem and the standard dominated convergence theorem, as 
$$\begin{aligned}\frac 12 &-  \exp \left(-\frac{(x-c_\mathrm{g} t)^2}{4\mathrm{d} x/c_\mathrm{g}}\right) \lim_{r\to \infty}\frac{1}{2\pi}\int_0^{\pi} \rme^{ - \frac{\mathrm{d}x}{c_\mathrm{g}^3}(r\rme^{i\theta} + i\frac{c_\mathrm{g}^2(x-c_\mathrm{g}t)}{2\mathrm{d}x})^2} d\theta
= \frac 12 \Big(1 -  \exp \left(-\frac{(x-c_\mathrm{g} t)^2}{4\mathrm{d} x/c_\mathrm{g}}\right )\Big),
\end{aligned}
$$
which is conveniently simplified to 
\begin{equation}\label{errfn-est}
\frac12\errfn\left(\frac{-x+c_\mathrm{g} t}{\sqrt{4\mathrm{d} |x/c_\mathrm{g}|}}\right) ,\end{equation}
plus a time-exponentially small error. The second and third terms are clearly bounded by $C\rme^{-\eta |x|}$ for $\eta$ sufficiently small relative to $r$, and thus time-exponentially
small for $t\le C|x|$. In the case 
$t\ge C|x|$, $C>0$ sufficiently large, we can simply move the
contour to $[-\eta-ir,-\eta+ir]$ to obtain  (complete)
residue $1$ plus a time-exponentially small error corresponding to 
the shifted contour integral. In this case, we note that the result again can be expressed as the errfn  
\eqref{errfn-est} plus a time-exponentially small error. Indeed, for $t \ge C |x|$, $C$ large, one can estimate   
$$1 - \errfn\left(\frac{-x+c_\mathrm{g} t}{\sqrt{4\mathrm{d} |x/c_\mathrm{g}|}}\right)  \sim  1 - \errfn (\sqrt{t}) = {1\over 2\pi } \int^{+\infty}_{\sqrt{t}} \rme^{-x^2}dx  = \cO(\rme^{-\eta t}).$$

Thus, we have obtained for $0\le y\le x$:
\begin{equation}\label{est-Phi-I1}
\begin{aligned}
G  (x,y,t) &=  
 \sum_{j=1}^2  V_j(x) \langle W_j(y), \cdot \rangle \errfn\left(\frac{-x+c_\mathrm{g} t}{\sqrt{4\mathrm{d} |x/c_\mathrm{g}|}}\right)  +\cO(t^{-\frac{1}{2}}\rme^{-\frac{(x-y-c_\mathrm{g}t)^2}{Mt}})  V_+(x)
 \\
 &\quad +  +\cO(t^{-\frac{1}{2}} (t^{-\frac 12} + e^{-\eta |y|})\rme^{-\frac{(x-y-c_\mathrm{g}t)^2}{Mt}} )+  \cO(\rme^{-\eta(|x-y|+t)}) 
\end{aligned}
\end{equation}
in which we recall that  $V_+(x)$ belongs to the span of $V_1(x)$ and $V_2(x)$. We note that the errfn in \eqref{est-Phi-I1} may be rewritten as
$$
 \errfn\left(\frac{-x+c_\mathrm{g} t}{\sqrt{4\mathrm{d}t}}\right)$$
plus error
$$
\begin{aligned}
\errfn&\left(\frac{-x+c_\mathrm{g} t}{\sqrt{4\mathrm{d} |x/c_\mathrm{g}|}}\right) -  \errfn\left(\frac{-x+c_\mathrm{g} t}{\sqrt{4\mathrm{d}t}}\right) 
 = \cO(t^{-1}\rme^{ (x-c_\mathrm{g} t)^2/Mt}),
\end{aligned}
$$
for $M>0$ sufficiently large. Note also that $\errfn\left(\frac{-x-c_\mathrm{g} t}{\sqrt{4\mathrm{d}t}}\right) $ is time-exponentially small since $x,t,c_\mathrm{g}$ are all positive. Thus, an equivalent expression for $G  (x,y,t)$ is
\[
\begin{aligned}
G  (x,y,t) =  &  \sum_{j=1}^2  V_j(x)  \Big[ e_+(x,t) \langle W_j(y), \cdot \rangle   +\cO(t^{-\frac{1}{2}}\rme^{-\frac{(x-y-c_\mathrm{g}t)^2}{Mt}} ) \Big]  \\ & + \cO(t^{-\frac{1}{2}} (t^{-\frac 12 } + e^{-\eta |y|})\rme^{-\frac{(x-y-c_\mathrm{g}t)^2}{Mt}} ) + \cO(\rme^{-\eta(|x-y|+t)}) , 
\end{aligned}
\]
with $e_\pm (x,t) =  \Big[\errfn\left(\frac{-x\pm c_\mathrm{g} t}{\sqrt{4\mathrm{d}t}}\right) - \errfn\left(\frac{-x\mp c_\mathrm{g} t}{\sqrt{4\mathrm{d}t}}\right)\Big]$. Again, we remark that at the leading term the first row (or the $R$-component) of the Green function matrix $G (x,y,t)$ enjoys a better bound due to the special structure of $V_j(x)$; see \eqref{def-V12}. 

\paragraph{Case II: $y\le x\le 0$.} In this case we recall that
\[
\begin{aligned}
G(x,y,\lambda) 
&= \mathcal{O}(\rme^{\nu_-^c(\lambda)x}\rme^{-\eta |y|})+ \mathcal{O}(\rme^{-\eta|x-y|}) - 
\frac{1}{\lambda} \sum_{j=1}^2   \langle W_j(y), \cdot \rangle   \rme^{\nu_-^c(\lambda)x} V_j(x),
\end{aligned}
\]
where 
\[
\nu_-^c(\lambda) = -\frac{\lambda}{c^-_\mathrm{g}} +\frac{\mathrm{d}\lambda^2}{(c^-_\mathrm{g})^3} + \cO(\lambda^3), 
\]
with $c^-_\mathrm{g} = - c_\mathrm{g} <0$. 
Similar computations as in the previous case yield
\[
\begin{aligned}
G  (x,y,t) =&\sum_{j=1}^2  V_j(x) e_-(x,t)\langle W_j(y), \cdot \rangle  + \cO(t^{-\frac{1}{2}}\rme^{-\eta |y|}\rme^{-\frac{(x+c_\mathrm{g}t)^2}{Mt}} ) + \cO(\rme^{-\eta(|x-y|+t)}) 
\end{aligned}
\]
with $e_\pm (x,t) =  \Big[\errfn\left(\frac{-x\pm c_\mathrm{g} t}{\sqrt{4\mathrm{d}t}}\right) - \errfn\left(\frac{-x\mp c_\mathrm{g} t}{\sqrt{4\mathrm{d}t}}\right)\Big]$. 

\paragraph{Case III: $y\le 0\le x$.} In this case we have
\[
\begin{aligned}
G(x,y,\lambda) 
=&\mathcal{O}(\rme^{-\eta x}\rme^{-\eta|y|}) + \mathcal{O}(\rme^{\nu_+^c(\lambda)x}\rme^{-\eta|y|} )  - 
\frac{1}{\lambda} \sum_{j=1}^2   \langle W_j(y), \cdot \rangle   \rme^{\nu_+^c(\lambda)x}V_{j}(x).
\end{aligned}
\]
Thus, similar computations as done in case I yield 
\[
\begin{aligned}
G  (x,y,t) =& \sum_{j=1}^2  V_j(x) e_+(x,t) \langle W_j(y), \cdot \rangle + \cO(t^{-\frac{1}{2}}\rme^{-\eta |y|}\rme^{-\frac{(x-c_\mathrm{g}t)^2}{Mt}} )+  \cO(\rme^{-\eta(|x|+|y|+t)}). \\
\end{aligned}
\]
This establishes Proposition~\ref{prop-Green}. 


\section{Asymptotic Ansatz}\label{sec-asymptotic}
In this section, we shall construct the asymptotic Ansatz that will be used in the analysis of \eqref{eqs-qCGLintro}. For convenience, let us recall that the qCGL equation \eqref{eqs-qCGLintro} (after rescaling so that $\mu = 1$) is 
\begin{equation}\label{eqs-qCGL}
A_t = (1+i\alpha)A_{xx} + A - (1+i\beta)A|A|^2 + (\gamma_1+i\gamma_2) A |A|^4.  
\end{equation}


\subsection{Setup}
Recall that the source solution of \eqref{eqs-qCGL} is given by
\[
A_\mathrm{source}(x,t) = r(x) e^{\rmi \varphi(x)}e^{-\rmi \omega_0t}.
\]
Given a solution $A(x,t)$ of \eqref{eqs-qCGL}, define
\begin{equation}\label{def-BB}
B(x,t): = A(x+p(x,t),t)  = [r(x) + R(x, t)]e^{i(\varphi(x) + \phi(x, t))}\rme^{-i\omega_0 t},
\end{equation} 
where for the moment $p(x,t)$ is an arbitrary smooth function. The functions $R$ and $\phi$ denote the perturbation in the amplitude and phase variables. We will insert \eqref{def-BB} into \eqref{eqs-qCGL} and derive a system for $U = (R,r\phi)$. We obtain the following simple lemma whose proof will be given below in \S~\ref{sec-pertsystem}. 
\begin{lemma}\label{lem-nonlinearUp} Let $p(x,t)$ be a given smooth function so that the map $(x,t)\mapsto (\xi(x,t),\tau(x,t))$, defined by $\xi= x+p(x,t)$ and $\tau = t$, is invertible. If $A(x,t)$ is a solution of \eqref{eqs-qCGL}, then $U = (R,r\phi)$, where $R$ and $\phi$ are defined by \eqref{def-BB}, satisfies
 \begin{equation}\label{eqs-nonUp}
(\partial_t - \mathcal{L}) U -   \mathcal{T}(p) - \mathcal{Q}(R,\phi,p) = \mathcal{N}(R,\phi,p) .
\end{equation}
In this equation, $\mathcal{L}$ is the linearized operator, $\mathcal{T}(p)$ denotes the linear residual effects resulting from the function $p$, $\mathcal{Q}(R,\phi,p)$ denotes the quadratic nonlinear terms, and $\mathcal{N}(R,\phi,p)$ denotes the higher-order terms. They are defined in equations \eqref{redef-L}, \eqref{def-Plin}, \eqref{def-Q}, and \eqref{good-Re}, respectively.
\end{lemma}

We recall from Lemma~\ref{lem-linearization} that 
 \begin{equation}\label{redef-L}
\mathcal{L}= D_2 \partial_x^2 - 2 \varphi_x D_1 \partial_x + D_0^\varphi(x) + D_0(x),
\end{equation}
with 
\begin{equation} \label{E:defD2}
\begin{aligned}
D_1: = \begin{pmatrix} \alpha & 1 \\ -1 & \alpha\end{pmatrix}, & \qquad  D_2 :=  \begin{pmatrix} 1 & -\alpha \\ \alpha & 1\end{pmatrix}  
\\ D_0^\varphi(x):= \begin{pmatrix} 0&  \frac 1r (2\varphi_x r_x + \alpha r_{xx})\\
0& \frac 1r ( 2\alpha \varphi_x r_x - r_{xx})
\end{pmatrix} ,&\qquad D_0(x) : =  \begin{pmatrix} 1 - 3r^2 - \alpha\varphi_{xx} - \varphi_x^2 + 5 \gamma_1 r^4 & 0\\
\omega_0 -3\beta r^2 - \alpha \varphi_x^2 + \varphi_{xx} + 5\gamma_2 r^4 & 0\end{pmatrix} .
\end{aligned}
\end{equation}
The function $\mathcal{T}(p)$ is defined by 
\begin{equation}\label{def-Plin}\begin{aligned}
\mathcal{T}(p) := p_t \begin{pmatrix} r_x\\r \varphi_x \end{pmatrix}  + 2 r \varphi_x^2 p_x \begin{pmatrix} 1\\\alpha\end{pmatrix}+ r \varphi_x p_{xx} \begin{pmatrix} \alpha \\-1\end{pmatrix},\end{aligned}\end{equation}
and the quadratic term $ \mathcal{Q}(R,\phi,p): = (Q_R(R,\phi,p),Q_\phi(R,\phi,p))$ is defined by
\begin{equation}\label{def-Q}\begin{aligned}
Q_R(R,\phi)  \quad&:=\quad -2\varphi_xR \phi_x  - r \phi_x^2 - 3rR^2  + 10 \gamma_1 r^3 R^2  - 3 r \varphi_x^2 p_x^2 + 4 r \varphi_x p_x \phi_x + 2 \varphi_x^2 R p_x 
\\ 
Q_\phi(R,\phi)   \quad&:= \quad  -R \phi_t -2\alpha\varphi_xR \phi_x  - \alpha r\phi_x^2 - 3 \beta r R^2 + 10 \gamma_2 r^3 R^2 \\&\quad \quad  - r \varphi_x p_t p_x + r p_t \phi_x + \varphi_x p_t R - 3 \alpha r \varphi_x^2 p_x^2 + 4 \alpha r \varphi_x p_x \phi_x + 2\alpha \varphi_x^2 p_x R.
\end{aligned}\end{equation}
The remainder $\mathcal{N}(R,\phi,p)$, corresponding to higher order nonlinear terms as well as terms that are small due to exponential localization in space, satisfies the uniform bound
\begin{equation}\label{good-Re}\begin{aligned}
\mathcal{N}(R,\phi,p) &= \cO\Big[ |R\phi_x^2|+|R_x\phi_x| + |R \phi_{xx}|  +|R|^3 +|RR_x|
 + |p_x|^3+|p_x^2p_t| 
 \\&\quad
 +|p_x p_{xx}| 
 + |p_t R_x| + |p_{xx}|(|R|+|\phi_x|)+|p_x|(|R_x| + |\phi_{xx}|+|R\phi_x| + e^{-\eta_0|x|}) \Big].
\end{aligned}\end{equation}  


\subsection{Approximate solution}\label{S:app-soln}

We will construct an approximate solution $U^a$ that solves \eqref{eqs-nonUp} up to some error terms that are harmless for the nonlinear analysis. In course of constructing the approximate Ansatz, the Burgers-type equation \eqref{intro-Burgers} appears as a governing equation at leading order. Let
\[
L_B: = \partial_t + 2(\alpha - \beta_*) \varphi_x \partial_x- \mathrm{d} \partial_x^2 
\]
denote the linear part of the Burgers equation, where we recall that
\[
\mathrm{d} := (1+\alpha \beta_*) - \frac{2k_0^2 (1+\beta^2_*)}{r_0^2(1-2\gamma_1r_0^2)}, \qquad \beta_* := \frac{\beta - 2 \gamma_2 r_0^2}{1-2\gamma_1 r^2_0}.
\]
Note that the operator $L_B$ corresponds precisely with the curve essential spectrum of the linearized operator $\mathcal{L}$ that touches the origin; see equation \eqref{gv-evalue}.  

Recall from the discussion of \S~\ref{sec-Intro} that the dynamics of $R$ are higher-order, relative to $\phi$, due to the zero-order terms in the $R$-component, which can be seen in the operator $D_0$ defined in \eqref{E:defD2}. In order to construct the approximate solution, we will (asymptotically) diagonalize the operator $D_0$ so that these zero-order terms only appear in the equation for $R$. Variables labeled with a hat correspond to variables to which this diagonalizing transformation has been applied. We first state the approximate solution in these variables, and then undo the diagonalizing transformation. 

Our approximate solution is constructed as follows. First, take $ \mathcal{B}(x,t)$ to be an arbitrary smooth solution to the linear Burgers equation: $L_B \mathcal{B} =0$.
Define
\begin{equation}\label{Burgers-solns}\left\{\begin{aligned}
\hat \phi^a(x,t,\delta^\pm)  &:= \frac {\mathrm{d}}{2q} \Big[ \log \Big(1 + \delta^+\mathcal{B}(x,t)\Big) + \log \Big(1 + \delta^- \mathcal{B}(x,t)\Big)\Big]\\
 p(x,t,\delta^\pm)  &:= \frac {\mathrm{d}}{2qk_0} \Big[ \log \Big(1 + \delta^+ \mathcal{B}(x,t)\Big) - \log \Big(1 + \delta^-  \mathcal{B}(x,t)\Big)\Big]
 \end{aligned}\right.\end{equation}
 where
 \begin{equation}\label{E:defq}
 q = (\alpha - \beta_*) + \frac{4 k_0^2(\gamma_1\beta - \gamma_2)}{(1-2\gamma_1r_0^2)^3}
 \end{equation}
and $\delta^\pm \in \mathbb{R}$ are for the moment arbitrary constants. 
\begin{remark}\label{rem-Burgerschoice}  For each fixed $\delta^\pm$, the function $(\hat \phi^a \pm k_0 p)(x,t,\delta^\pm)$, defined via \eqref{Burgers-solns}, solves the nonlinear Burgers-type equation
\begin{equation}\label{Govern-eqs}
L_B (\hat\phi^a \pm k_0p) = q (\hat \phi_x^a \pm k_0p_x)^2,
\end{equation}
where $q$ is defined in \eqref{E:defq}. In fact, $\hat \phi^a$ and $p$ were chosen essentially by applying the Cole-Hopf transformation to the function $\mathcal{B}$. Therefore, they will exactly cancel the potentially problematic quadratic terms in the nonlinearity. Moreover, \eqref{Govern-eqs} implies that
\begin{equation}\label{Dt-phi} \begin{aligned}
\hat \phi^a_t  +  2(\alpha - \beta_*) \varphi_x \hat\phi^a_x &= \cO\Big(|\hat \phi_{xx}^a| + |\hat \phi_x^a|^2 + |p_{xx}| + |p_x|^2 \Big),
\\ p_t  +  2(\alpha - \beta_*) \varphi_x p_x &= \cO\Big(|\hat \phi_{xx}^a| + |\hat \phi_x^a|^2 + |p_{xx}| + |p_x|^2 \Big).\end{aligned}
\end{equation}
\end{remark}
For the approximate solution of the $R$-component, define 
\begin{equation}\label{def-hatRa} \hat R ^a(x,t,\delta^\pm): = \hat R ^a_0(x,t,\delta^\pm) + \hat R ^a_1(x,t,\delta^\pm), 
\end{equation}
where
\[
\hat R ^a_0(x,t,\delta^\pm): = \frac{k_0}{ r_0 (1 - 2\gamma_1 r_0^2) } \Big( \hat \phi^a_x + \varphi_x p_x\Big)(x,t,\delta^\pm) 
\]
and $\hat R_1^a$ is a higher-order correction defined in \eqref{E:defR1}. Finally, define $\hat U^a := (\hat R^a,r\hat \phi^a)$. 

In the original variables, the approximate solution $U^a = (R^a,r\phi^a)$ is given by 
\begin{equation}\label{def-Uapp} U^a(x,t,\delta^\pm): = S \hat U^a (x,t,\delta^\pm), \qquad\text{with}\quad S: = \begin{pmatrix} 1 & 0\\ \beta_* & -1 \end{pmatrix}.\end{equation}
We shall prove that for each fixed $\delta^\pm$, $U^a(x,t,\delta^\pm)$ solves \eqref{eqs-nonUp} up to good error terms of the form 
\begin{equation}\label{def-Good}
\Upsilon (\hat \phi^a_x,p_x): = \sum_{\psi = \hat \phi^a, p} \Big[ e^{-\eta_0 |x|} |\psi_x|+ |\psi_{xx}\psi_x| + |\psi_x|^3\Big], \end{equation}
for $\eta_0>0$ as in Lemma~\ref{lem-NBexistence}. The main result of this section is the following proposition, whose proof will be given in Section~\ref{sec-Asymptotic}. 

\begin{Proposition}\label{prop-Uapprox} Let $\delta^\pm = \delta^\pm(t)$ be arbitrary smooth functions. The function $U^a = (R^a,r\phi^a)$ defined in \eqref{def-Uapp} solves  
\begin{equation}\label{eqs-Uapp}
 (\partial_t - \mathcal{L})U^a  - \mathcal{T}(p) -  \mathcal{Q}(R^a,\phi^a,p) =\sum_\pm \dot \delta^\pm(t) \Big(\Sigma^\pm(x,t,\delta^\pm(t)) + \cO(|\hat \phi^a_x| +|p_x|)\Big)+  \cO(\Upsilon(\hat \phi^a_x,p_x)).
 \end{equation}
Here $\mathcal{L}$, $\mathcal{T}(p)$ and $\mathcal{Q}(R,\phi,p)$ are defined in \eqref{redef-L}, \eqref{def-Plin}, and \eqref{def-Q}, respectively, $\Upsilon(\hat \phi^a_x,p_x)$ is defined in \eqref{def-Good}, and $\Sigma^\pm(x,t,\delta^\pm)$ are defined by 
\begin{equation}\label{def-Sgpm}\begin{aligned}
\Sigma^\pm(x,t,\delta^\pm): &= \frac{\partial U^a}{\partial \delta^\pm}(x,t,\delta^\pm) -  \frac{\partial p}{\partial \delta^\pm}(x,t,\delta^\pm)\begin{pmatrix}r_x\\ r\varphi_x 
\end{pmatrix} .
\end{aligned}
\end{equation}
\end{Proposition}

We now chose the function $\mathcal{B}(x,t)$ used in \eqref{Burgers-solns} to be 
\begin{equation}\label{choice-Be}
\mathcal{B}(x,t) = e(x,t+1) , \qquad e(x,t) = \errfn\left(\frac{x+c_\rmg t}{\sqrt{4\mathrm{d} t}}\right) - \errfn\left(\frac{x-c_\rmg t}{\sqrt{4\mathrm{d} t}}\right).
\end{equation}
Note that, although it is not the case that $L_B \mathcal{B} = 0$, this equation is satisfied asymptotically in an appropriate sense. In other words, the function $\mathcal{B}$ is still sufficient for our Ansatz in the sense that Proposition~\ref{prop-corUapp}, which will be proven in \S\ref{sec-proofProp52}, holds.

\begin{Proposition}\label{prop-corUapp} Let $\delta^\pm = \delta^\pm(t)$ be arbitrary smooth functions and let $\mathcal{B}(x,t)$ defined as in \eqref{choice-Be}. The approximate solution $U^a$ constructed  as in \eqref{def-Uapp} satisfies  
\[
\begin{aligned}
 (\partial_t - \mathcal{L})U^a  - \mathcal{T}(p) -  \mathcal{Q}(R^a,\phi^a,p) = \frac{\mathrm{d}}{2q}\sum_\pm \frac{\dot \delta^\pm(t)  }{ (1+ \delta^\pm)} \mathcal{E}^\pm(x,t)  +{\mathcal{R}}^{app}_1 (x,t, \delta^\pm),
\end{aligned}
\]
where $\mathcal{E}^\pm(x,t)$ is defined by 
\begin{equation}\label{def-Epm}\begin{aligned}
\mathcal{E}^\pm(x,t) :&= \begin{pmatrix}\frac{\varphi_x \pm k_0}{r_0(1-2\gamma_1r_0^2)} \mathcal{B}_x(x,t)  \\ r\mathcal{B}(x,t)
\end{pmatrix} \mp \frac{\mathcal{B}(x,t)}{k_0}  \begin{pmatrix}r_x\\ r\varphi_x 
\end{pmatrix}
 \end{aligned}
\end{equation} 
and the remainder ${\mathcal{R}}^{app}_1(x,t, \delta^\pm) $ satisfies
\begin{equation}\label{bound-Rdelta}
\begin{aligned}
{\mathcal{R}}^{app}_1 (x,t, \delta^\pm) & \le C (|\delta^+ | + |\delta^- |)e^{-\eta_0 |x|}  \theta(x,t)  + C(1+t)^{-1}\theta(x,t) (|\delta^+|^2 + |\delta^-|^2)
\\&\quad + C(|\delta^+|+|\delta^-|)(|\dot \delta^+|+|\dot \delta^-|)(1+t)^{1/2}\theta(x,t) ,
\end{aligned}
\end{equation}
for $\eta_0>0$ as in Lemma~\ref{lem-NBexistence}. Here $\theta(x,t)$ denotes the Gaussian-like behavior (as in \eqref{Gaussian-like}):
$$ \theta(x,t) = \frac{1}{(1+t)^{1/2}} \left( \rme^{-\frac{(x-c_\rmg t)^2}{M_0(t+1)}} + \rme^{-\frac{(x+c_\rmg t)^2}{M_0(t+1)}}\right) , $$
for some $M_0>0$.
\end{Proposition} 

The following lemma is relatively straightforward, but crucial to our analysis later on. 

\begin{lemma}\label{lem-linEest} Let $\mathcal{E}^\pm(x,t)$ be defined as in \eqref{def-Epm}. Then
$$(\partial_t - \mathcal{L})  \mathcal{E}^\pm (x,t)= \mathcal{O}((1+t)^{-1}\theta(x,t)) + \cO(\rme^{-\eta_0|x|} \theta(x,t)).$$
\end{lemma}


\subsection{Proof of Lemma~\ref{lem-nonlinearUp}}\label{sec-pertsystem}
In this subsection, we shall prove Lemma~\ref{lem-nonlinearUp}. First, we obtain the following simple lemma. 
\begin{lemma}\label{lem-newqCGL} Let $p(x,t)$ be a given smooth function so that the map $(x,t)\mapsto (\xi(x,t),\tau(x,t))$, defined by $\xi= x+p(x,t)$ and $\tau = t$, is invertible. If $A(x,t)$ solves the qCGL equation \eqref{eqs-qCGL}, then the function 
$$B(x,t): = A(x+p(x,t),t)$$ solves the modified qCGL equation
\begin{equation}\label{qCGL-newB}\partial_t B= (1+i\alpha)\partial_{x}^2B +   B - (1+i\beta)B|B|^2  + \gamma B |B|^4 + \mathcal{T}(p,B),\end{equation}
where $\gamma = \gamma_1 + \rmi \gamma_2$ and
\begin{equation}\label{def-Tresidual}\mathcal{T}(p,B) :=  \frac{p_t }{1+p_x}B_x - \frac{p_{xx}}{(1+p_x)^3} (1+i\alpha) B_x - \frac{p_x(2+p_x)}{(1+p_x)^2} (1+i\alpha) B_{xx}.\end{equation}
\end{lemma}

\begin{proof} Write $B(x,t) = A(\xi(x,t), t)$. Then
\[
B_t = A_\xi \dot p + A_t, \qquad B_x = A_\xi(1+p_x), \qquad B_{xx} = A_{\xi\xi}(1+p_x)^2 + A_\xi p_{xx},
\]
which implies
\[
A_t = B_t - \frac{\dot p B_x}{1+p_x}, \qquad A_\xi = \frac{B_x}{1+p_x}, \qquad A_{\xi\xi} = \frac{B_{xx}}{(1+p_x)^2} - \frac{p_{xx}B_x}{(1+p_x)^3}.
\]
Inserting these expressions into the equation
\[
A_t = (1+\rmi \alpha)A_{\xi\xi} + A - (1+i\beta)A|A|^2  + \gamma A |A|^4
\]
yields the result. \end{proof}

Next, we write the solution to the new qCGL equation  \eqref{qCGL-newB} in the amplitude and phase variables
$$B(x,t) =[r(x) + R(x, t)]e^{i(\varphi(x) + \phi(x, t))}\rme^{-i\omega_0 t},
$$
with $(R,\phi)$ denoting the perturbation variables. As in Section~\ref{sec-linearization}, but now keeping all nonlinear terms, we collect the real and imaginary parts of the equations for $R$ and $\phi$ to find
$$
\begin{aligned}
R_t &= R_{xx} -2\alpha\varphi_x R_x + (1 - 3r^2 - \alpha\varphi_{xx} - \varphi_x^2 + 5 \gamma_1 r^4)R
-\alpha r \phi_{xx} - 2(\alpha r_x + r\varphi_x)\phi_x \\
&\quad  + T_R(p) + Q_R(R,\phi,p) + N_R(R,\phi,p) ,\\
r\phi_t &= r\phi_{xx} + (-2\alpha r \varphi_x + 2r_x) \phi_x + \alpha R_{xx} + 2\varphi_x R_x    + (\omega_0 + \varphi_{xx} - \alpha\varphi_x^2-3\beta r^2 + 5 \gamma_2 r^4) R 
\\&\quad  + T_\phi(p) +  Q_\phi(R,\phi,p)  + N_\phi(R,\phi,p).
\end{aligned}
$$
The function $T_j(p)$, $j = R, \phi$, denotes the terms resulting from $\mathcal{T}(p, B)$ that are linear in $p$. The function $Q_j(R,\phi,p)$ collects terms that are quadratic in $(R,\phi,p)$, and $N_j(R,\phi,p) $ denotes the remaining terms. We now calculate these functions in detail. 

First, note that \eqref{def-Tresidual} can be written
$$\mathcal{T}(p,B) = p_t (1-p_x) B_x -  (1+i\alpha ) p_{xx}B_x  - (1+i\alpha ) p_x (2- 3p_x)B_{xx}+ \cO(p_x^3+p_x^2p_t +p_x p_{xx}),
$$
as long as $|p_x|<1$. Thus, by collecting the linear terms in the real and imaginary parts of $\mathcal{T}(p,B)$, we obtain \eqref{def-Plin}. Note that we do not include any terms that are of the form $p_x e^{-\eta_0|x|}$, as these are higher order and appear in $N(R, \phi, p)$. Similarly, we also obtain \eqref{def-Q} and the bounds \eqref{good-Re} for the remainder $\mathcal{N}(R,\phi,p)$.
This completes the proof of Lemma~\ref{lem-nonlinearUp}.


\subsection{Proof of Proposition~\ref{prop-Uapprox}}\label{sec-Asymptotic}
We next turn to the construction of the approximate solution to (\ref{eqs-nonUp}) and prove Proposition~\ref{prop-Uapprox}. For simplicity, let us first consider the case where $\delta^\pm$ are constants. Define
\begin{equation}\label{eqs-approx}
\mathcal{F}(U): =  (\partial_t - \mathcal{L}^a)U  - \mathcal{T}(p) -  \mathcal{Q}(R ,\phi,p) ,
\end{equation}
where $\mathcal{L}^a$ is the approximate linear operator defined by 
$$\mathcal{L}^a = D_2 \partial_x^2 - 2 \varphi_x D_1 \partial_x  + D_0^\varphi(x) + D_0^\infty, \qquad D_0^\infty: = -2 r_0^2 \begin{pmatrix} 1 - 2 \gamma_1 r_0^2 &  0\\
\beta  -2\gamma_2 r_0^2 & 0\end{pmatrix} . $$
See \eqref{redef-L}. By Lemma~\ref{lem-NBexistence}, it follows that 
\begin{equation}\label{DL-exp} (|\mathcal{L}- \mathcal{L}^a)U| = \cO(e^{-\eta_0|x|} |R|).\end{equation}
Our goal in this section is to construct $U^a$ so that $\mathcal{F}(U^a)  = 0$ up to the good error term as claimed in Proposition~\ref{prop-Uapprox}.  

The zero-order term $D_0^\varphi$ is asymptotically zero. However, we keep $D_0^\varphi(x)$ in the definition of $\mathcal{L}^a$ precisely because it may be viewed as part of the  first- and second-order derivative terms when acting on $(R,r\phi)$. To see this, define 
$$\mathcal{L}^a_1 = D_2 \partial_x^2 - 2 \varphi_x D_1 \partial_x  + D_0^\varphi(x) ,$$
and let $Z$ to be the diagonal matrix $\mathrm{diag}(1,r)$ so that $(R,r\phi) = Z(R,\phi)$. A direct calculation shows that
\[
\mathcal{L}_1^a \begin{pmatrix} R\\r\phi\end{pmatrix} = \Big(D_2 Z \partial_x^2 - 2 \varphi_x D_1Z  \partial_x \Big)\begin{pmatrix}R\\\phi\end{pmatrix} + \begin{pmatrix} 0 & -2\alpha r_x \\ 0 & 2r_x \end{pmatrix} \partial_x \begin{pmatrix}R\\\phi\end{pmatrix}.
\]
Thus, we see that the right hand side consists of no zero-order term. This cancellation is crucial to our analysis as it implies that the zero-order term in $\mathcal{L}^a$ involves $R$ only, which leads to the faster decay of $R$,  relative to $\phi$. 

To extract the governing (leading order) equations in the system \eqref{eqs-approx}, we first diagonalize the zero order constant matrix $D_0^\infty$. We can write 
$$ D_0^\infty = S \begin{pmatrix} -2 r_0^2  (1 - 2\gamma_1 r_0^2) & 0 \\ 0 & 0 \end{pmatrix} S^{-1}, \qquad \text{with} \qquad S: = \begin{pmatrix} 1 & 0\\ \beta_* & -1 \end{pmatrix}$$,
where we recall that $\beta_* = \frac{\beta - 2 \gamma_2 r_0^2}{1-2\gamma_1 r^2_0}$. Define the change of variables $\hat U = (\hat R, r\hat \phi) = S^{-1} U$, and set
\begin{equation}\label{eqs-hatU} 
\hat{\mathcal{F}} (\hat U): =  \Big(\partial_t - S^{-1} \mathcal{L}_1^a S - S^{-1}D_0^\infty S\Big) \hat U - S^{-1}\Big(\mathcal{T}(p) +  \mathcal{Q}(R ,\phi,p)\Big),
\end{equation}
where $(R,\phi) = Z^{-1} S \hat U$. It then follows that $$\hat{\mathcal{F}}(\hat U) = S^{-1} \mathcal{F}(U),$$ with $\mathcal{F}(\cdot)$ defined as in \eqref{eqs-approx}. 
We note that due to the structure of $S^{-1}D_0^\infty S$, the zero order terms in $\hat{\mathcal{F}}(\hat U)$ appear only in the first component. Using this structural property of \eqref{eqs-hatU}, we will choose the Ansatz for $\hat R$ so that the zero order term in the $\hat R$ equation cancels with the leading order terms involving $\hat \phi$ and $p$, that come from the functions $\mathcal{T}$ and $\mathcal{Q}$. 

To compute the terms in $\hat{F}(\hat U)$, notice that
\begin{eqnarray*}
\mathcal{L}_1^a S\hat U &=& \mathcal{L}_1^a S Z \begin{pmatrix} \hat R \\ \hat \phi \end{pmatrix} \\
&=& [D_2 \partial_x^2 - 2 \varphi_x D_1 \partial_x + D_0^\varphi] \begin{pmatrix} 1 & 0 \\ \beta^* & -r \end{pmatrix} \begin{pmatrix} \hat R \\ \hat \phi \end{pmatrix} \\
&=&  \left[ \begin{pmatrix} 1-\alpha \beta^* & r\alpha \\ \alpha + \beta^* & -r \end{pmatrix} \partial_x^2 + \begin{pmatrix} -2\varphi_x(\alpha+\beta^*) & 2\alpha r_x + 2 \varphi_x r \\ 2\varphi_x(1-\alpha\beta^*) & -2r_x + 2 \alpha \varphi_x r \end{pmatrix}\partial_x  + \begin{pmatrix} \frac{\beta^*}{r}(2\varphi_xr_x + \alpha r_{xx}) & 0 \\ \frac{\beta^*}{r}(2\alpha\varphi_xr_x - r_{xx}) & 0 \end{pmatrix}\right]\begin{pmatrix} \hat R \\ \hat \phi \end{pmatrix} 
\end{eqnarray*}
and so
\begin{eqnarray*}
S^{-1}\mathcal{L}_1^a S\hat U &=& \begin{pmatrix} (1 - \alpha\beta_*)\partial_x^2 - 2(\alpha+\beta_*)\varphi_x \partial_x & \alpha  r\partial_x^2 + 2 (r \varphi_x + \alpha r_x) \partial_x\\
-(1+\beta_*^2) (\alpha \partial_x^2 + 2\varphi_x \partial_x) & (1+\alpha \beta_*)r\partial_x^2 + 2[ (\beta^*-\alpha)r \varphi_x + (1 + \alpha \beta^*)r_x] \partial_x 
\end{pmatrix} \begin{pmatrix} \hat R\\\hat \phi\end{pmatrix} \\
&& \quad + \begin{pmatrix} \frac{\beta^*}{r}(2\varphi_xr_x + \alpha r_{xx}) & 0 \\ \frac{(\beta^*)^2}{r}(2\varphi_xr_x + \alpha r_{xx}) + \frac{\beta^*}{r}(2\alpha\varphi_xr_x - r_{xx}) & 0 \end{pmatrix} \begin{pmatrix} \hat R\\\hat \phi\end{pmatrix}. 
\end{eqnarray*}
Note that the zero-order terms in the $\hat R$ and $\hat \phi$ equation are exponentially localized in space, and so they will not contribute to the leading order dynamics. 
Similarly, we have 
$$S^{-1} \mathcal{T}(p) = p_t \begin{pmatrix} r_x \\ \beta_* r_x - r \varphi_x\end{pmatrix} + r \varphi_x p_{xx} \begin{pmatrix} \alpha \\ 1 + \alpha \beta_*\end{pmatrix} - 2r \varphi_x^2 p_x \begin{pmatrix} -1 \\ \alpha - \beta_* \end{pmatrix}$$
and
\begin{eqnarray*}
S^{-1} \mathcal{Q}(R ,\phi,p) &=&  \begin{pmatrix}   2\varphi_x\hat \phi_x \hat R - r\hat \phi_x^2  - 3r\hat R^2 + 10 \gamma_1 r^3 \hat R^2 - 3 r \varphi_x^2 p_x^2 - 4 r \varphi_x p_x \hat \phi_x + 2 \varphi_x^2 \hat R p_x 
\\ (\alpha-\beta_*)r(\hat \phi_x + \varphi_x p_x)^2 + \frac{4r_0^3(\gamma_1 \beta - \gamma_2)}{1 - 2 \gamma_1 r_0^2}\hat R^2
\end{pmatrix}.
\end{eqnarray*}

Next, let us denote by $\Pi_R$ and $\Pi_\phi$ the projection on the first and second component of a vector in $\mathbb{R}^2$, respectively. First consider the $\hat{R}$ equation, by taking the projection $\Pi_R$ of \eqref{eqs-hatU}. We wish to chose and approximate solution so that all leading orders terms in \eqref{eqs-hatU} cancel. First consider the terms that are linear in $\hat R,\hat \phi, p$:
$$
\begin{aligned}
-2(r \varphi_x + \alpha r_x) \hat \phi_x  + 2 r_0^2  (1 - 2\gamma_1 r_0^2) \hat R  = 2r \varphi_x^2 p_x. 
 \end{aligned}
 $$
Recall that we are primarily concerned with the behavior at $\pm \infty$. Since $r \varphi_x \to r_0$ as $x \to \pm\infty$, we therefore define
\begin{equation}\label{def-hatR0} \hat R ^a_0: = \frac{k_0}{ r_0 (1 - 2\gamma_1 r_0^2) } \Big( \hat \phi^a_x + \varphi_x p_x\Big),\end{equation}
for some $\hat \phi^a$ to be determined later. It is then clear that the function $\hat U^a_0  = (\hat R_0^a,r\hat \phi^a)$ with $\hat R_0^a$ defined as in \eqref{def-hatR0} approximately solves the first equation \eqref{eqs-hatU} in the following sense:
\begin{equation}\label{1st-approxUhat1} 
\Pi_R \hat{\mathcal{F}}(\hat U_0^a) = \cO(e^{-\eta_0 |x|} (|\hat \phi^a_x| + |p_x|))+ \cO\Big(|\hat \phi_{xx}^a| + |\hat \phi_x^a|^2 + |p_{xx}| + |p_x|^2 \Big).\end{equation}
This is not quite good enough, due to the terms that are quadratic in $\hat \phi_x^a$ -- recall that we need only the good terms appearing in \eqref{def-Good}. To fix this, we revise the Ansatz \eqref{def-hatR0} by introducing 
\begin{equation}\label{def-hatRa-re} 
\hat R ^a: = \frac{k_0}{ r_0 (1 - 2\gamma_1 r_0^2) } \Big( \hat \phi^a_x + \varphi_x p_x\Big) +\hat R^a_1 = \hat R^a_0 + \hat R^a_1,
\end{equation}
with $\hat R^a_1$ to be chosen so that it takes care of all quadratic nonlinear or next order decaying terms in the error that appear on the right of \eqref{1st-approxUhat1}. Precisely, we take
\begin{eqnarray}
2r_0^2(1-2\gamma_1 r_0^2)\hat R_1^a &=& - (\hat R_0^a)_t - 2 (\alpha + \beta_*)\varphi_x (\hat R_0^a)_x + \alpha r (\hat \phi_{xx} + \varphi_x p_{xx}) + 2 \varphi_x (\hat\phi_x + \varphi_x p_x) \hat R_0^a - r \hat \phi_x^2 \nonumber \\
&&\qquad -3r \varphi_x^2 p_x^2 -4r \varphi_x p_x \hat \phi_x - (3r - 10 \gamma_1 r^3)(\hat R_0^a)^2,
\label{E:defR1}
\end{eqnarray}
Here the $t$-derivative term appearing on the right can be replaced by its spatial derivatives thanks to the estimate \eqref{Dt-phi}. Thus, if we define $\hat U^a = (\hat R^a,r\hat \phi^a)$, then $\hat U^a$ approximately solves the first equation in \eqref{eqs-hatU} with a good error as claimed:
\begin{equation}\label{Good-UR}  
\Pi_R \hat{\mathcal{F}}(\hat U ^a) = \cO(\Upsilon (\hat \phi^a_x,p_x)),
\end{equation} 
where $\Upsilon$ is defined in \eqref{def-Good}. 

It remains to show that 
\begin{equation}\label{Good-Uphi} \Pi_\phi \hat{\mathcal{F}}(\hat U ^a) = \cO(\Upsilon (\hat \phi^a_x,p_x)).\end{equation}
Observe that since there is no zero order term $\hat R$ appearing in $\Pi_\phi\hat{\mathcal{F}}(\hat U^a)$, the contribution of $\hat R^a_1$ from \eqref{def-hatRa-re} can be absorbed into the good error term. For this reason, we shall check the contribution of the first term in $\hat R^a$, or precisely $\hat R^a_0$ defined as in \eqref{def-hatR0}. By plugging the Ansatz for $\hat R^a$ into $\hat{\mathcal{F}}(\hat U^a)$ and using \eqref{Dt-phi}, it follows by direct calculations that
\begin{equation}\label{Burgers}\begin{aligned}
\Pi_\phi \hat{\mathcal{F}}(\hat U ^a)  &=  r\left[ L_B \hat\phi^a + \varphi_x L_B  p  -  q\Big(\hat \phi^a_x+\varphi_x p_x\Big)^2\right]  + \cO(\Upsilon(\hat \phi^a_x,p_x))
 \end{aligned}\end{equation}
where we recall that $L_B$ is defined by
$$L_B = \partial_t + 2(\alpha - \beta_*) \varphi_x \partial_x- \mathrm{d} \partial_x^2 .$$ 
By Remark~\ref{rem-Burgerschoice}, we have $L_B (\hat\phi^a \pm k_0p) = q(\hat \phi_x^a \pm k_0p_x)^2 $
for each $+/-$ case. This, together with the fact that $\varphi_x \to \pm k_0$ as $|x|\to \infty$, shows that the leading term on the right hand side of \eqref{Burgers} vanishes. We thus obtain \eqref{Good-Uphi} as claimed. 

Going back to the original variable $U^a = S \hat U^a$, we find that
$$ \mathcal{F}(U^a) =   (\partial_t - \mathcal{L}^a)U^a  - \mathcal{T}(p) -  \mathcal{Q}(R ^a,\phi^a,p)  = \cO(\Upsilon(\hat \phi^a_x, p_x)).$$
This together with \eqref{DL-exp} establishes Proposition~\ref{prop-Uapprox} in the case that $\delta^\pm$ are constants. The case when $\delta^\pm(t)$ are not constants follows similarly. 


\subsection{Proof of Proposition~\ref{prop-corUapp}}\label{sec-proofProp52}
We now choose $\mathcal{B}(x,t)$ to be $e(x,t+1)$, the difference of error functions defined in \eqref{def-epm}. We then have \begin{eqnarray*}
\Big|\mathcal{B} (1-\mathcal{B})\Big |(x,t) &\leq & C \left(\rme^{-\frac{(x + c_\rmg t)^2}{8\mathrm{d}(t+1)}} +\rme^{-\frac{(x - c_\rmg t)^2}{8\mathrm{d}(t+1)}} \right) \le C (1+t)^{1/2} \theta(x,t),
\end{eqnarray*}
and so 
\begin{eqnarray}\label{est-dB}
\left[ \frac{\mathcal{B}}{1+\delta^\pm\mathcal{B}}  - \frac{\mathcal{B}}{1+\delta^\pm}  \right]  (x,t) &\leq & C|\delta^\pm|\left(\rme^{-\frac{(x + c_\mathrm{g}t)^2}{8\rmd(t+1)}} +\rme^{-\frac{(x - c_\mathrm{g}t)^2}{8\rmd(t+1)}} \right) \le C (1+t)^{1/2} \theta(x,t)|\delta^\pm|.
\end{eqnarray}
We also have
\begin{eqnarray*}
\Big|\frac{\mathcal{B}_x}{1+\delta^\pm \mathcal{B}}\Big|(x,t) &\leq & C(1+t)^{-1/2} \left(\rme^{-\frac{(x + c_\mathrm{g}t)^2}{8\rmd(t+1)}} +\rme^{-\frac{(x - c_\mathrm{g}t)^2}{8\rmd(t+1)}} \right)  \le C \theta(x,t).
\end{eqnarray*}
Thus,
$$ |\partial_x^k \hat \phi^a|  \le C (1+t)^{1/2-k/2} \theta(x,t) (|\delta^+ | + |\delta^- |), \qquad |\partial_x^k p|  \le C (1+t)^{1/2-k/2} \theta(x,t) (|\delta^+ | + |\delta^- |),$$
for $k\ge 1$. Here these bounds are valid uniformly in $x\in\mathbb{R}$ and $t\geq0$, provided that $\delta^\pm$ is sufficiently small, with a bound on $\delta^\pm$ that depends only on the maximum of $|\mathcal{B}|$. By using these bounds, it follows immediately that 
\begin{equation}\label{est-Gapp}
\begin{aligned}
\Upsilon(\hat \phi^a_x,p_x)  & \le C (|\delta^+ | + |\delta^- |)e^{-\eta_0 |x|}  \theta(x,t)  + C(1+t)^{-1}\theta(x,t) (|\delta^+|^2 + |\delta^-|^2).
\end{aligned}
\end{equation}

\begin{proof}[Proof of Proposition~\ref{prop-corUapp}] The estimate for ${\mathcal{R}}^{app}_1(x,t, \delta^\pm) $ follows directly from \eqref{eqs-Uapp} and \eqref{est-Gapp}. It remains to give details of $\Sigma^\pm(x,t,\delta^\pm)$ introduced in Proposition~\ref{prop-Uapprox}. 
By \eqref{Burgers-solns}, we get
$$\begin{aligned}
  \frac{\partial p}{\partial \delta^+}(x,t,\delta^\pm(t)) & = \frac{\mathrm{d}}{2qk_0} \frac{\mathcal{B}}{1 + \delta^+ \mathcal{B}}, \qquad   \frac{\partial p}{\partial \delta^-}(x,t,\delta^\pm(t))  = -  \frac{\mathrm{d}}{2qk_0} \frac{\mathcal{B}}{1 + \delta^- \mathcal{B}} \\
    \frac{\partial \hat \phi^a}{\partial \delta^+}(x,t,\delta^\pm(t)) & =  \frac{\mathrm{d}}{2q} \frac{\mathcal{B}}{1 + \delta^+ \mathcal{B}} , \qquad   \frac{\partial \hat\phi^a}{\partial \delta^-}(x,t,\delta^\pm(t))  =  \frac{\mathrm{d}}{2q} \frac{\mathcal{B}}{1 + \delta^- \mathcal{B}} .
   \end{aligned}$$
By definition \eqref{def-hatRa-re}, we obtain a similar result for $\frac{\partial }{\partial \delta^\pm} \hat R^a(x,t,\delta^\pm(t)) $. This together with \eqref{est-dB} yields $$\begin{aligned}
 \frac{\partial p}{\partial \delta^\pm}(x,t,\delta^\pm(t)) & = \pm  \frac{\mathrm{d}}{2qk_0}\frac{\mathcal{B}}{1 + \delta^\pm } + \cO(|\delta^\pm|) (1+t)^{1/2} \theta(x,t)\\
    \frac{\partial \hat \phi^a}{\partial \delta^\pm}(x,t,\delta^\pm(t)) & =  \frac{\mathrm{d}}{2q}\frac{\mathcal{B}}{1 + \delta^\pm}  + \cO(|\delta^\pm|) (1+t)^{1/2} \theta(x,t),
   \end{aligned}$$
and 
$$\begin{aligned}
 \frac{\partial \hat R^a}{\partial \delta^\pm}(x,t,\delta^\pm(t)) & =  \frac{\mathrm{d}}{2qr_0(1-2\gamma_1r_0^2)} \frac{\mathcal{B}_x}{1 + \delta^\pm } (k_0 \pm \varphi_x) + \cO(|\delta^\pm|) (1+t)^{1/2} \theta(x,t) .   \end{aligned}$$
Putting these estimates into $\Sigma^\pm(x,t,\delta^\pm)$ (see \eqref{def-Sgpm}), we obtain $$\begin{aligned}
  \Sigma^\pm(x,t) &=  \frac{\mathrm{d}}{2q(1+ \delta^\pm)} \Big[ \begin{pmatrix}\frac{\varphi_x \pm k_0}{r_0(1-2\gamma_1r_0^2)} \mathcal{B}_x  \\  r\mathcal{B}
\end{pmatrix} \mp \frac{\mathcal{B}}{k_0 } \begin{pmatrix}r_x\\ r \varphi_x 
\end{pmatrix}\Big]+ \cO(|\delta^\pm|) (1+t)^{1/2} \theta(x,t) .
\end{aligned}
$$
Here the error term $\cO(|\delta^\pm|) (1+t)^{1/2} \theta(x,t) $ times $\dot \delta^\pm$ can be absorbed into the last term in the remainder $\mathcal{R}^{app}_1 (x,t,\delta^\pm) $. Rearranging terms, we obtain the proposition as claimed. 
\end{proof}

\begin{proof}[Proof of Lemma~\ref{lem-linEest}] Note that
\[
\mathcal{E}^\pm =  \mathcal{B} V_2 \mp \frac{\mathcal{B}}{k_0} V_1 + \mathcal{O}(\theta),
\]
where $V_{1,2}$ are the eigenfunctions of $\mathcal{L}$ defined in \eqref{def-V12}. A direct calculation shows that
\[
(\partial_t - \mathcal{L}) \mathcal{E}^\pm = \mathcal{O}(e^{-\eta|x|}\mathcal{B}_x) + \mathcal{O}((\partial_t - \mathcal{L}) \theta).
\]
The final term on the right hand side is of the claimed order because the Gaussian terms each satisfy the asymptotic equation at one end ($\pm \infty$), where the maximum $x = \pm c_\mathrm{g} t$ is located, and hence decay faster at the other end, away from this maximum.
\end{proof}


\section{Stability analysis}\label{sec-stability}

In this section, we show how the Ansatz \eqref{def-BB} and the approximate solution $(\phi^a, R^a, p)$, constructed in \S\ref{S:app-soln}, can be used to prove Theorem~\ref{theo-main}.

\subsection{Nonlinear perturbed equations}
Recall, that the Ansatz \eqref{def-BB} is given by
\[
B(x,t): =A(x+p(x,t),t) = [r(x) + R(x, t)]e^{i(\varphi(x) + \phi(x,t)}\rme^{-i\omega_0 t},
\]
where $A(x,t)$ solves the qCGL equation \eqref{eqs-qCGL}. The functions $(R, \phi)$ represents the perturbation, and we define $(\tilde R, \tilde \phi)$ via
\begin{equation}\label{def-Rphi} 
\begin{pmatrix} R\\\phi \end{pmatrix}(x,t) = \begin{pmatrix} R^a\\\phi^a \end{pmatrix} (x,t,\delta^\pm(t))+ \begin{pmatrix} \tilde R\\ \tilde \phi \end{pmatrix} (x,t),
\end{equation}
where $U^a = (R^a,r\phi^a)$ is the approximate solution constructed in Proposition~\ref{prop-corUapp}. 
Note that the functions $\delta^\pm(t)$ are still arbitrary smooth functions. They will be chosen, in equation \eqref{def-intdelta12} below, so that they cancel non-decaying terms in the nonlinear iteration scheme. The following lemma is then a direct combination of Lemma~\ref{lem-nonlinearUp} and Proposition~\ref{prop-corUapp}. 
\begin{Lemma}\label{lem-tUsystem}  Let $\delta^\pm = \delta^\pm(t)$ be arbitrary smooth functions. The perturbation function $\tilde U = (\tilde R,r\tilde \phi)$ defined through \eqref{def-Rphi} satisfies 
\begin{equation} \label{tUsystem}
\begin{aligned}
(\partial_t - \mathcal{L}) \tilde U  &= -\frac{\mathrm{d}}{2q} \sum_\pm \frac{\dot \delta^\pm(t)  }{(1+ \delta^\pm(t))} \mathcal{E}^\pm(x,t) +  \mathcal{N}_1(\tilde R,\tilde \phi, \delta^\pm )(x,t),\end{aligned}\end{equation}
where $\mathcal{E}^\pm(x,t)$ is defined as in \eqref{def-Epm} and the remainder $  \mathcal{N}_1(\tilde R,\tilde \phi, \delta^\pm )$ satisfies
\begin{equation}\label{bound-NN}\begin{aligned}
\mathcal{N}_1(\tilde R,\tilde \phi, \delta^\pm )(x,t)  & \le C (|\delta^+ | + |\delta^- |)e^{-\eta_0 |x|}  \theta(x,t)  + C(1+t)^{-1}\theta(x,t) (|\delta^+|^2 + |\delta^-|^2)
\\&\quad + C(|\delta^+|+|\delta^-|)(|\dot \delta^+|+|\dot \delta^-|)(1+t)^{1/2}\theta(x,t) 
\\& \quad + (|\delta^+(t)|+|\delta^-(t)|)\theta(x,t)L(\tilde R,\tilde \phi)   + Q(\tilde R, \tilde \phi) 
\end{aligned} \end{equation}
with  \begin{equation}\label{def-LQ}\begin{aligned}
L(\tilde R,\tilde \phi)  &= \cO( |\tilde R| + |\tilde R_x| +|\tilde  \phi_x|), \\
Q(\tilde R,\tilde \phi)  &= \cO(|\tilde R|^2 + |\tilde R_x|^2 +|\tilde  \phi_x|^2 + |\tilde R\tilde\phi_{xx}|+ |\tilde R\tilde R_{xx}|+|\tilde R \tilde \phi_t|).
\end{aligned}\end{equation}
\end{Lemma}

\begin{proof} By Lemma~\ref{lem-nonlinearUp}, the perturbation $U = (R,r\phi)$ solves 
\[
(\partial_t - \mathcal{L}) U -   \mathcal{T}(p) - \mathcal{Q}(R,\phi,p) = \mathcal{N}(R,\phi,p), 
\]
and, by Proposition~\ref{prop-corUapp}, the approximate solution $U^a = (R^a,r\phi^a)$ solves 
\[
 (\partial_t - \mathcal{L})U^a  - \mathcal{T}(p) - \mathcal{Q}(R^a,\phi^a,p) = \frac{\mathrm{d}}{2q}\sum_\pm \frac{\dot \delta^\pm(t)  }{ (1+ \delta^\pm(t))} \mathcal{E}^\pm(x,t)  +{\mathcal{R}}^{app}_1 (x,t, \delta^\pm).
\]
Here the same notations as in Lemma~\ref{lem-nonlinearUp} and Proposition~\ref{prop-corUapp} are used. It then follows that the difference $\tilde U = U - U^a$ solves \eqref{tUsystem} with the remainder $\mathcal{N}_1(\tilde R,\tilde \phi, \delta^\pm )(x,t)$ defined by 
$$\mathcal{N}_1(\tilde R,\tilde \phi, \delta^\pm )(x,t): =   \mathcal{N}(R,\phi,p) + \mathcal{Q}(R,\phi,p)  - \mathcal{Q}(R^a,\phi^a,p)  -  {\mathcal{R}}^{app}_1 (x,t, \delta^\pm) .
$$
The claimed estimate for $\mathcal{N}_1(\tilde R,\tilde \phi, \delta^\pm )$ now follows directly from \eqref{good-Re} and \eqref{bound-Rdelta}.  
\end{proof}


\subsection{Green function decomposition}
We recall the Green function decomposition of Proposition~\ref{prop-Green}:
$$\begin{aligned}
 \mathcal{G}(x,y,t) = e(x,t) V_1(x) \langle \psi_1(y),\cdot \rangle + e(x,t) V_2(x) \langle \psi_2(y),\cdot \rangle + \widehat{\mathcal{G}}(x,y,t), 
\end{aligned}
$$
where $e(x,t)$, defined in \eqref{def-epm}, is the sum of error functions and $\widehat{\mathcal{G}}(x,y,t)$ satisfies the Gaussian-like estimate
$$\widehat{\mathcal{G}} (x,y,t) = \sum_\pm\cO(t^{-\frac{1}{2}}\rme^{-\frac{(x-y\pm c_\mathrm{g}t)^2}{Mt}} ) +  \cO(\rme^{-\eta(|x-y|+t)}) .$$ 
Now observe that from the definitions of $V_j(x)$ in \eqref{def-V12} and $\mathcal{E}^\pm(x,t)$ in \eqref{def-Epm}, we have 
\begin{equation}\label{app-Epm} \mathcal{E}^\pm(x,t) = -\mathcal{B}(x,t) V_2
\mp \frac{\mathcal{B}(x,t)}{k_0}  V_1 + \cO(\theta(x,t)),
\end{equation}
where $\mathcal{B}(x,t) =  e(x,t+1)$. In addition, since $\psi_1(y)$ and $\psi_2(y)$ are exponentially localized and $e_x(x,t)$ and $\mathcal{B}_x^\pm(x,t)$ behave as Gaussians, we can write 
\begin{equation} e(x,t+1) V_1(x) \langle \psi_1(y),\cdot \rangle + e(x,t+1) V_2(x) \langle \psi_2(y),\cdot \rangle = \mathcal{E}^+(x,t)  \langle \Psi^+(y),\cdot \rangle + \mathcal{E}^-(x,t)  \langle \Psi^-(y),\cdot \rangle, \end{equation}
where
\[
\Psi^+ = - \frac{1}{2}\psi_1 - \frac{k_0}{2}\psi_2, \qquad \Psi^- = -\frac{1}{2}\psi_1 + \frac{k_0}{2}\psi_2,
\]
with an error of $ \cO(e^{-\eta|y|}(|\mathcal{B}_x | + |e_x|) = \cO(e^{-\eta|y|}\theta(x,t))$ that can easily be absorbed into the bound for $\widehat{\mathcal{G}}(x,y,t)$. In addition, we observe that the difference $e(x,t) - e(x,t+1)$ satisfies  
\begin{eqnarray*}
|e(x,t) - e(x,t+1)|  &=& \left|\int_{t}^{t+1} e_\tau(x,\tau)\rmd \tau\right| \le Ct^{-\frac{1}{2}}\rme^{-\frac{(x\pm c_\mathrm{g}t)^2}{8t}}, \qquad t\ge 1,
\end{eqnarray*}
and
$$|e(x,t) - e(x,t+1)| \leq C \int_{\frac{x^2}{t+1}}^{\frac{x^2}{t}} \rme^{-z^2} \rmd z \leq C \rme^{-x^2/2}\rme^{-\eta t}, \qquad  t\le 1.$$
These, together with the localization property of $\psi_j(y)$, show that the term 
$$(e(x,t) - e(x,t+1)) V_1(x) \langle \Psi_1(y),\cdot \rangle + (e(x,t) - e(x,t+1) )V_2(x) \langle \Psi_2(y),\cdot \rangle $$
can be absorbed into the bound for $\widehat{\mathcal{G}}(x,y,t)$. Therefore, let us decompose the the Green function as follows: 
\begin{equation}\label{Green-decompose}
 \mathcal{G}(x,y,t) = \mathcal{E}^+(x,t)  \langle \Psi^+(y),\cdot \rangle + \mathcal{E}^-(x,t)  \langle \Psi^-(y),\cdot \rangle + \widetilde{\mathcal{G}}(x,y,t), 
\end{equation}
where $\widetilde{\mathcal{G}}(x,y,t)$ satisfies the same bound as $\widehat{\mathcal{G}}(x,y,t)$. The motivation for this will become clear below.

We obtain the following simple lemma. 
\begin{Lemma}\label{lem-id-E12} For all $s,t\ge 0$
$$\begin{aligned}
\int_{\mathbb{R}}  \mathcal{G}(x,y,t-s) & \mathcal{E}^\pm(y,s) \; \rmd y 
=  \mathcal{E}^\pm(x,t) +  \int_s^t \int_{\mathbb{R}}  \mathcal{G}(x,y,t-\tau) \mathcal{R}^\pm(y,\tau)\; \rmd y\rmd \tau,
\end{aligned}$$
where 
\begin{equation}\label{bounds-SR}
|\mathcal{R}^\pm(y,s)| \leq C((1+s)^{-1} + \rme^{-\eta|y|})\theta(y,s).
\end{equation}
\end{Lemma}
\begin{proof} We recall from Lemma~\ref{lem-linEest} that
$$(\partial_t - \mathcal{L})  \mathcal{E}^\pm (x,t)= \mathcal{O}((1+t)^{-1} + \rme^{-\eta |x|}) \theta(x,t).$$
By applying the Duhamel principle to the solution of this equation with ``initial'' data at $t=s$, we easily obtain the lemma.
\end{proof}


\subsection{Initial data for the asymptotic Ansatz}

Let us denote $U^a_0(x): = U^a(x,0,\delta^\pm(0)), \tilde U_0(x) = \tilde U(x,0),$ $p_0(x): = p(x,0,\delta^\pm(0))$, and 
$$U_\mathrm{source}(x): = \begin{pmatrix} r(x) \\ r(x)\varphi(x) \end{pmatrix} , \qquad U_{\mathrm{source},p}(x): = \begin{pmatrix} r(x+p_0(x)) \\ r(x)\varphi(x+p_0(x)) \end{pmatrix} .$$ 
By \eqref{in-norm}, \eqref{def-Rphi} and the assumption on the initial data $A_\mathrm{in}(x)$, we have 
\begin{equation}\label{in-assumption} \Big \| e^{x^2/M_0}\Big( U^a_0(\cdot) + \tilde U_0(\cdot) + U_\mathrm{source}(\cdot) - U_{\mathrm{source},p}(\cdot)  \Big) \Big \|_{C^3(\mathbb{R})}\le \epsilon,
\end{equation}
for $\epsilon = \| A_\mathrm{in}(\cdot) - A_\mathrm{source}(\cdot)\|_\mathrm{in}$. 
We prove the following:
\begin{Lemma}\label{lem-initialdata} There are small constants $\delta_0^\pm$ so that if we take $\delta^\pm (0) = \delta^\pm_0$ in the definition of $U^a_0(x)$ and $p_0(x)$ then 
\begin{equation}\label{in-choice} \int_\mathbb{R} \langle \Psi^\pm(y),\tilde U_0(y)\rangle  \rmd y   =0\end{equation}
for each $+/-$ case, with $\Psi^\pm(\cdot)$ the same as in \eqref{Green-decompose}. In addition, $ \sum_\pm |\delta^\pm_0| \le C\epsilon$, 
and 
\begin{equation}\label{tU0-decay} \| e^{x^2/M_0} \tilde U_0(\cdot)  \|_{C^3(\mathbb{R})} \le C \epsilon,
\end{equation}
for some positive constant $C$.
\end{Lemma}
\begin{proof} Let us denote 
$$
F_0(x) := U^a_0(x) + \tilde U_0(x) + U_\mathrm{source}(x) - U_{\mathrm{source},p}(x),
$$
so that $\| e^{x^2/M_0} F_0(\cdot)  \|_{C^3(\mathbb{R})}\le \epsilon$ by \eqref{in-assumption}. Note that $\Psi^\pm(y)$ are linear combinations of $\Psi_1(y),\Psi_2(y)$, which are in the nullspace of the adjoint of $\mathcal{L}$. Therefore, \eqref{in-choice} will follow if we can find $\delta_0^\pm$ so that 
\begin{equation}\label{in-est} G_j(\delta_0^\pm ): =  \int_\mathbb{R} \langle \psi_j(y), U^a_0(y) +U_\mathrm{source}(y) - U_{\mathrm{source},p}(y) \rangle  \rmd y =  \int_\mathbb{R} \langle \psi_j(y), F_0(y)\rangle  \rmd y , \end{equation}
for $j = 1,2$. Notice that $ G_j(0) = 0$ since $p_0$ and $U^a_0$ vanish when $\delta_0^\pm =0$. Thus, we can define a new function 
\[
\tilde{G}(F, \delta_0^\pm) = (G_1(\delta_0^\pm),G_2(\delta_0^\pm)) - \left( \int_\mathbb{R} \langle \psi_1(y), F_0(y)\rangle  \rmd y, \int_\mathbb{R} \langle \psi_2(y), F_0(y)\rangle  \rmd y\right).
\]
Since $\tilde{G}(0,0) = 0$, it suffices to show that the Jacobian determinant $J_G$ of $(G_1(\delta_0^\pm),G_2(\delta_0^\pm))$ with respect to $\delta_0^\pm$ is nonzero at $(\delta_0^+,\delta_0^-) = 0$. For this, we compute
$$\begin{aligned}
 \frac{\partial G_j(\delta_0^\pm ) }{\partial \delta_0^\pm } {\vert_{(\delta_0^+, \delta_0^-) =0}}  &=   \Big \langle \psi_j(\cdot),  \frac{\partial U^a_0(\cdot)}{\partial \delta_0^\pm } - \frac{\partial p_0(\cdot)}{\partial \delta_0^\pm } \begin{pmatrix} r_y \\ r \varphi_y\end{pmatrix}  \Big \rangle_{L^2}
\\
& =  \Big \langle \psi_j(\cdot), \frac {d}{2q} \begin{pmatrix} \mathcal{B}_y \\ \beta^* \mathcal{B}_y - r \mathcal{B}\end{pmatrix}(\cdot,0) \mp  \frac{d}{2qk_0}  \begin{pmatrix} r_y \\ r \varphi_y \end{pmatrix} \mathcal{B}(\cdot,0)\Big \rangle_{L^2}
\end{aligned} 
$$
where we have used the definition of $p_0$ and $U_a^0$; see \eqref{Burgers-solns}. For our convenience, let us denote
$$\tilde V_1 = \begin{pmatrix} r_y \\ r \varphi_y \end{pmatrix}  \mathcal{B}(y,0), \qquad \tilde V_2 = \begin{pmatrix}\mathcal{B}_y \\ \beta^* \mathcal{B}_y - r \mathcal{B}\end{pmatrix}(\cdot,0) .$$
It then follows that 
$$ \begin{aligned}
J_G &= \frac{d^2 }{2 q^2 k_0} \Big[ \langle \psi_1, \tilde V_1\rangle_{L^2}  \langle \psi_2, \tilde V_2\rangle_{L^2} -  \langle \psi_1, \tilde V_2\rangle_{L^2} \langle \psi_2, \tilde V_1\rangle_{L^2}\Big]. 
\end{aligned}$$
To show that $J_G$ is nonzero, we recall that $\mathcal{B}(x,0) = e(x,1)$, the plateau function; see \eqref{choice-Be} and Figure~\ref{fig-plateau}. By redefining $\mathcal{B}(x,0)$ to be $e(x,k)$ for $k$ large, if necessary, we observe that $\tilde V_j  \approx V_j $, $j = 1,2$, for bounded $x$, with $V_j$ defined as in Lemma~\ref{lem-Lspectrum}. This, together with the fact that the adjoint functions $\psi_j$ are localized, shows that 
$$ J_G \approx  \frac{d^2 }{2 q^2 k_0} \Big[ \langle \psi_1, V_1\rangle_{L^2}  \langle \psi_2, V_2\rangle_{L^2} -  \langle \psi_1, V_2\rangle_{L^2} \langle \psi_2,  V_1\rangle_{L^2}\Big]  =  \frac{d^2 }{2 q^2 k_0}  \det M^\psi(0),$$
where the matrix $M^\psi(0)$ is defined as in \eqref{E:def_mij} and is invertible. This proves that $J_G \not = 0$. The existence of small $\delta^\pm_0 = \delta^\pm_0(F_0)$ satisfying \eqref{in-est} then follows from the standard Implicit Function Theorem.

Finally, the estimate \eqref{tU0-decay} follows directly from the fact that the difference of the two error functions $ \mathcal{B}(x,0) = e(x,1)$ decays to zero as fast as $\rme^{-x^2}$ when  $x\to \pm\infty$. 
\end{proof}


\subsection{Integral representations}

Using Lemmas~\ref{lem-Greenfn} and~\ref{lem-tUsystem}, we obtain the integral formulation of the solution $\tilde U = (\tilde R, r\tilde \phi)$ of \eqref{tUsystem}: 
\begin{equation}\label{int-formtR01}
\begin{aligned}
\tilde U (x,t)&=  \int_\mathbb{R} \mathcal{G} (x,y,t) \tilde U_0(y) \rmd y  
+  \int_0^t \int_\mathbb{R}  \mathcal{G} (x,y,t-s)  \mathcal{N}_1(\tilde R,\tilde \phi_y, \delta^\pm )(y,s) \rmd y \rmd s
\\&\quad - \frac{\mathrm{d}}{2q}\int_0^t \int_{\mathbb{R}}  \mathcal{G} (x,y,t-s) \Big[ \frac{\dot \delta^+(s)  }{1+\delta^+(s)} \mathcal{E}^+(y,s) + \frac{\dot \delta^-(s)  }{1+ \delta^-(s)}  \mathcal{E}^-(y,s)  \Big] \; \rmd y \rmd s. 
\end{aligned}\end{equation}
For simplicity, let us denote 
\begin{equation}\label{def-Theta} \Theta^\pm(t) : = \frac{\mathrm{d}}{2q}  \Big[ \log(1+ \delta^\pm(t)) - \log(1+ \delta^\pm_0)\Big] \end{equation} with $\delta_0^\pm$ defined as in Lemma~\ref{lem-initialdata}. 
Using Lemma~\ref{lem-id-E12}, the last integral term in \eqref{int-formtR01} is equal to 
 $$\begin{aligned}
  -\sum_\pm\int_0^t &\frac{d}{ds} \Theta^\pm(s) \Big[  \mathcal{E}^\pm(x,t) +  \int_s^t \int_{\mathbb{R}}  \mathcal{G}(x,y,t-\tau) \mathcal{R}^\pm(y,\tau)\; \rmd y\rmd \tau\Big]\rmd s
 \\ &=  -\sum_\pm \Theta^\pm(t) \mathcal{E}^\pm(x,t)  - \sum_\pm \int_0^t  \int_{\mathbb{R}} \mathcal{G} (x,y,t-s) \Theta^\pm(s)  \mathcal{R}^\pm(y,s)\; \rmd y\rmd s.
\end{aligned}$$
Here in the last equality integration by parts in $s$ has been used. That is, we can write the integral formulation \eqref{int-formtR01} as
\begin{equation}\label{int-formtR03}
\begin{aligned}
\tilde U(x,t) &=  \int_\mathbb{R} \mathcal{G} (x,y,t) \tilde U_0(y)\rmd y  
+  \int_0^t \int_\mathbb{R}  \mathcal{G} (x,y,t-s)  \mathcal{N}_2(\tilde R,\tilde \phi_y, \delta^\pm )(y,s) \rmd y \rmd s
-\sum_\pm \Theta^\pm(t) \mathcal{E}^\pm(x,t)  ,
\end{aligned}\end{equation}
with $$ \mathcal{N}_2(\tilde R,\tilde \phi_y, \delta^\pm )(y,s): =  \mathcal{N}_1(\tilde R,\tilde \phi_y, \delta^\pm ) (y,s) - \sum_\pm\Theta^\pm(s)\mathcal{R}^\pm(y,s),$$
where $\mathcal{N}_1(\tilde R,\tilde \phi_y,\delta^\pm)$ is defined in \eqref{bound-NN} and $\mathcal{R}^\pm(y,s)$ is defined in \eqref{bounds-SR}. 

Let us recall the decomposition \eqref{Green-decompose} of the Green's function:
$$
 \mathcal{G} (x,y,t) = \mathcal{E}^+(x,t)  \langle \Psi^+(y),\cdot \rangle + \mathcal{E}^-(x,t)  \langle \Psi^-(y),\cdot \rangle + \widetilde{ \mathcal{G} }(x,y,t) .
$$ 
Using this decomposition, we write the integral formulation \eqref{int-formtR03} as follows. To avoid repeatedly writing the lengthy integrals, let us denote 
\[
\begin{aligned}
\mathcal{I}^+(t)&: = \int_\mathbb{R} \langle \Psi^+(y),\tilde U_0(y)\rangle  \rmd y   +  \int_0^t \int_\mathbb{R} \langle \Psi^+(y),  \mathcal{N}_2(\tilde R,\tilde \phi_y, \delta^\pm )(y,s)  \rangle   \rmd y \rmd s
\\
\mathcal{I}^- (t)&: = \int_\mathbb{R} \langle \Psi^-(y),\tilde U_0(y) \rangle  \rmd y   +  \int_0^t \int_\mathbb{R} \langle \Psi^-(y),  \mathcal{N}_2(\tilde R,\tilde \phi_y, \delta^\pm )(y,s)  \rangle   \rmd y \rmd s
\\
\mathcal{I}_d(x,t)&:=  \int_\mathbb{R} \widetilde{ \mathcal{G} }(x,y,t) \tilde U_0(y)\rmd y   +  \int_0^t \int_\mathbb{R} \widetilde{ \mathcal{G} } (x,y,t-s)  \mathcal{N}_2(\tilde R,\tilde \phi_y, \delta^\pm )(y,s) \rmd y \rmd s
\\&\quad + \int_0^t \int_\mathbb{R}\Big(\mathcal{E}^+(x,t-s) - \mathcal{E}^+(x,t)\Big)   \langle \Psi^+(y),  \mathcal{N}_2(\tilde R,\tilde \phi_y, \delta^\pm )(y,s)  \rangle   \rmd y \rmd s
\\&\quad + \int_0^t \int_\mathbb{R}  \Big(\mathcal{E}^-(x,t-s) - \mathcal{E}^-(x,t)\Big)  \langle \Psi^-(y),  \mathcal{N}_2(\tilde R,\tilde \phi_y, \delta^\pm )(y,s)  \rangle   \rmd y \rmd s,
\end{aligned}
\]
referring to the contributions accounting for the translation and phase shifts, and the decaying part, respectively. 
Thus, the integral formulation \eqref{int-formtR03} for solutions $(\tilde R,\tilde \phi)$ to \eqref{tUsystem} now becomes
\[
\begin{aligned}
\tilde U(x,t)&=    - \sum_\pm \Theta^\pm(t) \mathcal{E}^\pm(x,t)  + \sum_\pm  \mathcal{I}^\pm (t) \mathcal{E}^\pm(x,t) +  \mathcal{I}_{d}(x,t).\end{aligned}
\]
Neither of the two terms $ \mathcal{I}^\pm (t) \mathcal{E}^\pm(x,t)$ decay in time. To capture these non-decaying terms, we are led to choose $\Theta^\pm$ such that 
$$
\begin{aligned}
\Theta^\pm (t) =  \mathcal{I}^\pm(t),
\end{aligned}
$$
or equivalently, 
\begin{equation}\label{def-intdelta12}
\begin{aligned}
\log (1 +  \delta^\pm(t)) = \log (1 +  \delta^\pm_0) + \frac{2q}{\mathrm{d}}\mathcal{I}^\pm(t),
\end{aligned}
\end{equation}
Note that such a choice is possible because for the following reason. Lemma~\ref{lem-initialdata} implies that the above equation is satisfied at $t = 0$ if $\delta^\pm_0$ are appropriated chosen. We can then just define $\delta^\pm(t)$ to be a solution of the corresponding integral equation. Thus, the representation for solutions $(\tilde R,\tilde \phi)$ simply reads 
\begin{equation}\label{int-formtR05}
\begin{aligned}
 \tilde U(x,t) &=    \mathcal{I}_{d}(x,t).\end{aligned}\end{equation}


\subsection{Spatio-temporal template functions}
In this section, we introduce template functions that are useful for the construction and estimation of the solutions. We let
$$h(t) := h_1(t) + h_2(t)$$ with
$$h_1(t) := \sup_{0 \leq s \leq t}  \Big (|\dot \delta^+(s)| + |\dot \delta^-(s)|\Big )e^{\eta s},
$$ 
$$\begin{aligned}
h_2(t) &:= \sup_{0 \leq s \leq t, \; y\in \mathbb{R}} (1+s)^{-\kappa}\Big(\frac{ |\tilde \phi|}{ \theta}+ \frac{|\tilde \phi _y|+|\tilde \phi_t|+|\tilde \phi _{yy}| +|\tilde R| + |\tilde R_{y}|+|\tilde R_{yy}|}{(1+s)^{-1/2}\theta}\Big) (y,s) 
\end{aligned}$$ for some $\kappa \in (0,1/2)$ and some fixed, small $\eta > 0$. Here, $\theta(x,t)$ denotes the Gaussian-like behavior defined in \eqref{Gaussian-like}. We note that the constant $M_0$ in \eqref{Gaussian-like} is a fixed, large, positive number. At various points in the below estimates, there will be a similar quantity, which we denote by $M$, that will need to be taken to be sufficiently large. The number $M_0$ is then the maximum value of $M$, at the end of the proof. From standard short time theory, we see that $h(t)$ is well defined and continuous for $0<t\ll1$. In addition, standard parabolic theory implies that $h(t)$ retains these properties as long as $h(t)$ stays bounded. The key issue is therefore to show that $h(t)$ stays bounded for all times $t>0$, and this is what the following proposition asserts.

\begin{Proposition} \label{P:claim-zeta} There exists an $\epsilon_0 > 0$ sufficiently small such that the following holds. Given any initial data $A_\mathrm{in}$ with $\epsilon:=\| A_\mathrm{in}(\cdot) - A_\mathrm{source}(\cdot)\|_\mathrm{in}\leq\epsilon_0$ and any $\kappa\in(0,\frac12)$, there exist positive constants $\eta, C_0$, and $M_0$ such that
\begin{equation}\label{keyh-ineq}
\begin{aligned}
h_1(t) \;\le\; C_0 (\e  + h(t)^2), \qquad h_2(t) \;\le\; C_0 (\e  + h_1(t) + h(t)^2) ,
\end{aligned}
\end{equation}
for all $t\geq0$.
\end{Proposition}

Using this proposition, we can add the inequalities in \eqref{keyh-ineq} and eliminate $h_1$ on the right-hand side to obtain
\[
h(t) \leq C_0(C_0+2)(\epsilon + h(t)^2).
\]
Using this inequality and the continuity of $h(t)$, we find that $h(t)\leq 2C_0(C_0+2)\epsilon $ provided $0\leq \epsilon\leq\epsilon_0$ is sufficiently small. Thus, the main theorem will be proved once we establish Proposition~\ref{P:claim-zeta}. The following sections will be devoted to proving this proposition.


\subsection{Bounds on the nonlinear terms}
We recall that the nonlinear remainder $\mathcal{N}_2(\tilde R,\tilde \phi, \delta^\pm )(x,t)$ is defined by 
$$\begin{aligned}
\mathcal{N}_2(\tilde R,\tilde \phi, \delta^\pm )(x,t) &  =\mathcal{N}_1(\tilde R,\tilde \phi, \delta^\pm )(x,t) - \sum_\pm\Theta^\pm(t)\mathcal{R}^\pm(x,t).
\end{aligned}$$
with $\Theta^\pm(s)$ defined in \eqref{def-Theta}, $\mathcal{R}^\pm(x,t)$ defined in \eqref{bounds-SR}, and $\mathcal{N}_1(\tilde R,\tilde \phi, \delta^\pm )$ defined in \eqref{bound-NN}.
We first note that 
\[
|\delta^\pm (t)| \leq |\delta^\pm (0)| + \int_0^t |\dot \delta^\pm (s)| \rmd s
\leq |\delta^\pm_0| + \int_0^t \rme^{-\eta s} h_1(s) \rmd s
\leq C (\epsilon + h_1(t)),
\]
where we recall that Lemma~\ref{lem-initialdata} implies that $\delta_0^\pm \sim \epsilon$.
By \eqref{bound-Rdelta} and the definition of $h_1(t)$
\[
\begin{aligned}
| {\mathcal{R}}^{app}_1 (x,t, \delta^\pm) |& \le C (|\delta^+ | + |\delta^- |)\rme^{-\eta |x|}  \theta(x,t) + C(1+t)^{-1}\theta(x,t) (\epsilon + h_1(t)^2).
\end{aligned}
\]
Here we leave the linear term in $\delta^\pm(t)$ in the above estimate for a different treatment, below. Similarly, by definitions \eqref{def-Theta} and \eqref{bounds-SR}, we get 
$$|\sum_\pm\Theta^\pm(t)\mathcal{R}^\pm(\cdot,t)| \le C\Big((1+t)^{-1}  + \rme^{-\eta |x|} \Big) (|\delta^+ | + |\delta^- |) \theta(x,t)  .
$$
Next, from \eqref{def-LQ} and the definition of $h_2(t)$, we have
$$\begin{aligned}
L(\tilde R,\tilde \phi)  &= \cO(|\tilde R| + |\tilde R_x| +|\tilde  \phi_x|) \le C  (1+t)^{-\frac 12 + \kappa}\theta(x,t) h_2(t)
, \\
Q(\tilde R,\tilde \phi)  &= \cO(\tilde R^2 + \tilde R_x^2 +|\tilde R\tilde R_{xx}|+\tilde  \phi_x^2 + |\tilde R\tilde\phi_{xx}| + |\tilde R \tilde \phi_t|) \le C (1+t)^{-1+2\kappa} \theta(x,t)^2 h_2^2(t).
\end{aligned}$$
Recalling the estimate \eqref{bound-NN} for $\mathcal{N}_1$, we find that
the nonlinear term $\mathcal{N}_2(\tilde R,\tilde \phi, \delta^\pm )(x,t)$ satisfies
\begin{equation}\label{bound-2NN}
| \mathcal{N}_2(\tilde R,\tilde \phi, \delta^\pm )(x,t)| \le  C \Big[e^{-\eta |x|}  + (1+t)^{-1}\Big]\theta(x,t) (\epsilon + |\delta^+(t)| + |\delta^-(t)| )   + C(1+t)^{-1+\kappa} \theta(x,t) h^2(t),
\end{equation}
where we have used $\theta(x,t) \le 2(1+t)^{-\kappa}$ and $(1+t)^{\kappa}e^{-\eta t} \le C (1+t)^{-1}$. 


\subsection{Estimates for $h_1(t)$}

To establish the claimed estimate for $h_1(t)$, we differentiate the expression (\ref{def-intdelta12}) to get 
\begin{equation}\label{eqs-dotq}\dot{\delta }^\pm (t) =  \frac{2q}{\mathrm{d}}(1+\delta^\pm (t) )  \int_\mathbb{R} \langle \Psi^\pm (y),  \mathcal{N}_2(\tilde R,\tilde \phi_y, \delta^\pm )(y,t)  \rangle   \rmd y 
\end{equation} for each $+/-$ case. We first recall that $|\Psi^\pm (y)|\leq 2\rme^{-\eta_0|y|}$ and that
\[
\rme^{-\frac{\eta_0|y|}{2}} \rme^{-\frac{(y\pm c_\mathrm{g}t)^2}{M(1+t)}} \leq C_1 \rme^{-\frac{\eta_0|y|}{4}} \rme^{-\frac{c_\mathrm{g}^2 t}{M}} \le  C_1 \rme^{-\frac{\eta_0|y|}{4}} \rme^{-\eta t} ,
\]
which holds for each $M\geq8c_\mathrm{g}/\eta_0$ and $\eta$ sufficiently small. This, together with the bound \eqref{bound-2NN} on the nonlinear term $\mathcal{N}_2(\tilde R,\tilde \phi_y, \delta^\pm )$, implies 
\begin{equation}
\label{loc-nonbound}
|\langle \Psi^\pm (y),  \mathcal{N}_2(\tilde R,\tilde \phi_y, \delta^\pm )(y,t)  \rangle |   \leq 
 C\rme^{-\frac{\eta_0}{4}|y| } \rme^{-\eta t}  \Big[ \epsilon + |\delta^+(t)| + |\delta^-(t)| + h(t)^2\Big] ,
\end{equation} 
and so 
\begin{equation}\label{est-deltaN} \Big|\int_\mathbb{R} \langle \Psi^\pm (y),  \mathcal{N}_2(\tilde R,\tilde \phi_y, \delta^\pm )(y,t)  \rangle   \rmd y \Big| \le 
 C \rme^{-\eta t} \Big[ \epsilon + |\delta^+(t)| + |\delta^-(t)| + h(t)^2\Big] .\end{equation}
Now, multiplying the equation \eqref{eqs-dotq} by $\delta^\pm$ and using the above estimate on the integral, we get 
$$\begin{aligned}
\frac{d}{dt} \sum_\pm|\delta^\pm(t)|^2 &\le C \rme^{-\eta t} \sum_\pm|\delta^\pm(t)|^2  + C e^{-\eta t} \Big[ \epsilon + h(t)^2\Big] ^2,
\end{aligned}$$
where we have used Young's inequality and the fact that, as long as $\delta^\pm$ is bounded, higher powers of $\delta^\pm$ can be bounded by C$|\delta^\pm|^2$. Applying the standard Gronwall's inequality, we get 
$$ \sum_\pm|\delta^\pm(t)|^2 \le \sum_\pm|\delta^\pm(0)|^2 \rme^{\int_0^t C e^{-\eta s} \; ds} + C \int_0^t \rme^{\int_s^t C\rme^{-\eta \tau}\; d\tau} \rme^{-\eta s}\Big[ \epsilon + h(s)^2\Big]^2\; ds .$$
Since $\int_0^t \rme^{-\eta s}\; ds$ is bounded, $|\delta^\pm(0)|\le C \epsilon$, and $h(t)$ is an increasing function, the above estimate yields 
\begin{equation}\label{fin-bdelta} \sum_\pm|\delta^\pm(t)|^2 \le C\Big[ \epsilon + h(t)^2\Big]^2.\end{equation}
Using this into \eqref{est-deltaN} and in \eqref{eqs-dotq}, we immediately obtain 
$$ |\dot \delta^\pm(t)| \le C \rme^{-\eta t} \Big[ \epsilon + h(t)^2\Big],$$ 
which yields the first inequality in \eqref{keyh-ineq}. 

Finally, not that if we combine the bound \eqref{fin-bdelta} with the nonlinear estimate \eqref{bound-2NN}, we obtain 
\begin{equation}\label{bound-Nnew}
 \mathcal{N}_2(\tilde R,\tilde \phi, \delta^\pm )(x,t) \le  C \Big[e^{-\eta |x|}  + (1+t)^{-1+\kappa}\Big]\theta(x,t) (\epsilon +  h^2(t)). 
\end{equation}


\subsection{Pointwise estimates for $\tilde R$ and $r\tilde \phi$}
Let $\tilde U = (\tilde R, r\tilde \phi)$ satisfy the integral formulation \eqref{int-formtR05}. We shall establish the following pointwise bounds
\begin{equation}\label{est-key-V}\begin{aligned}
|\tilde R (x,t)| & \quad \leq \quad C\Big (\epsilon  + h(t)^2\Big) (1+t)^{-1/2+\kappa}\theta(x,t)
\\|r\tilde \phi (x,t)| & \quad \leq \quad C \Big(\epsilon  + h(t)^2\Big)\theta(x,t),
\end{aligned}\end{equation}
and the derivative bounds
\begin{equation}\label{est-key-derV}\begin{aligned}
\Big|\partial_t^\ell\partial_x^k\tilde U(x,t)\Big| & \quad \leq \quad C\Big (\epsilon  + h(t)^2\Big) (1+t)^{-1/2+\kappa}\theta(x,t)
\end{aligned}\end{equation}
for $k + \ell \le 3$ with $\ell = 0,1$ and $k =1,2,3$. We recall the integral formulation \eqref{int-formtR05}: \begin{equation}\label{int-form-V}\begin{aligned}
\tilde U(x,t)&=  \int_\mathbb{R}\widetilde{ \mathcal{G} }(x,y,t) \tilde U_0(y)\rmd y  
+  \int_0^t \int_\mathbb{R} \widetilde{ \mathcal{G}  }(x,y,t-s)  \mathcal{N}_2(\tilde R,\tilde \phi_y, \delta^\pm )(y,s) \rmd y \rmd s
\\&\quad + \sum_\pm\int_0^t \int_\mathbb{R}\Big(\mathcal{E}^\pm(x,t-s) - \mathcal{E}^\pm(x,t)\Big)   \langle \Psi^\pm(y),  \mathcal{N}_2(\tilde R,\tilde \phi_y, \delta^\pm )(y,s)  \rangle   \rmd y \rmd s
\end{aligned}\end{equation}
We give estimates for each term in this expression. First, we consider the integral term in (\ref{int-form-V}) that involves the initial data. We recall that 
\begin{equation}\label{Green-est}
|\widetilde{ \mathcal{G}  }(x,y,t)| \quad \leq \quad C t^{-1/2} \left( \rme^{-\frac{(x-y+c_\mathrm{g}t)^2}{4t}} + \rme^{-\frac{(x-y-c_\mathrm{g}t)^2}{4t}} \right).
\end{equation}
Using this bound, together with equation \eqref{tU0-decay}, we see that
\begin{eqnarray}\label{est-initial}
\int_\mathbb{R} |\widetilde{ \mathcal{G}  }(x,y,t) \tilde U_0(y)| \rmd y & \leq & C \epsilon \int_{\mathbb{R}} t^{-1/2} \left(\rme^{-\frac{(x-y+c_\mathrm{g}t)^2}{4t}} + \rme^{-\frac{(x-y-c_\mathrm{g}t)^2}{4t}} \right) \rme^{-\frac{y^2}{M}} \rmd y.
\end{eqnarray}
Using the fact that, for $t\geq1$,
\[
\rme^{-\frac{(x-y\pm c_\mathrm{g}t)^2}{4t}} \rme^{-\frac{y^2}{M}} \quad \leq \quad C_1 \rme^{-\frac{(x\pm c_\mathrm{g}t)^2}{Mt}} \rme^{-\frac{y^2}{2M}},
\]
and, for $t\leq1$, 
\begin{eqnarray*}
\rme^{-\frac{(x-y\pm c_\mathrm{g}t)^2}{8t}} \rme^{-\frac{y^2}{M}} \quad \leq \quad 2\rme^{-\frac{(x-y)^2}{8t}}\rme^{-\frac{y^2}{M}} \quad \leq \quad C_1 \rme^{-\frac{x^2}{2M}} 
\end{eqnarray*}
and 
\begin{eqnarray*}
\int_{\mathbb{R}} t^{-1/2} \left(\rme^{-\frac{(x-y+c_\mathrm{g}t)^2}{8t}} + \rme^{-\frac{(x-y-c_\mathrm{g}t)^2}{8t}} \right) \rmd y & \leq & C_1,
\end{eqnarray*}
we conclude that the integral in (\ref{est-initial}) is again bounded by $\epsilon C_1 \theta(x,t)$.

Now, we note that if we project the Green's function on the $R$-component, say $\widetilde{ \mathcal{G} }_R(x,y,t)$, we get a better bound than that of \eqref{Green-est}; see \eqref{prop-Green}. More precisely, we find 
\begin{equation}\label{Green-Rest}
| \widetilde{ \mathcal{G} }_R(x,y,t)| \quad\leq \quad C t^{-1/2} (t^{-1/2} + e^{-\eta |y|}) \left( \rme^{-\frac{(x-y+c_\mathrm{g}t)^2}{4t}} + \rme^{-\frac{(x-y-c_\mathrm{g}t)^2}{4t}} \right).
\end{equation}
Using this better bound on the $R$-component, the above argument shows that 
$$\int_\mathbb{R}\widetilde{ \mathcal{G} }(x,y,t)\tilde R_0(y) \rmd y \quad \quad\leq \quad \quad C\epsilon (1+t)^{-1/2} \theta(x,t).$$
Next, for the second term in \eqref{int-form-V}, we write 
$$\begin{aligned}  \int_0^t \int_\mathbb{R} & \widetilde{ \mathcal{G} }(x,y,t-s)  \mathcal{N}_2(\tilde R,\tilde \phi_y, \delta^\pm )(y,s)  \rmd y \rmd s 
\\&=  \int_0^t \Big[\int_{\{|y|\le 1\}}+\int_{\{|y|\ge 1\}}\Big] \widetilde{ \mathcal{G} }(x,y,t-s) \mathcal{N}_2(\tilde R,\tilde \phi_y, \delta^\pm )(y,s) \rmd y \rmd s  \\
&= I_1 + I_2.
\end{aligned}$$
By the nonlinear estimates in \eqref{bound-Nnew}, we have 
$$\begin{aligned}  |I_1| &\quad\leq \quad C \Big(\epsilon + h(t)^2\Big) \int_0^t \int_{\{|y|\le 1\}} \widetilde{ \mathcal{G} }(x,y,t-s) \rme^{-\frac{c_\mathrm{g}^2}{M}s} \; \rmd y\rmd s 
 \end{aligned}$$
and $$\begin{aligned}  |I_2| &\quad\leq \quad C \Big(\epsilon + h(t)^2\Big) \int_0^t \int_{\mathbb{R}} \widetilde{ \mathcal{G} }(x,y,t-s) \Big[ e^{-\eta |y|} + (1+s)^{-1+\kappa}\Big] \theta(y,s) \; \rmd y\rmd s .
\end{aligned}$$
The following lemma is precisely to give the convolution estimates on the right-hand sides of the above inequalities, and so the desired bound for the second term in \eqref{int-form-V} is obtained. We shall prove the lemma in the Appendix. 

 \begin{Lemma} \label{lem-conv-est01} For some $C$ and $M$ sufficiently large, 
\[
\begin{aligned}  \Big|\int_0^t \int_{\{|y|\le 1\}} \partial_x^k \widetilde{ \mathcal{G} }(x,y,t-s) \rme^{-\frac{c_\mathrm{g}^2}{M}s} \; \rmd y\rmd s \Big| \quad &\leq \quad C(1+t)^{-\frac k2 + \kappa}\theta(x,t) ,
\\
\Big| \int_0^t \int_{\mathbb{R}} \partial_x^k \widetilde{ \mathcal{G} }(x,y,t-s)\Big[ \rme^{-\eta |y|} + (1+s)^{-1+\kappa} \Big] \theta(y,s)  \; \rmd y\rmd s \Big|   \quad &\leq \quad C(1+t)^{-\frac k2 + \kappa}\theta(x,t),
\end{aligned}\] 
for $k=0,1,2,3$. In addition, similar estimates hold for $\Pi_R \widetilde{ \mathcal{G} }(x,y,t)$, with a gain of an extra factor $(1+t)^{-1/2}$.  
\end{Lemma} 

Next, we consider the last integral term 
\begin{equation}\label{est-ediff}
\sum_\pm\int_0^t \int_\mathbb{R}\Big(\mathcal{E}^\pm(x,t-s) - \mathcal{E}^\pm(x,t)\Big)   \langle \Psi^\pm(y),  \mathcal{N}_2(\tilde R,\tilde \phi_y, \delta^\pm )(y,s)  \rangle   \rmd y \rmd s
\end{equation}
in \eqref{int-form-V}. Due to the bounds \eqref{loc-nonbound} and \eqref{fin-bdelta}, we get 
$$|  \langle \Psi^\pm(y),  \mathcal{N}_2(\tilde R,\tilde \phi_y, \delta^\pm )(y,s)  \rangle| \le  C\rme^{-\frac{\eta_0}{2}|y| } \rme^{-\eta s}  \Big[ \epsilon + h(s)^2\Big] .$$
Also, due to \eqref{app-Epm} and the fact that $\mathcal{B}(x,t) =  e(x,t+1) $ and $\mathcal{B}_x(x,t) = \cO(\theta(x,t))$, we have 
\begin{equation}\label{bound-dE}\begin{aligned}
\Big| &\sum_\pm\int_0^t \int_\mathbb{R}\Big(\mathcal{E}^\pm(x,t-s) - \mathcal{E}^\pm(x,t)\Big)   \langle \Psi^\pm(y),  \mathcal{N}_2(\tilde R,\tilde \phi_y, \delta^\pm )(y,s)  \rangle   \rmd y \rmd s \Big|
\\& \le C \Big(\epsilon +h(t)^2\Big) \int_0^t  |e(x,t-s+1) - e(x,t+1)| \rme^{-\eta s}   \rmd s + C \Big(\epsilon+h(t)^2\Big) \theta(x,t).
\end{aligned}\end{equation}
Again note that the $R$-component of $\mathcal{E}^\pm(x,t)$ is bounded by 
$$ C |\mathcal{B}_x(x,t)|  + C \rme^{-\eta|x|} |\mathcal{B}(x,t)|.$$
The estimate \eqref{bound-dE} is thus improved by either $(1+t)^{-1/2}$ or $\rme^{-\eta |x|}$ when projected on the $R$-component. Here recall again that $\rme^{-\eta |x|} \theta(x,t)$ is in fact decaying exponentially in time and space. Thus it suffices to show that the integral on the right hand side of \eqref{bound-dE} is bounded by $C\theta(x,t)$.  We shall prove the following lemma in the appendix. 

\begin{Lemma}\label{lem-est-ediff}
For each sufficiently large $M$, there is a constant $C$ so that
\begin{eqnarray*}
&& \int_0^t \Big|\partial_x^k\Big[e(x,t-s+1) - e(x,t+1)\Big]\Big|\rme^{-\eta s} \rmd s \quad \leq \quad C(1+t)^{-k/2}\theta(x,t), 
\end{eqnarray*}
for $k=0,1,2,3$. \end{Lemma}

In summary, collecting all these estimates into \eqref{int-form-V}, we have obtained the desired estimate \eqref{est-key-V} for $k=0$. The estimates for the derivatives follow exactly the same way as done above with a gain of time decay; we omit the proof. 


\subsection{Estimates for $h_2(t)$}\label{sec-esth2}

In this section we prove the claimed estimate for $h_2(t)$, stated in Proposition~\ref{P:claim-zeta}. The estimates \eqref{est-key-V} almost prove the claimed inequality, except the estimate near the core $x=0$ at which $r(0) = 0$. By Lemma~\ref{lem-NBexistence}, we can assume without loss of generality that there exist positive constants $a,b_{1,2}$ so that 
\[
|r(x) |\ge a,\qquad \forall ~ |x|\ge 1,
\]
and 
\begin{equation}\label{x-ncore}b_1|x|\le  |r(x)|\le b_2|x| ,\qquad |r_x(x)|\ge b_1, \qquad |r_{xx}|\le b_2 |r(x)| ,\forall ~ |x|\le 1.\end{equation}

\bigskip

{\bf Away from the core $|x|\ge 1$.}  In this case, the second estimate in \eqref{est-key-V}, \eqref{est-key-derV}, and the fact that $|r_x| + |r_{xx}| = \cO(\rme^{-\eta |x|})$ imply
$$\begin{aligned}
|\tilde \phi (x,t)|  \quad &\leq \quad Ca^{-1} \Big(\epsilon + h(t)^2\Big)\theta(x,t)
\end{aligned}$$
and 
$$\begin{aligned}
|\tilde \phi_x (x,t)| + |\tilde \phi_{xx} (x,t)|   \quad &\leq \quad Ca^{-1} \Big(\epsilon + h(t)^2\Big)\Big( \rme^{-\eta |x|}+ (1+t)^{-1/2} \Big ) \theta(x,t).
\end{aligned}$$
Note that $\rme^{-\eta |x|} \theta(x,t)$ can be bounded by $C \rme^{-\eta(|x| + t)}$, which may be neglected. In addition, from \eqref{tUsystem} we can write $r\tilde \phi_t$ in terms of $(\tilde R,\tilde \phi)$ and their spatial derivatives. Thus, $r\tilde \phi_t$ is bounded by $C(\epsilon + h(t)^2)(1+t)^{-1/2}\theta(x,t)$, where the extra $(1+t)^{-1/2}$ is due precisely to the fact that the right hand side of \eqref{tUsystem} does not contain $\tilde \phi$ (the term with the slowest decay in the equation). Therefore, the claimed estimate on $\tilde \phi_t$ follows when $|x|\ge 1$. 

\bigskip

{\bf Near the core $|x|\le 1$.} The second estimate in \eqref{est-key-derV} with $k=2$ gives
$$ |(r\tilde \phi)_{xx}| \le C \Big(\epsilon + h(t)^2\Big) \theta(x,t).$$
Thus, if we write 
$$ (r^2 \tilde \phi_x) _x =  r(2  r_x \tilde \phi_x + r \tilde \phi_{xx}) = r ((r\tilde \phi)_{xx} - r_{xx} \tilde \phi) ,$$
by integration together with \eqref{x-ncore} we have 
$$|r^2\tilde \phi_x| =\Big| \int_0^x r(y) (g(y,t) - r_{yy} \tilde \phi) \; dy \Big| \le C \int_0^x |y| \Big(  |(r\tilde \phi)_{yy}| + |r\tilde \phi| \Big ) \; dy \le C x^2 \Big(\epsilon +  h(t)^2\Big) \rme^{-\eta t} .$$
Here we note that $r\tilde \phi(x,t)$ is finite, and so $r^2 \tilde \phi$ vanishes at $x=0$ since $r(0)=0$. Again by the estimate $|r(x)| \ge b_1 |x|$ from \eqref{x-ncore}, the above estimate yields 
\begin{equation}\label{phi-estnear0} |\tilde \phi_x| \le C b_1^{-2} \Big(\epsilon + h(t)^2\Big) \rme^{-\eta t} .\end{equation}
In addition, if we write $r_x \tilde \phi = (r\tilde \phi)_x - r\tilde \phi_x$ and use the above estimate together with \eqref{est-key-V}, we obtain the claimed estimate for $\tilde \phi$ at once thanks to the assumption that $|r_x|\ge b_1>0$. 

Similarly, let us check the claimed estimate for $\tilde \phi_t$. As above, we write 
$$ (r^2 \tilde \phi_{xt})_x =   r \Big ((r\tilde \phi)_{xxt} - \frac{r_{xx}}{r} (r\tilde \phi)_t\Big)$$
and note that, by \eqref{est-key-derV}, $(r\tilde \phi)_{xxt}$ and $(r\tilde \phi)_t$ are already bounded by $C(\epsilon + h(t)^2) \rme^{-\eta t} $. It thus follows similarly to \eqref{phi-estnear0} that 
$$ |\tilde \phi_{xt}| \le C b_1^{-2} \Big(\epsilon + h(t)^2\Big) \rme^{-\eta t} .$$
This yields the desired estimate for $\tilde \phi_t$ near the core by writing $r_x \tilde \phi_t = (r\tilde \phi)_{xt} - r \tilde \phi_{xt}$ and using \eqref{est-key-derV}. 

Finally, we turn to the claimed estimate for $\tilde \phi_{xx}$. First notice that we also have an estimate for $r\tilde \phi _{xx}$ for all $x$ by writing 
\begin{equation}\label{rphixx}r\tilde \phi _{xx} = (r\tilde \phi )_{xx} - 2r_x\tilde \phi _x - r_{xx}\tilde \phi .\end{equation}
Now to estimate $\tilde \phi_{xx}$ for $x$ near zero, we can write 
$$ (r^3 \tilde \phi _{xx})_x = r^2 \Big[ (r\tilde \phi )_{xxx} - 3 r_{xx}\tilde \phi _{x} - r_{xxx} \tilde \phi \Big] $$ and integrate the identity from $0$ to $x$. By \eqref{rphixx}, $r\tilde \phi_{xx} $ is finite at $x=0$ and so $r^3 \tilde \phi_{xx}$ vanishes at $x=0$. Using the estimates \eqref{est-key-derV} with $k=3$ and the estimates on $\tilde \phi_{x}$ and on $\tilde \phi$, we thus obtain 
$$|r^3\tilde \phi_{xx}|  \le C \int_0^x |y|^2 \Big[ (r\tilde \phi )_{yyy} - 3 r_{yy}\tilde \phi _{y} - r_{yyy} \tilde \phi \Big)\; dy \le C |x|^3 \Big(\epsilon  + h(t)^2\Big) \rme^{-\eta t} .$$
Again, since $|r(x)|\ge b_1|x|$, we then obtain 
$$ |\tilde \phi_{xx}|  \le C b_1^{-3} \Big(\epsilon +  h(t)^2\Big) \rme^{-\eta t} ,$$
for $|x|\le 1$. 

This completes the proof of the claimed estimate $h_2(t)$. The key proposition (Proposition~\ref{P:claim-zeta}) is  therefore proved, and so is the main theorem.  


\appendix 

\section{Convolution estimates}
In this section, we prove the convolution estimates that we used in the previous sections. These estimates can also be found in \cite{BNSZ}. 

\begin{proof}[\textbf{Proof of Lemma~\ref{lem-conv-est01}}] Let us recall that 
\[
|\widetilde{ \mathcal{G} }(x,y,t)| \quad\leq \quad C t^{-1/2} \left( \rme^{-\frac{(x-y+c_\mathrm{g}t)^2}{4t}} + \rme^{-\frac{(x-y-c_\mathrm{g}t)^2}{4t}} \right),
\]
and 
\begin{equation}\label{Green-Rest-app}
|\widetilde{ \mathcal{G}  }_R(x,y,t)| \quad\leq \quad C t^{-1/2} (t^{-1/2} + e^{-\eta |y|}) \left( \rme^{-\frac{(x-y+c_\mathrm{g}t)^2}{4t}} + \rme^{-\frac{(x-y-c_\mathrm{g}t)^2}{4t}} \right).
\end{equation}

We will show that \begin{equation}\label{app-Gest01}
\begin{aligned}
& \int_0^t \int_{\mathbb{R}} \widetilde{ \mathcal{G}  }(x,y,t-s) (1+s)^{-1+\kappa} \theta(y,s) \; \rmd y\rmd s   \quad \leq \quad C(1+t)^\kappa\theta(x,t)\\
& \int_0^t \int_{\mathbb{R}}  \widetilde{ \mathcal{G}  }_R(x,y,t-s)  (1+s)^{-1+\kappa} \theta(y,s)    \; \rmd y\rmd s   \quad \leq \quad C(1+t)^{-1/2+\kappa}\theta(x,t).
\end{aligned}\end{equation}

Let us start with a proof of the first estimate in \eqref{app-Gest01}. We first note that there are constants $C_1,\tilde{C}_1>0$ such that
\[
\tilde{C}_1 \rme^{-y^2/M} \quad\leq \quad |\theta(y,s)| \quad\leq \quad C_1 \rme^{-y^2/M}
\]
for all $0\leq s\leq1$. Thus, for some constant $C_1$ that may change from line to line, we have
\begin{eqnarray*}
\lefteqn{\int_0^t \int_\mathbb{R} |\widetilde{ \mathcal{G}  }(x,y,t-s) (1+s)^{-1+\kappa} \theta(y,s)  (y,s) \rmd y \rmd s}
\\ & \leq &
C_1 \int_0^t \int_\mathbb{R} (t-s)^{-1/2} \rme^{-\frac{(x-y)^2}{4(t-s)}} \rme^{-y^2/M} \rmd y\rmd s
\\ & \leq &
C_1 \int_0^t \left[\int_{\{|y|\ge 2|x|\}} (t-s)^{-1/2} \rme^{-\frac{(x-y)^2}{8(t-s)}} \rme^{-\frac{x^2}{8(t-s)}} \rmd y+ \int_{\{|y|\le 2|x|\}} (t-s)^{-1/2} \rme^{-\frac{(x-y)^2}{4(t-s)}} \rme^{-\frac{4x^2}{M}} \rmd y\right]\rmd s
\\ & \leq &
C_1 \int_0^t \Big[ \rme^{-\frac{x^2}{8(t-s)}}+  \rme^{-\frac{4x^2}{M}} \Big]\rmd s
\\ & \leq &
C_1  \rme^{-\frac{4x^2}{M}} \quad \leq \quad \frac{C_1}{\tilde{C}_1} \theta(x,t)
\end{eqnarray*}
for all $0\leq t\leq1$. 
Next, we write the first estimate in \eqref{app-Gest01} as
\[
\theta(x,t)^{-1} \int_0^t \int_\mathbb{R} |\widetilde{ \mathcal{G}  }(x,y,t-s)| (1+s)^{-1+\kappa} \theta(y,s)  \rmd y \rmd s
\]
for $t\geq1$. Combining only the exponentials in this expression, we obtain terms that can be bounded by
\begin{equation}\label{exp-form}
\exp\left(\frac{(x+\alpha_3 t)^2}{M(1+t)} - \frac{(x-y+\alpha_1 (t-s))^2}{4(t-s)} - \frac{(y+\alpha_2 s)^2}{M(1+s)} \right)
\end{equation}
with $\alpha_j=\pm c_\mathrm{g}$. To estimate this  expression, we proceed as in \cite[Proof of Lemma~7]{HowardZumbrun06} and complete the square of the last two exponents in (\ref{exp-form}). Written in a slightly more general form, we obtain
\begin{eqnarray*}
\lefteqn{\frac{(x-y-\alpha_1(t-s))^2}{M_1(t-s)} + \frac{(y-\alpha_2 s)^2}{M_2(1+s)} \;=\;
\frac{(x-\alpha_1(t-s)-\alpha_2 s)^2}{M_1(t-s)+M_2(1+s)} } \\ &&
+ \frac{M_1(t-s)+M_2(1+s)}{M_1M_2(1+s)(t-s)}\left( y - \frac{xM_2(1+s) - (\alpha_1M_2(1+s) + \alpha_2M_1s)(t-s)}{M_1(t-s)+M_2(1+s)}\right)^2
\end{eqnarray*}
and conclude that the exponent in (\ref{exp-form}) is of the form
\begin{eqnarray}\label{exp-form-2}
\lefteqn{ \frac{(x+\alpha_3t)^2}{M(1+t)} - \frac{(x-\alpha_1(t-s)-\alpha_2s)^2}{4(t-s)+M(1+s)} } \\ \nonumber &&
- \frac{4(t-s)+M(1+s)}{4M(1+s)(t-s)}\left( y - \frac{xM(1+s) - (\alpha_1M(1+s) + 4\alpha_2s)(t-s)}{4(t-s)+M(1+s)}\right)^2,
\end{eqnarray}
with $\alpha_j=\pm c_\mathrm{g}$. Using that the maximum of the quadratic polynomial $\alpha x^2+\beta x+\gamma$ is $-\beta^2/(4\alpha)+\gamma$, it is easy to see that the sum of the first two terms in (\ref{exp-form-2}), which involve only $x$ and not $y$, is less than or equal to zero. Omitting this term, we therefore obtain the estimate
\begin{eqnarray}\label{est-on-exp}
\lefteqn{ \exp\left( \frac{(x\pm c_\mathrm{g}t)^2}{M(1+t)} - \frac{(x-y\delta_1c_\mathrm{g}(t-s))^2}{4(t-s)} - \frac{(y-\delta_2c_\mathrm{g}s)^2}{M(1+s)} \right) } \\ \nonumber & \leq &
\exp \left( - \frac{4(t-s)+Ms}{4M(1+s)(t-s)}\left( y - \frac{xM(1+s)+c_\mathrm{g}(\delta_1M(1+s)+ 4\delta_2s)(t-s)}{4(t-s)+M(1+s)}\right)^2 \right)
\end{eqnarray}
for $\delta_j=\pm1$. Using this result, we can now estimate the integral (\ref{app-Gest01}). Indeed, we have
\begin{eqnarray*}
\lefteqn{ \theta(x,t)^{-1} \int_0^t \int_\mathbb{R}   |\widetilde{ G }(x,y,t-s)|(1+s)^{-1+\kappa}\theta(y,s) \rmd y \rmd s } \\ & \leq &
C_1 (1+t)^{1/2}\int_0^t \frac{1}{\sqrt{t-s}(1+s)^{3/2-\kappa}} \\ && \times 
\int_\R \exp\left(- \frac{4(t-s)+M(1+s)}{4M(1+s)(t-s)}\left( y - \frac{[xM(1+s) \pm c_\mathrm{g}(M(1+s) + 4s)(t-s)]}{4(t-s)+M(1+s)}\right)^2 \right) \rmd y\rmd s \\ & \leq &
C_1(1+t)^{1/2}\int_0^t \frac{1}{\sqrt{t-s}(1+s)^{3/2 - \kappa}}  \sqrt{\frac{4M(1+s)(t-s)}{4(t-s)+M(1+s)}} \rmd s \\ & \leq &
C_1(1+t)^{1/2}\int_0^{t/2} \frac{1}{(1+s)^{1 - \kappa}}\frac{1}{(1+t)^{1/2}} \rmd s + C_1(1+t)^{1/2-\kappa}\int_{t/2}^t \frac{1}{(1+t)^{3/2-\kappa}}\rmd s \\ & \leq &
C_1(1+t)^{-\kappa} + C_1,
 \end{eqnarray*}
which is bounded since $\kappa>0$. 
This proves the first estimate in \eqref{app-Gest01}. The second estimate is entirely the same, using the refined estimate \eqref{Green-Rest-app} for $ \widetilde{ \mathcal{G}  }_R$. Also, derivative estimates follow very similarly. We omit these further details.

Finally, it remains to show that 
\begin{equation}\label{app-est-co}
\begin{aligned}&   \int_0^t \int_{\{|y|\le 1\}}  \widetilde{ \mathcal{G}  }(x,y,t-s) \rme^{-\frac{2c_\mathrm{g}^2}{M}s} \; \rmd y\rmd s  \quad \leq \quad C\theta(x,t),
\end{aligned}\end{equation}
where the Green function bounds read 
$$|\widetilde{ \mathcal{G}  }(x,y,t)| \quad \le \quad  C t^{-1/2} \left( \rme^{-\frac{(x+c_\mathrm{g}t)^2}{4t}} + \rme^{-\frac{(x-c_\mathrm{g}t)^2}{4t}} \right) ,$$
for $|y|\le 1$. The estimate \eqref{app-est-co} is clear when $0\le t\le 1$. Let us consider the case $t\ge 1$. The proof of this estimate uses the following bound: $$\rme^{\frac{(x-c_\mathrm{g}t)^2}{M(1+t)}}\rme^{-\frac{(x-c_\mathrm{g}\tau)^2}{B\tau}} \quad\leq \quad C \rme^{\frac{c_\mathrm{g}^2(t-\tau)}{M}},$$
for fixed constant $B$ and for large $M$. This is a simpler version of \eqref{est-on-exp}; see also \eqref{est-2exp}, below. We thus have 
\begin{eqnarray*}
\lefteqn{\theta(x,t)^{-1}\int_0^t (t-s)^{-1/2}\rme^{-\frac{(x+c_\mathrm{g}(t-s))^2}{4(t-s)}}  \rme^{-\frac{2c_\mathrm{g}^2}{M}s} \rmd s} \\ & \leq &
C(1+t)^{1/2} \int_0^t  (t-s)^{-1/2}\rme^{\frac{c_\mathrm{g}^2s}{M}} \rme^{-\frac{2c_\mathrm{g}^2}{M}s}\rmd s \\ & \leq &
C(1+t)^{1/2} \Big[t^{-1/2}\int_0^{t/2}  \rme^{-\eta s}\rmd s +  \rme^{-\frac{c_\mathrm{g}^2}{2M}t} \int_{t/2}^t  (t-s)^{-1/2}\rmd s\Big] \\ & \leq &
C (1+t)^{1/2} t^{-1/2} + C(1+t)\rme^{-\frac{c_\mathrm{g}^2}{2M}t} ,
\end{eqnarray*}
which is bounded for $t\ge 1$. This proves the estimate \eqref{app-est-co}, and completes the proof of Lemma 
\ref{lem-conv-est01}. \end{proof}

\begin{proof}[\textbf{Proof of Lemma~\ref{lem-est-ediff}}]
We need to show that
\begin{eqnarray*}
&& \int_0^t \Big|\Big[e(x,t-s+1) - e(x,t+1)\Big]\Big|\rme^{-\frac{2c_\mathrm{g}^2}{M}s} \rmd s \quad \leq \quad C_1\theta(x,t), 
\end{eqnarray*}
Intuitively, this integral should be small for the following reason. The difference $e(x,t-s)-e(x,t+1)$ converges to zero as long as $s$ is not too large, say on the interval $s\in[0,t/2]$. For $s\in[t/2,t]$, on the other hand, we use the exponential decay in $s$. Indeed, we have
\begin{eqnarray*}
\lefteqn{|e(x,t-s+1) - e(x,t+1)|} \\ & = & |\int^{t+1}_{t-s+1} e_\tau(x,\tau)\rmd \tau|
\\ & \leq &
\int_{t-s+1}^{t+1} \left| \frac{c}{\sqrt{4\pi \tau}} \left( \rme^{-\frac{(x-c_\mathrm{g}\tau)^2}{4\tau}} + \rme^{-\frac{(x+c_\mathrm{g}\tau)^2}{4\tau}}\right) + \frac{1}{\tau\sqrt{4\pi}} \left( \frac{(x-c_\mathrm{g}\tau)}{\sqrt{4\tau}}\rme^{-\frac{(x-c_\mathrm{g}\tau)^2}{4\tau}} - \frac{(x+c_\mathrm{g}\tau)}{\sqrt{4\tau}}\rme^{-\frac{(x+c_\mathrm{g}\tau)^2}{4\tau}}\right)   \right|\rmd \tau
\\ & \leq &
C\int_{t-s+1}^{t+1} \left( \frac{1}{\sqrt{\tau}} +  \frac{1}{\tau}  \right)  \left( \rme^{-\frac{(x-c_\mathrm{g}\tau)^2}{8\tau}} + \rme^{-\frac{(x+c_\mathrm{g}\tau)^2}{8\tau}}\right) \rmd \tau,
\end{eqnarray*}
where the last estimate follows by the fact that $z\rme^{-z^2} $ is bounded for all $z$.

We shall give estimate for $\theta^{-1}(x,t)(e(x,t-s+1) - e(x,t+1))$.  For instance, let us consider the single exponential term
\[
\rme^{\frac{(x-c_\mathrm{g}t)^2}{M(1+t)}}\rme^{-\frac{(x-c_\mathrm{g}\tau)^2}{B\tau}}.
\]
By combining these and completing the square in $x$, the terms in the exponential become
\[
-\frac{[M(t-\tau+1) + (M-B)\tau]}{MB(t+1)\tau}\left[ x + \frac{c_\mathrm{g}(B-M)\tau(t+1)}{M(t-\tau) + (M-B)\tau}\right]^2 + \frac{c_\mathrm{g}^2(t-\tau+1)^2}{M(t-\tau+1) + (M-B)\tau}.
\]
Since $\tau \leq t$, if $B$ is some fixed constant and 
$M$ is sufficiently large we can neglect the exponential in $x$. That is, we have
\begin{equation}\label{est-2exp}\rme^{\frac{(x-c_\mathrm{g}t)^2}{M(1+t)}}\rme^{-\frac{(x-c_\mathrm{g}\tau)^2}{B\tau}} \quad\leq \quad C \rme^{\frac{c_\mathrm{g}^2(t-\tau)}{M}}.\end{equation}
We therefore obtain 
\begin{eqnarray*}
\theta^{-1}(x,t) |e(x,t-s+1) - e(x,t+1)| & \leq &
C (1+t)^{1/2}\int_{t-s+1}^{t+1} \left( \frac{1}{\sqrt{\tau}} +  \frac{1}{\tau}  \right) \rme^{\frac{c_\mathrm{g}^2(t-\tau)}{M}} \rmd \tau \\ \nonumber & \leq & 
C (1+t)^{1/2} (1+t-s)^{-1/2}\rme^{\frac{c_\mathrm{g}^2s}{M}},
\end{eqnarray*}
Using this and taking $M$ large and $\eta = c_\mathrm{g}^2/M$, we obtain 
\begin{eqnarray*}
\lefteqn{\theta(x,t)^{-1}\int_0^t [e(x,t-s+1) - e(x,t+1)] \rme^{-\frac{2c_\mathrm{g}^2}{M}s} \rmd s} \\ & \leq &
C(1+t)^{1/2} \int_0^t  (1+t-s)^{-1/2}\rme^{\frac{c_\mathrm{g}^2s}{M}} \rme^{-2\eta s}\rmd s \\ & \leq &
C(1+t)^{1/2} \Big[(1+t)^{-1/2}\int_0^{t/2}  \rme^{-\eta s}\rmd s + \rme^{-\eta t/2}\int_{t/2}^t  (1+t-s)^{-1/2}\rmd s\Big] \\ & \leq &
C.
\end{eqnarray*}
This proves the lemma for the case $k=0$. The derivative estimates follow easily from the above proof. \end{proof}

\bibliography{bib-cgl-sources}

\end{document}